\documentclass[review,onefignum,onetabnum]{siamart171218}

\usepackage{a4wide}
\usepackage{amsfonts}
\usepackage{amscd}
\usepackage{amssymb}
\usepackage{amsmath}
\usepackage{latexsym}
\usepackage{dsfont}
\usepackage{subfigure}
\usepackage{placeins}
\usepackage{mathrsfs}

\usepackage{floatflt}
\usepackage{nicefrac}
\usepackage{MnSymbol}
\usepackage[normalem]{ulem}
\usepackage{pgfplots}
\pgfplotsset{compat=newest}
\usepackage{booktabs}
\usepackage{epstopdf}

\usepackage{mathtools}
\usepackage{esvect} 
\usepackage{blindtext}
\usepackage{graphicx}
\usepackage[font=small,labelfont=bf]{caption}

\theoremstyle{plain}
\newtheorem{rem}[theorem]{Remark}
\newtheorem{assumption}[theorem]{Assumption}

{%
\setcounter{enumi}{0}
\renewcommand{\theenumi}{(\textbf{A}.\arabic{enumi})}

\begin{enumerate}}%
{\end{enumerate} }
{%
\setcounter{enumi}{0}
\renewcommand{\theenumi}{(\textbf{D}.\arabic{enumi})}

\begin{enumerate}}%
{\end{enumerate} }

\newcommand{\hione}{{(1)}}
\newcommand{\hitwo}{{(2)}}
\newcommand{\apprlevy}{{(\varepsilon_l)}}
\newcommand{\apprgrf}{{(\varepsilon_W)}}

\makeatletter
\usepackage{color}

\makeatother

\newcommand{\be}{\begin {equation}}
\newcommand{\ee}{\end  {equation}}
\numberwithin{equation}{section} \allowdisplaybreaks[1]

\definecolor{darkgreen}{rgb}{0,.6,0}

\newcommand{\bee}{\begin {equation*}}
\newcommand{\eee}{\end {equation*}}

\nolinenumbers

\title{Multilevel Monte Carlo estimators for elliptic PDEs with Lévy-type diffusion coefficient}

\author{Andrea Barth \thanks{IANS\textbackslash SimTech, University of Stuttgart 
(\email{andrea.barth@mathematik.uni-stuttgart.de}).},
\and Robin Merkle \thanks{IANS\textbackslash SimTech, University of Stuttgart 
	(\email{robin.merkle@mathematik.uni-stuttgart.de}).}}

\headers{MLMC for elliptic PDEs with Lévy-type diffusion coefficient}{A. Barth and R. Merkle}

\begin{document}
		
	\maketitle
		
	\begin{abstract}
General elliptic equations with spatially discontinuous diffusion coefficients may be used as a simplified model for subsurface flow in heterogeneous or fractured porous media. In such a model, data sparsity and measurement errors are often taken into account by a randomization of the diffusion coefficient of the elliptic equation which reveals the necessity of the construction of flexible, spatially discontinuous random fields. Subordinated Gaussian random fields are random functions on higher dimensional parameter domains with discontinuous sample paths and great distributional flexibility. In the present work, we consider a random elliptic partial differential equation (PDE) where the discontinuous subordinated Gaussian random fields occur in the diffusion coefficient. Problem specific multilevel Monte Carlo (MLMC) Finite Element methods are constructed to approximate the mean of the solution to the random elliptic PDE. We prove a-priori convergence of a standard MLMC estimator and a modified MLMC - Control Variate estimator and validate our results in various numerical examples. 
	\end{abstract}
	
	

	
\section{Introduction}\label{sec:Intro}
Partial differential equations with random operators\textbackslash~data\textbackslash~domain are widely studied. For problems with sparse data or where measurement errors are unavoidable, uncertainties may be quantified using stochastic models. Methods to quantify uncertainty could be divided into two different branches: intrusive and non-intrusive. The former requires solving a high dimensional partial differential equation, where part of the dimensionality stems from the smoothness of the random field or process (see among others~\cite{GalerkiNFEApprOfStochEllPDEs}, \cite{FEsForEllProblemsWithStochCoeff}, \cite{OnTheConvOfTheStochGalerkinMethodForRandomEllPDEs} and the references therein). The latter are (essentially) sampling methods and require repeated solutions of lower dimensional problems (see, among others,~
\cite{MLMCForStochEllMultiscalePDEs},
\cite{MLMCFEMEllPDE},
\cite{AStudyOfElliptic}, \cite{AMultilevelMonteCarloAlgorithmForParabolicAdvectionDiffusionProblemsWithDiscontinuousCoefficients},  \cite{MultiLevelMonteCarloWeakGalerkinMethodForEllipticEquationsWithStochasticJumpCoefficients}, \cite{FurtherAnalysisOfMultilevelMonteCarloMethodsForEllipticPDEsWithRandomCoefficients}). Among the sampling method the multilevel Monte Carlo approach has been successfully established to lower the computational complexity for various uncertain problems, to the point where (depending on the dimension) it is asymptotically as costly as a single solve of the deterministic partial differential equation on a fine discretization level (see \cite{MLMCFEMEllPDE}, \cite{AStudyOfElliptic}, \cite{FiniteElementErrorAnalysisOfEllipticPDEsWIthRandomCoefficients} and \cite{GilesMLMCMethods}).
In the cited papers mostly Gaussian random fields were used as diffusivity coefficients in the elliptic equation. Gaussian random fields are stochastically very well understood objects and they may be used in both approaches. The distributions underlying the field are, however, Gaussian and therefore the model lacks flexibility, in the sense that fields cannot have pointwise marginal distributions having heavy-tails. Furthermore, Gaussian random fields with Matérn-type covariance operators have $\mathbb{P}$-almost surely spatial continuous paths. There are some extensions in the literature (see, for example, \cite{AStudyOfElliptic}, \cite{MultiLevelMonteCarloWeakGalerkinMethodForEllipticEquationsWithStochasticJumpCoefficients} and \cite{Gottschalk2021}). 

In this paper we investigate multilevel Monte Carlo methods for an elliptic PDE where the coefficient is given by a subordinated Gaussian random field. The subordinated Gaussian random field is a type of a (discontinuous) L\'evy field. Different subordinators display unique patterns in the discontinuities and have varied marginal distributions (see \cite{SGRF}). Existence and uniqueness of pathwise solutions to the problem was demonstrated in \cite{SGRFPDE}. Spatial regularity of the solution depends on the subordinated Gaussian random field which itself depends on the subordinator. The discontinuities in the spatial domain pose additional difficulty in the pathwise discretization. A sample-adapted approach was considered in \cite{SGRFPDE}, but is limited to certain subordinators. Here we investigate not only the limitations of a sample-adapted approach in multilevel sampling, but also a Control Variates ansatz as presented first in \cite{NobileTeseiMLCV}.  

We structured the rest of the paper as follows: In Section~\ref{sec:elliptic_problem} we introduce a general stochastic elliptic equation and its weak solution under mild assumptions on the coefficient. These assumptions accommodate the subordinated Gaussian random fields we introduce in Section~\ref{sec:subordinated_GRF}. In Section~\ref{sec:subGRFelliptic} we approximate the diffusion coefficient and state a convergence result of the elliptic equation with the approximated coefficient to the unapproximated solution. In Section~\ref{sec:approx_solution} we discuss spatial approximation methods, which are needed for the multilevel Monte Carlo methods introduced in Section~\ref{sec:MLMC} and its Control Variates variant in Section~\ref{sec:MLMCCV}. Numerical examples are presented in the last section.


\section{The stochastic elliptic problem}\label{sec:elliptic_problem}
In this section, we briefly introduce the general stochastic elliptic boundary value problem. For more details on the existence, uniqueness and measurability of the solution to the considered PDE, we refer the reader to \cite{SGRFPDE} and \cite{AStudyOfElliptic}.\\
For the rest of this paper we assume that a complete probability space $(\Omega,\mathcal{F},\mathbb{P})$ is given. Let $(H,( \cdot,\cdot)_H)$ be a Hilbert space. A $H$-valued random variable is a measurable function $Z:\Omega\rightarrow H$. The space $L^p(\Omega;H)$ contains all strongly measurable functions $Z:\Omega\rightarrow H$ with $\|Z\|_{L^p(\Omega;H)}<\infty$, for $p\in [1,+\infty]$, where the norm is defined by
	\begin{align*}
	\|Z\|_{L^p(\Omega;H)} := \begin{cases}\mathbb{E}(\|Z\|_H^p)^\frac{1}{p}&,\text{ if } 1\leq p<+\infty, \\ \underset{\omega\in\Omega}{ess\,sup} \|Z\|_H &,\text{ if } p=+\infty. \end{cases}
	\end{align*}

\noindent For a $H$-valued random variable $Z\in L^1(\Omega;H)$ we define the \textit{expectation} by the Bochner integral $\mathbb{E}(Z):=\int_\Omega Z\,d\mathbb{P}$. Further, for a square-integrable, $H$-valued random variable $Z\in L^2(\Omega;H)$, the \textit{variance} is defined by $Var(Z):=\|Z-\mathbb{E}(Z)\|_{L^2(\Omega;H)}^2$. We refer to \cite{StochasticEquationsInInfiniteDimensions}, \cite{KallenbergFoundationsOfModernProb}, \cite{WTheorie} or \cite{PeszatZabzykSPDEWithLevyNoise} for more details on general probability theory and Hilbert space-valued random variables.

	\subsection{Problem formulation}
	Let $\mathcal{D}\subset \mathbb{R}^d$, for $d\in\mathbb{N}$, be a bounded, connected Lipschitz domain. We consider the elliptic PDE
	\begin{align}\label{EQ:EllProblem}
	-\nabla\cdot(a(\omega,\underline{x})\nabla u(\omega,\underline{x}))=f(\omega,\underline{x}) \text{ in }\Omega\times\mathcal{D},
	\end{align}
	where we impose the following boundary conditions
	\begin{align}
	u(\omega,\underline{x})&=0 \text{ on } \Omega\times \Gamma_1,\label{EQ:EllProblemBCD}\\
	a(\omega,\underline{x}) \overrightarrow{n}\cdot\nabla u(\omega,\underline{x})&=g(\omega,\underline{x}) \text{ on } \Omega\times \Gamma_2.\label{EQ:EllProblemBCN}
	\end{align}
	Here, we split the domain boundary in two $(d-1)$-dimensional manifolds $\Gamma_1,~\Gamma_2$, i.e. $\partial\mathcal{D}=\Gamma_1\overset{.}{\cup}\Gamma_2$, where we assume that $\Gamma_1 $ is of positive measure and that the exterior normal derivative $\overrightarrow{n}\cdot\nabla v$ on $\Gamma_2$ is well-defined for every $v\in C^1(\overline{\mathcal{D}})$. The mapping $a:\Omega\times\mathcal{D}\rightarrow\mathbb{R}$ is a stochastic (jump diffusion) coefficient and $f:\Omega\times\mathcal{D}\rightarrow\mathbb{R}$ is a (measurable) random source function. Further,   $\overrightarrow{n}$ is the outward unit normal vector to $\Gamma_2$ and $g:\Omega\times\Gamma_2\rightarrow\mathbb{R}$ a measurable function.
	Note that we just reduce the theoretical analysis to the case of homogeneous Dirichlet boundary conditions on $\Gamma_1$ to simplify notation. One could also consider non-homogeneous Dirichlet boundary conditions, since such a problem can always be considered as a version of \eqref{EQ:EllProblem} - \eqref{EQ:EllProblemBCN} with modified source term and Neumann data (see also \cite[Remark 2.1]{AStudyOfElliptic}).
	
	The following general assumptions ensure the well-posedness of the elliptic boundary value problem (see also \cite[Assumption 2.2]{SGRFPDE} and \cite[Assumption 2.3]{AStudyOfElliptic}).
	\begin{assumption}\label{ASS:ProblemAssumptionsGeneral}
	Let $H:=L^2(\mathcal{D})$. We assume that 
	\begin{enumerate}
	\item for any fixed $x\in\mathcal{D}$ the mapping $\omega\mapsto a(\omega,x)$ is measurable, i.e. $a(\cdot,x)$ is a (real-valued) random variable,
	\item for any fixed $\omega\in\Omega$ the mapping $a(\omega,\cdot)$ is $\mathcal{B}(\mathcal{D})-\mathcal{B}(\mathbb{R}_+)$-measurable and it holds $a_{-}(\omega):=\underset{\underline{x}\in\mathcal{D}}{ess\,inf}\,a(\omega,\underline{x})>0$ and $a_+(\omega):=\underset{\underline{x}\in\mathcal{D}}{ess\,sup}\,a(\omega,\underline{x})<+\infty$,
	\item $\frac{1}{a_{-}}\in L^p(\Omega;\mathbb{R})$, $f\in L^q(\Omega;H)$ and $g\in L^q(\Omega;L^2(\Gamma_2))$ for some $p,q\in [1,+\infty]$ such that $r:=(\frac{1}{p} + \frac{1}{q})^{-1}\geq 1$. 
	\end{enumerate}
	\end{assumption}

\subsection{Weak solution} In this subsection, we introduce the pathwise weak solution of problem \eqref{EQ:EllProblem} - \eqref{EQ:EllProblemBCN} following \cite{SGRFPDE}.
We denote by $H^1(\mathcal{D})$ the Sobolev space on $\mathcal{D}$ equipped with the norm 
\begin{align*}
\|v\|_{H^1(\mathcal{D})}=\left(\int_D|v(\underline{x})|^2 + \|\nabla v(\underline{x})\|_2^2d\underline{x}\right)^\frac{1}{2} \text{ for } v\in H^1(\mathcal{D}),
\end{align*}
with the Euclidean norm $\|\underline{x}\|_2:=(\sum_{i=1}^d \underline{x}_i^2)^\frac{1}{2}$, for $\underline{x}\in \mathbb{R}^d$
(see for example \cite[Section 5.2]{PartialDifferentialEquations} for an introduction to Sobolev spaces). We denote by $T$ the trace operator  $T:H^1(\mathcal{D})\rightarrow H^{\frac{1}{2}}(\partial \mathcal{D})$ 
where $Tv=v|_{\partial \mathcal{D}}$ for $v\in C^\infty (\overline{\mathcal{D}})$ (see \cite{TraceTheoremLipDomain}) and we introduce the solution space $V\subset H^1(\mathcal{D})$ by
\begin{align*}
V:=\{v\in H^1(\mathcal{D})~|~Tv|_{\Gamma_1}=0\},
\end{align*}
where we take over the standard Sobolev norm, i.e. $\|\cdot\|_V:=\|\cdot\|_{H^1(\mathcal{D})}$. 	We identify $H$ with its dual space $H'$ and work on the Gelfand triplet $V\subset H\simeq H'\subset V'$. Hence, Assumption 
\ref{ASS:ProblemAssumptionsGeneral} guarantees that $f(\omega,\cdot)\in V'$ and $g(\omega,\cdot)\in H^{-\frac{1}{2}}(\Gamma_2)$ for $\mathbb{P}$-almost every $\omega\in \Omega$. 
We multiply the left hand side of Equation \eqref{EQ:EllProblem} by a test function $v\in V$ and integrate by parts (see e.g. \cite[Section 6.3]{ValliACompactCourseOnLinPDEs}) to obtain
\begin{align*}
\int_\mathcal{D}-\nabla\cdot(a(\omega,\underline{x})\nabla u(\omega,\underline{x}))v(\underline{x}) d\underline{x} &=\int_\mathcal{D}a(\omega,\underline{x})\nabla u (\omega,\underline{x})\cdot\nabla v(\underline{x})d\underline{x} - \int_{\Gamma_2}g(\omega,\underline{x})[Tv](\underline{x})d\underline{x}.
\end{align*}
This leads to the following pathwise weak formulation of the problem:\\ For $\mathbb{P}$-almost all $\omega\in\Omega$, given $f(\omega,\cdot)\in V'$ and $g(\omega,\cdot)\in H^{-\frac{1}{2}}(\Gamma_2)$, find $u(\omega,\cdot)\in V$ such that 
\begin{align}\label{EQ:WeakFormProblem}
B_{a(\omega)}(u(\omega,\cdot),v) = F_\omega(v)
\end{align}
for all $v\in V$. The function $u(\omega,\cdot)$ is then called pathwise weak solution to problem \eqref{EQ:EllProblem} - \eqref{EQ:EllProblemBCN}. The bilinear form $B_{a(\omega)}$ and the operator $F_\omega$ are defined by
\begin{align*}
B_{a(\omega)}:V\times V\rightarrow \mathbb{R}, ~(u,v)\mapsto \int_{\mathcal{D}}a(\omega,\underline{x})\nabla u(\underline{x})\cdot \nabla v(\underline{x})d\underline{x},
\end{align*}
and
\begin{align*}
F_\omega:V\rightarrow\mathbb{R}, ~v\mapsto \int_\mathcal{D}f(\omega,x)v(\underline{x})d\underline{x} + \int_{\Gamma_2}g(\omega,\underline{x})[Tv](\underline{x})d\underline{x},
\end{align*}	
for fixed $\omega\in \Omega$, where the integrals in $F_\omega$ are understood as the duality pairings:
\begin{align*}
\int_\mathcal{D}f(\omega,\underline{x})v(\underline{x})d\underline{x} = \prescript{}{V'}{\langle}f(\omega,\cdot),v\rangle_V \text{ and }
\int_{\Gamma_2}g(\omega,\underline{x})[Tv](\underline{x})d\underline{x} = \prescript{}{H^{-\frac{1}{2}}(\Gamma_2)}{\langle	}g(\omega,\cdot),Tv\rangle_{H^\frac{1}{2}(\Gamma_2)},
\end{align*}
for $v\in V$. We have the following theorem on the existence of a unique solution to the random elliptic PDE \eqref{EQ:EllProblem} - \eqref{EQ:EllProblemBCN}.

\begin{theorem}\label{TH:ExistenceTheoremElliptic}(see \cite[Theorem 2.5]{AStudyOfElliptic})
Under Assumption \ref{ASS:ProblemAssumptionsGeneral}, there exists a unique pathwise weak solution $u(\omega,\cdot)\in V$ to problem \eqref{EQ:WeakFormProblem} for $\mathbb{P}$-almost every $\omega\in \Omega$. Furthermore, $u\in L^r(\Omega;V)$ and 
\begin{align*}
\|u\|_{L^r(\Omega;V)}\leq C(a_-,\mathcal{D}, p)(\|f\|_{L^q(\Omega;H)} + \|g\|_{L^q(\Omega;L^2(\Gamma_2))}),
\end{align*}
where $C(a_-,\mathcal{D},p)>0$ is a constant depending only on the indicated parameters.
\end{theorem}	
In addition to the (pathwise) existence of the weak solution, the authors gave a rigorous justification of the measurability of the solution mapping 
\begin{align*}
u:\Omega&\rightarrow V\\
\omega &\mapsto u(\omega,\cdot),
\end{align*}
in \cite[Remark 2.5]{SGRFPDE}.


\section{Subordinated Gaussian random fields}\label{sec:subordinated_GRF}
In \cite{SGRF}, the authors proposed a new subordination approach to construct discontinuous Lévy-type random fields: the \textit{subordinated Gaussian random field}. Motivated by the subordinated Brownian motion, the subordinated Gaussian Random field is constructed by replacing the spatial variables of a  Gaussian random field (GRF) $W$ on a general $d$-dimensional domain $\mathcal{D}\subset	\mathbb{R}^d$ by $d$ independent Lévy subordinators (see \cite{SGRF}, \cite{LevyProcessesInFinance}, \cite{LevyProcessesAndStochasticCalculus}). For $d=2$, the detailed construction is as follows:
For two positive horizons $T_1,T_2<+\infty$, we define the domain $\mathcal{D}=[0,T_1]\times[0,T_2]$. We consider a GRF $W = (W(x,y),~(x,y)\in \mathbb{R}_+^2)$ with $\mathbb{P}$-a.s. continuous paths and assume two independent Lévy subordinators $l_1=(l_1(x),~x\in[0,T_1])$ and $l_2=(l_2(y),~y\in[0,T_2])$ are given (see \cite{SGRF} and \cite{LevyProcessesAndStochasticCalculus}). The subordinated GRF is then defined by
\begin{align}\label{EQ:SGRFDefi}
L(x,y)=W(l_1(x),l_2(y)) \text{, for } (x,y)\in [0,T_1]\times [0,T_2].
\end{align}
The corresponding random field $L=(L(x,y),~(x,y)\in [0,T_1],\times [0,T_2])$ is in general discontinuous on the spatial domain $\mathcal{D}$.

\begin{figure}[ht]
	\centering
	\subfigure{\includegraphics[scale=0.18]{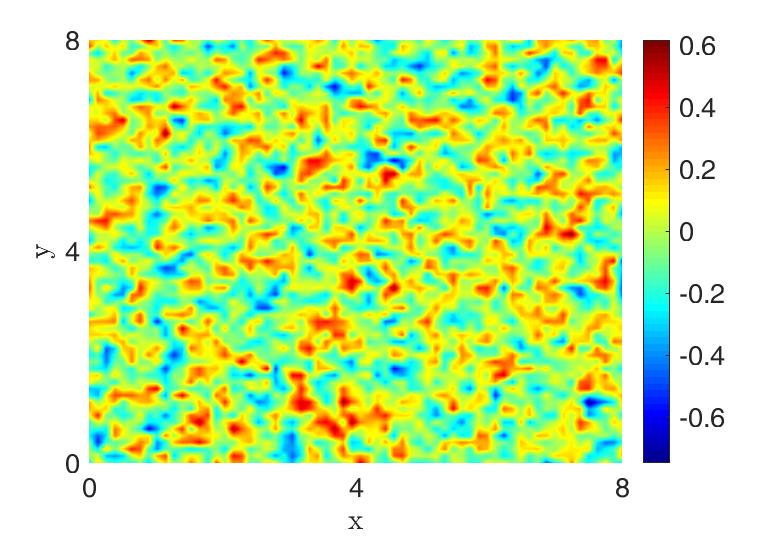}}
	\subfigure{\includegraphics[scale=0.18]{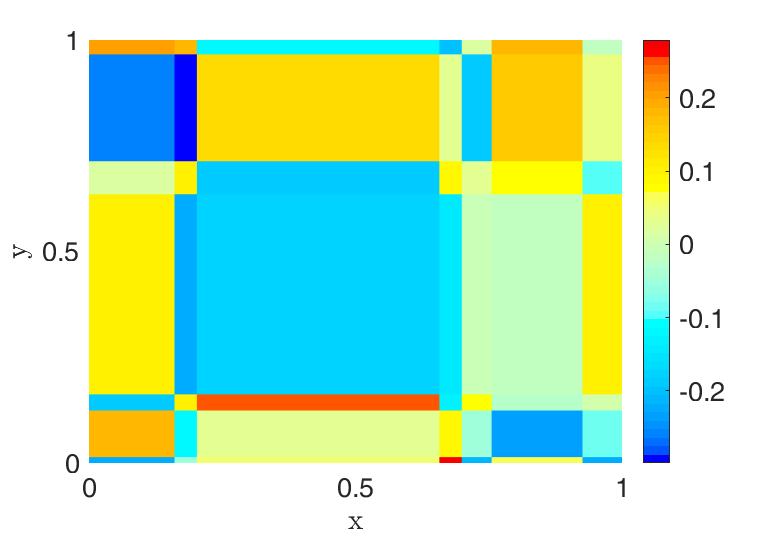}}
	\subfigure{\includegraphics[scale=0.18]{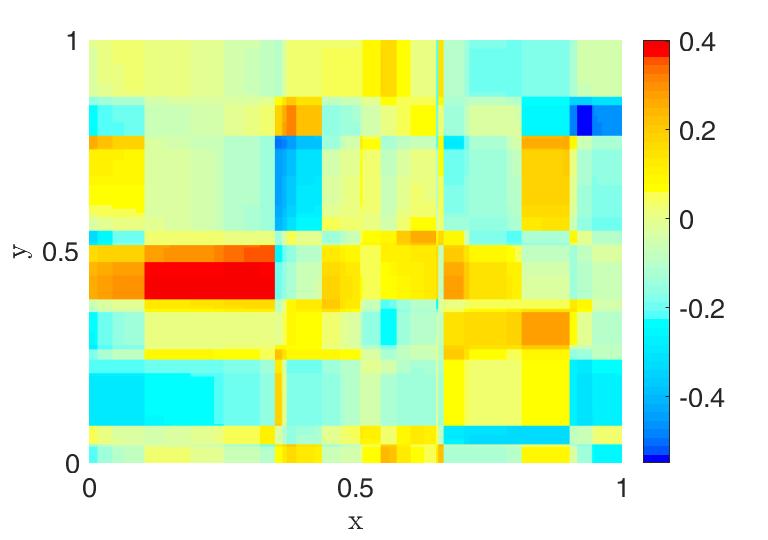}}
	\caption{Sample of a Matérn-1.5-GRF (left) and a corresponding Poisson-subordinated GRF (middle) and Gamma-subordinated GRF (right).}\label{Fig:SamplesSubordGRF}
\end{figure} 
\noindent Figure \ref{Fig:SamplesSubordGRF} demonstrates how the subordinators $l_1$ and $l_2$ create discontinuities in the subordinated GRF. In the presented samples, the underlying GRF is a Matérn-1.5 GRF. We recall that, for a given smoothness parameter $\nu_M > 1/2$, correlation parameter $r_M>0$ and variance $\sigma_M^2>0$, the Matérn-$\nu_M$ covariance function on $\mathbb{R}_+^d\times \mathbb{R}_+^d$ is given by $q_M(\underline{x},\underline{y})=\rho_M(\|\underline{x}-\underline{y}\|_2)$, for $(\underline{x},\underline{y})\in\mathbb{R}_+^d\times \mathbb{R}_+^d$, with 
\begin{align*}
\rho_M(s) = \sigma_M^2 \frac{2^{1-\nu_M}}{\Gamma(\nu_M)}\Big(\frac{2s\sqrt{\nu_M}}{r_M}\Big)^{\nu_M} K_{\nu_M}\Big(\frac{2s\sqrt{\nu_M}}{r_M}\Big), \text{ for }s\geq 0,
\end{align*}
where $\Gamma(\cdot)$ is the Gamma function and $K_\nu(\cdot)$ is the modified Bessel function of the second kind (see \cite[Section 2.2 and Proposition 1]{QuasiMonteCarloFEMethodsForEllipticPDEsWithLognormalRandomCoefficients}). A Matérn-$\nu_M$ GRF is a centered GRF with covariance function $q_M$. It has been shown in \cite{SGRF} that the subordinated GRF constructed in \eqref{EQ:SGRFDefi} is separately measurable. Further, the corresponding random fields display great distributional flexibility, allow for a Lévy-Khinchin-type formula and formulas for their covariance functions can be derived which makes them attractive for applications. We refer the interested reader to \cite{SGRF} for a theoretical investigation of the constructed random fields.\\


\section{The subordinated GRF in the elliptic model equation}\label{sec:subGRFelliptic}

In this section we incorporate the subordinated GRF in the diffusion coefficient of the elliptic PDE \eqref{EQ:EllProblem} - \eqref{EQ:EllProblemBCN}. Further, we show how to approximate the diffusion coefficient and state the most important results on the approximation of the corresponding PDE solution following \cite{SGRFPDE}. For the proofs and a more detailed study of subordinated GRFs in the elliptic model equation we refer the reader to \cite{SGRFPDE}.
\subsection{Subordinated GRFs in the diffusion coefficient}
It follows from the Lévy-It\^o decomposition that any Lévy process on a one-dimensional (time) domain can be additively decomposed into a deterministic drift part, a continuous noise part and a pure-jump process (see \cite[Section 2.4]{LevyProcessesAndStochasticCalculus}). Motivated by this, we construct the diffusion coefficient $a$ in the elliptic PDE as follows.
\begin{definition}(see \cite[Definition 3.3]{SGRFPDE})\label{DEF:DefCoeff}
We consider the domain $\mathcal{D}=(0,D)^2$ with $D<+\infty$\footnote{For simplicity we chose a square domain, rectangular ones may be considered in the same way.}. We define the jump diffusion coefficient $a$ in problem \eqref{EQ:EllProblem} - \eqref{EQ:EllProblemBCN} with $d=2$ as
\begin{align}\label{EQ:DiffCoeffDefi}
a:\Omega\times\mathcal{D}\rightarrow (0,+\infty),~(\omega,x,y)\mapsto \overline{a}(x,y) + \Phi_1(W_1(x,y)) + \Phi_2(W_2(l_1(x),l_2(y))),
\end{align}
where 
\begin{itemize}
\item $\overline{a}:\mathcal{D}\rightarrow (0,+\infty)$ is deterministic, continuous and there exist constants $\overline{a}_+,\overline{a}_->0$ with $\overline{a}_-\leq \overline{a}(x,y)\leq \overline{a}_+$ for $(x,y)\in\mathcal{D}$.
\item $\Phi_1,~\Phi_2:\mathbb{R}\rightarrow [0,+\infty)$ are continuous .
\item $W_1$ and $W_2$ are zero-mean GRFs on $\mathcal{D}$ respectively on $[0,+\infty)^2$ with $\mathbb{P}-a.s.$ continuous paths.
\item $l_1$ and $l_2$ are L\'evy subordinators on $[0,D]$.
\end{itemize}
\end{definition}	

It follows by a pathwise application of the Lax-Milgram lemma that the elliptic model problem ~\eqref{EQ:EllProblem} - \eqref{EQ:EllProblemBCN} with the diffusion coefficient constructed in Definition~\ref{DEF:DefCoeff} has a unique pathwise weak solution. 
\begin{theorem}(see \cite[Theorem 3.6]{SGRFPDE})\label{TH:ExistenceOfSolutionSubord}
Let $a$ be as in Definition \ref{DEF:DefCoeff} and let $f\in L^q(\Omega;H),~g\in L^q(\Omega;L^2(\Gamma_2))$ for some $q\in [1,+\infty)$. Then there exists a unique pathwise weak solution $u(\omega,\cdot)\in V$ to problem~\eqref{EQ:EllProblem} - \eqref{EQ:EllProblemBCN} for $\mathbb{P}$-almost every $\omega\in\Omega$. Furthermore, $u\in L^r(\Omega;V)$ for all $r\in[1,q)$ and 
\begin{align*}
\|u\|_{L^r(\Omega;V)}\leq C(\overline{a}_-,\mathcal{D})(\|f\|_{L^q(\Omega;H)} + \|g\|_{L^q(\Omega;L^2(\Gamma_2))}),
\end{align*}
where $C(\overline{a}_-,\mathcal{D})>0$ is a constant depending only on the indicated parameters.
\end{theorem}

\subsection{Problem modification}

Theorem \ref{TH:ExistenceOfSolutionSubord} guarantees the existence of a unique solution $u$ to problem \eqref{EQ:EllProblem} - \eqref{EQ:EllProblemBCN} for the specific diffusion coefficient $a$ constructed in Definition \ref{DEF:DefCoeff}. However,  accessing this pathwise weak solution numerically is a different matter. Here, we face several challenges: The first difficulty is related to the domain on which the GRF $W_2$ is defined. The Lévy subordinators $l_1$ and $l_2$ can in general attain any value in $[0,+\infty)$. Hence, it is necessary to consider the GRF $W_2$ on the unbounded domain $[0,+\infty)^2$. However, most regularity and approximation results on GRFs are formulated for the case of a parameter space which is at least bounded and cannot easily be extended to unbounded domains (see e.g.  \cite[Chapter 1]{RandomFieldsAndGeometry}). Therefore, we modify the diffusion coefficient $a$ from Definition \ref{DEF:DefCoeff} and cut the Lévy-subordinators at a deterministic threshold $K>0$ depending on the choice of the subordinator. The resulting problem then coincides with the original problem up to a set of samples, whose probability can be made arbitrary small (see \cite[Remark 4.1]{SGRFPDE}). Furthermore, we have to bound the diffusion coefficient itself by a deterministic upper bound $A$ in order to show the convergence of the solution (see \cite[Section 5]{SGRFPDE} for details). Therefore, we also cut the diffusion coefficient at a deterministic level $A>0$. It can be shown that this induces an additional error in the solution approximation which can be controlled and vanishes for growing threshold $A$ (see \cite[Section 5.1, esp. Theorem 5.3 and Theorem 5.4]{SGRFPDE}). The two described modifications of the original problem \eqref{EQ:EllProblem} - \eqref{EQ:EllProblemBCN} are formalized in the following subsection.

We define the cut function $\chi_{K}(z) :=\min(z,K)$, for $z\in[0,+\infty)$, with a positive number $K>0$. Further, for fixed numbers $K,A>0$, we consider the following problem

\begin{align}\label{EQ:EllProblemCutUpperCut}
-\nabla\cdot(a_{K,A}(\omega,\underline{x})\nabla u_{K,A}(\omega,\underline{x}))=f(\omega,\underline{x}) \text{ in }\Omega\times\mathcal{D},
	\end{align}
where we impose the boundary conditions
	\begin{align}
	u_{K,A}(\omega,\underline{x})&=0 \text{ on } \Omega\times \Gamma_1,\label{EQ:EllProblemBCDCutUpperCut}\\
	a_{K,A}(\omega,\underline{x}) \overrightarrow{n}\cdot\nabla u_{K,A}(\omega,\underline{x})&=g(\omega,\underline{x}) \text{ on } \Omega\times \Gamma_.\label{EQ:EllProblemBCNCutUpperCut}
\end{align} 
The diffusion coefficient $a_{K,A}$ is defined by \footnote{We assume one fixed $K$ for all spacial dimensions to keep notation simple. However, the results presented in the subsequent sections also hold for individual values in each spacial dimension.}
\begin{align}\label{EQ:DiffCoeffDefiCutUpperCut}
a_{K,A}:\Omega\times\mathcal{D}&\rightarrow (0,+\infty),\notag\\
(\omega,x,y)&\mapsto \chi_{A}\Big(\overline{a}(x,y) + \Phi_1(W_1(x,y)) + \Phi_2(W_2(\chi_K(l_1(x)),\chi_K(l_2(y))))\Big).
\end{align}	
Again, Theorem~\ref{TH:ExistenceOfSolutionSubord} applies in this case and yields the existence of a pathwise weak solution $u_{K,A}\in L^r(\Omega;V)$, for $r\in[1,q)$, if $f\in L^q(\Omega;H)$ and $g\in L^q(\Omega;L^2(\Gamma_2))$. In \cite{SGRFPDE}, the authors investigated in detail how this modification affects the solution $u_{K,A}$ and how the resulting error can be controlled by the choice of the deterministic thresholds $K$ and $A$. Therefore, from now on we decide to consider problem \eqref{EQ:EllProblemCutUpperCut} - \eqref{EQ:DiffCoeffDefiCutUpperCut} for a fixed choice of $K$ and $A$ and focus on the approximation of the GRFs $W_1,W_2$ and the Lévy subordinators $l_1,l_2$ in the following. We come back on the choice of $K$ and $A$ in specific situations in Section \ref{sec:numerics}.

\subsection{Approximation of the GRFs and the Lévy subordinators and convergence of the approximated solution}\label{sec:apprOfGRFAndSubord}

In order to approximate the random solution $u_{K,A}$ of problem \eqref{EQ:EllProblemCutUpperCut} - \eqref{EQ:DiffCoeffDefiCutUpperCut} we have to approximate the GRFs $W_1,W_2$ and the Lévy subordinators $l_1,l_2$ to be able to generate samples of the diffusion coefficient $a_{K,A}$ defined in Equation \eqref{EQ:DiffCoeffDefiCutUpperCut}. Therefore, we have to impose some additional assumptions on the GRFs and the Lévy subordinators. We summarize our working assumptions in the following.

\begin{assumption}(See \cite[Assumption 4.2]{SGRFPDE})\label{ASS:CutProblemEigenvalues}
Let $W_1$ be a zero-mean GRF on $[0,D]^2$ and $W_2$ be a zero-mean GRF on $[0,K]^2$. We denote by $q_1:[0,D]^2\times [0,D]^2\rightarrow\mathbb{R}$ and $q_2:[0,K]^2\times [0,K]^2\rightarrow\mathbb{R}$ the covariance functions of these random fields and by $Q_1,Q_2$ the associated covariance operators defined by
\begin{align*}
Q_j\phi=\int_{[0,z_j]^2}q_j((x,y),(x',y'))\phi(x',y')d(x',y'),
\end{align*}
for $\phi\in L^2([0,z_j]^2)$ with $z=(D,K)$ and $j=1,2$. We denote by $(\lambda_i^{(1)},e_i^{(1)},~i\in \mathbb{N})$ resp. $(\lambda_i^{(2)},e_i^{(2)},~i\in \mathbb{N})$ the eigenpairs associated to the covariance operators $Q_1$ and $Q_2$. In particular, $(e_i^{(1)},~i\in \mathbb{N})$ resp. $(e_i^{(2)},~i\in \mathbb{N})$ are orthonormal bases of $L^2([0,D]^2)$ resp. $L^2([0,K]^2)$.
\begin{enumerate}
\item We assume that the eigenfunctions are continuously differentiable and there exist positive constants $\alpha, ~\beta, ~C_e, ~C_\lambda>0$ such that for any $i\in\mathbb{N}$ it holds

\begin{align*}
\|e_i^\hione\|_{L^\infty([0,D]^2)},~\|e_i^\hitwo\|_{L^\infty([0,K]^2)}&\leq C_e,\\
  \|\nabla e_i^\hione\|_{L^\infty([0,D]^2)},~\|\nabla	e_i^\hitwo\|_{L^\infty([0,K]^2)}&\leq C_e i^\alpha,~ \\
\sum_{i=1}^ \infty (\lambda_i^\hione + \lambda_i^\hitwo)i^\beta&\leq C_\lambda	< + \infty.
\end{align*}

\item There exist constants $\phi,~\psi, C_{lip}>0$ such that the continuous functions $\Phi_1,~\Phi_2:\mathbb{R}\rightarrow[0,+\infty)$ from Definition \ref{DEF:DefCoeff} satisfy
\begin{align*}
|\Phi_1'(x)|\leq \phi\, \exp(\psi |x|),~ |\Phi_2(x)-\Phi_2(y)|\leq C_{lip}\,|x-y| \text{ for } x,y\in \mathbb{R}.
\end{align*}
In particular, $\Phi_1\in C^1(\mathbb{R})$.

\item $f\in L^q(\Omega;H)$ and $g\in L^q(\Omega;L^2(\Gamma_2))$ for some $q\in (1,+\infty).$

\item $\overline{a}:\mathcal{D}\rightarrow (0,+\infty)$ is deterministic, continuous and there exist constants $\overline{a}_+,\overline{a}_->0$ with $\overline{a}_-\leq \overline{a}(x,y)\leq \overline{a}_+$ for $(x,y)\in\mathcal{D}$.

\item $l_1$ and $l_2$ are Lévy subordinators on $[0,D]$ which are independent of the GRFs $W_1$ and $W_2$. Further, we assume that we have approximations $l_1^\apprlevy,~l_2^\apprlevy$ of these processes and there exist constants $C_l>0$ and $\eta>1$ such that for every $s\in[1,\eta)$ it holds 
\begin{align*}
\mathbb{E}(|l_j(x)-l_j^\apprlevy(x)|^s)\leq C_l\varepsilon_l,
\end{align*}
for $\varepsilon_l >0$, $x\in[0,D]$ and $j=1,2$.
\end{enumerate}
\end{assumption}

The first assumption on the eigenpairs of the GRFs is natural (see \cite{AStudyOfElliptic} and \cite{QuasiMonteCarloFEMethodsForEllipticPDEsWithLognormalRandomCoefficients}). Assumption~\ref{ASS:CutProblemEigenvalues} \textit{ii} is necessary to be able to quantify the error of the approximation of the diffusion coefficient and Assumption~\ref{ASS:CutProblemEigenvalues} \textit{iii} guarantees the existence of a solution. The last assumption ensures that we can approximate the L\'evy subordinators with a controllable $L^s$-error, which can always be achieved using piecewise constant approximations of the processes under appropriate assumptions on the tails of the distribution of the Lévy subordinators, see~\cite[Assumption 3.6, Assumption 3.7 and Theorem 3.21]{ApproximationAndSimulation}.\\
For technical reasons we have to work under the following assumption on the integrability of the gradient of the solution $\nabla u_{K,A}$ of problem \eqref{EQ:EllProblemCutUpperCut} - \eqref{EQ:DiffCoeffDefiCutUpperCut}. This assumption is necessary for the proof of the convergence of the approximation to the solution $u_{K,A}$ in Theorem \ref{TH:ErrorBoundE2}. Its origin lies in the fact that we cannot approximate the Lévy subordinators in an $L^s(\Omega;L^\infty([0,D]))$-sense on the domain due to the discontinuities. There are several results on higher integrability of the gradient of the solution to an elliptic PDE of the form \eqref{EQ:EllProblemCutUpperCut} - \eqref{EQ:DiffCoeffDefiCutUpperCut} which guarantee the condition of Assumption \ref{ASS:IntegrabilityOfSolGradient}. We refer to \cite[Section 5.2]{SGRFPDE} and especially Remark 5.6 and Remark 5.7 therein for more details.

\begin{assumption}(See \cite[Assumption 5.5]{SGRFPDE})  \label{ASS:IntegrabilityOfSolGradient}
We assume that there exist constants $j_{reg}>0$ and $k_{reg}\geq 2$ such that
\begin{align*}
C_{reg}:=\mathbb{E}(\|\nabla u_{K,A}\|_{L^{2 + j_{reg}}(\mathcal{D})}^{k_{reg}})< +\infty. 
\end{align*}
\end{assumption}

We now turn to the final approximation of the diffusion coefficient using approximations $W_1^{(\varepsilon_W)}\approx W_1$, $W_2^{(\varepsilon_W)}\approx W_2$ of the GRFs and $l_1^{(\varepsilon_l)}\approx l_1$, $l_2^{(\varepsilon_l)}\approx l_2$ of the Lévy subordinators (see Assumption \ref{ASS:CutProblemEigenvalues}): We consider discrete grids $G_1^\apprgrf=\{(x_i,x_j)|~i,j=0,\dots,M_{\varepsilon_W}^{(1)}\}$ on $[0,D]^2$ and $G_2^\apprgrf=\{(y_i,y_j)|~i,j=0,\dots,M_{\varepsilon_W}^{(2)}\}$ on $[0,K]^2$ where $(x_i,~i=0,\dots,M_{\varepsilon_W}^{(1)})$ is an equidistant grid on $[0,D]$ with maximum step size $\varepsilon_W$ and $(y_i,~i=0,...,M_{\varepsilon_W}^{(2)})$ is an equidistant grid on $[0,K]$ with maximum step size $\varepsilon_W$. Further, let $W_1^\apprgrf$ and $W_2^\apprgrf$ be approximations of the GRFs $W_1,~W_2$ on the discrete grids $G_1^\apprgrf$ resp. $G_2^\apprgrf$ which are constructed by point evaluation of the random fields $W_1$ and $W_2$ on the grid points and linear interpolation between the them.

\noindent We approximate the diffusion coefficient $a_{K,A}$ from Equation \eqref{EQ:DiffCoeffDefiCutUpperCut} by $a_{K,A}^{(\varepsilon_W,\varepsilon_l)}:\Omega\times\mathcal{D}\rightarrow(0,+\infty)$ with
\begin{align}\label{EQ:DiffCoeffApprDef}
a_{K,A}^{(\varepsilon_W,\varepsilon_l)}(x,y) = \chi_{A}\Big(\overline{a}(x,y) + \Phi_1(W_1^\apprgrf(x,y)) + \Phi_2(W_2^\apprgrf(\chi_K(l_1^\apprlevy(x)),\chi_K(l_2^\apprlevy(y))))\Big)
\end{align}
for $(x,y)\in\mathcal{D}$. 
Further, we denote by $u_{K,A}^{(\varepsilon_W,\varepsilon_l)}\in L^r(\Omega;V)$, with $r\in[1, q)$, the weak solution to the corresponding elliptic problem
   \begin{align}\label{EQ:EllProblemCutAppr}
-\nabla\cdot (a_{K,A}^{(\varepsilon_W,\varepsilon_l)}(\omega,\underline{x})\nabla u_{K,A}^{(\varepsilon_W,\varepsilon_l)}(\omega,\underline{x}))=f(\omega,\underline{x}) \text{ in }\Omega\times\mathcal{D},
	\end{align}
with boundary conditions
	\begin{align}
	u_{K,A}^{(\varepsilon_W,\varepsilon_l)}(\omega,\underline{x})&=0 \text{ on } \Omega\times \Gamma_1,\label{EQ:EllProblemBCDCutAppr}\\
	a_{K,A}^{(\varepsilon_W,\varepsilon_l)}(\omega,\underline{x}) \overrightarrow{n}\cdot\nabla u_{K,A}^{(\varepsilon_W,\varepsilon_l)}(\omega,x)&=g(\omega,x) \text{ on } \Omega\times \Gamma_2.\label{EQ:EllProblemBCNCutAppr}
	\end{align}
	
\noindent Note that Theorem~\ref{TH:ExistenceOfSolutionSubord} also applies to the elliptic problem with coefficient $a_{K,A}^{(\varepsilon_W,\varepsilon_l)}$. We are now able to state the most important result on the convergence of the approximated solution $u_{K,A}^{(\varepsilon_W,\varepsilon_l)}$ to $u_{K,A}$. For a proof we refer the reader to \cite{SGRFPDE}.

\begin{theorem}(See \cite[Theorem 5.9]{SGRFPDE})\label{TH:ErrorBoundE2}
	Let $r\geq 2$ and $b,c\in[1,+\infty]$ be given such that it holds
	\begin{align*}
	rc\gamma\geq 2 \text{ and }2b\leq rc< \eta
	\end{align*}
	with a fixed real number $\gamma\in(0,min(1,\beta/(2\alpha))$. Here, the parameters $\eta, \alpha$ and $\beta$ are determined by the GRFs $W_1$, $W_2$ and the L\'evy subordinators $l_1$, $l_2$ (see Assumption \ref{ASS:CutProblemEigenvalues}). \\
	Let $m,n\in[1,+\infty]$ be real numbers such that 
	\begin{align*}
	\frac{1}{m} + \frac{1}{c} = \frac{1}{n} + \frac{1}{b}=1,
	\end{align*}
	and let $k_{reg}\geq 2$ and $j_{reg}>0$ be the regularity specifiers given by Assumption~\ref{ASS:IntegrabilityOfSolGradient}.
	If it holds that
	\begin{align*}
	n<1+\frac{j_{reg}}{2} \text{ and } rm< k_{reg},
	\end{align*}
	then the approximated solution $u_{K,A}^{(\varepsilon_W,\varepsilon_l)}$ converges to the solution $u_{K,A}$ of the truncated problem for $\varepsilon_W,\varepsilon_l\rightarrow 0$ and it holds
	\begin{align*}
	\|u_{K,A}-u_{K,A}^{(\varepsilon_W,\varepsilon_l)}\|_{L^r(\Omega;V)}&\leq  C(\overline{a}_-,\mathcal{D},r)C_{reg}\|a_{K,A}^{(\varepsilon_W,\varepsilon_l)}-a_{K,A}\|_{L^{rc}(\Omega;L^{2b}(\mathcal{D}))}\\
	&\leq  C_{reg}C(\overline{a}_-,\mathcal{D},r)(\varepsilon_W^\gamma + \varepsilon_l^\frac{1}{rc}).
	\end{align*}
	\end{theorem}

This result is essential since it guarantees the convergence of the approximated solution $u_{K,A}^{(\varepsilon_W,\varepsilon_l)}$ to the solution $u_{K,A}$ with a controllable upper bound on the error. Further, the error estimate given by Theorem \ref{TH:ErrorBoundE2} will be used in the error equilibration for the MLMC estimator in Section \ref{sec:MLMC}. It allows to balance the errors resulting from the approximation of the diffusion coefficient and the Finite Element (FE) error resulting from the pathwise  numerical approximation of the PDE solution.


\section{Pathwise Finite Element approximation}\label{sec:approx_solution}
In this section, we describe the numerical method which is used to compute pathwise  approximations of the solution to the considered elliptic PDE following \cite[Section 6]{SGRFPDE}. We use a FE approach with standard triangulations and sample-adapted triangulations of the spatial domain, which is described in the following.
\subsection{The standard pathwise Finite Element approximation}\label{subs:PathwiseFE}
 We approximate the solution $u$ to problem~\eqref{EQ:EllProblem} - \eqref{EQ:EllProblemBCN} with diffusion coefficient $a$ given by Equation~\eqref{EQ:DiffCoeffDefi} using a pathwise FE approximation of the solution $u_{K,A}^{(\varepsilon_W,\varepsilon_l)}$ of problem~\eqref{EQ:EllProblemCutAppr} - \eqref{EQ:EllProblemBCNCutAppr} with the approximated diffusion coefficient $a_{K,A}^{(\varepsilon_W,\varepsilon_l)}$ given by \eqref{EQ:DiffCoeffApprDef}. Therefore, for almost all $\omega\in\Omega$, we aim to approximate the function $u_{K,A}^{(\varepsilon_W,\varepsilon_l)}(\omega,\cdot)\in V$ such that it holds
	\begin{align}\label{EQ:VariationalProblemApprSol}
	B_{a_{K,A}^{(\varepsilon_W,\varepsilon_l)}(\omega)} (u_{K,A}^{(\varepsilon_W,\varepsilon_l)}(\omega,\cdot), v) &:=\int_{\mathcal{D}}a_{K,A}^{(\varepsilon_W,\varepsilon_l)}(\omega,\underline{x})\nabla	u_{K,A}^{(\varepsilon_W,\varepsilon_l)}(\omega,\underline{x})\cdot\nabla v(\underline{x})d\underline{x} \notag \\ 
	&=\int_\mathcal{D}f(\omega,\underline{x})v(\underline{x})d\underline{x} + \int_{\Gamma_2} g(\omega,\underline{x})[Tv](\underline{x})d\underline{x}=:F_\omega(v),
	\end{align}
	for every $v\in V$ with fixed approximation parameters $K,A,\varepsilon_W,\varepsilon_l$. We compute a numerical approximation of the solution to this variational problem using a standard Galerkin approach with linear elements: assume $\mathcal{V}=(V_\ell,~\ell\in\mathbb{N}_0)$ is a sequence of finite-dimensional subspaces $V_\ell\subset	V$ with increasing $\dim(V_\ell)=d_\ell$. Further, we denote by $(h_\ell,~\ell\in\mathbb{N}_0)$ the corresponding refinement sizes which are assumed to converge monotonically to zero for $\ell\rightarrow \infty$. Let $\ell\in\mathbb{N}_0$ be fixed and denote by $\{v_1^{(\ell)},\dots,v_{d_\ell}^{(\ell)}\}$ a basis of $V_\ell$. The (pathwise) discrete version of~\eqref{EQ:VariationalProblemApprSol} reads: Find $u_{K,A,\ell}^{(\varepsilon_W,\varepsilon_l)}(\omega,\cdot)\in V_\ell$ such that \begin{align*}
	B_{a_{K,A}^{(\varepsilon_W,\varepsilon_l)}(\omega)}(u_{K,A,\ell}^{(\varepsilon_W,\varepsilon_l)}(\omega,\cdot),v_\ell^{(i)})=F_\omega(v_\ell^{(i)}) \text{ for all } i=1,\dots,d_\ell.
\end{align*}		
Expanding the function $u_{K,A,\ell}^{(\varepsilon_W,\varepsilon_l)}(\omega,\cdot)$ with respect to the basis $\{v_1^{(\ell)},\dots,v_{d_\ell}^{(\ell)}\}$ yields the representation
\begin{align*}
u_{K,A,\ell}^{(\varepsilon_W,\varepsilon_l)}(\omega,\cdot)=\sum_{i=1}^{d_\ell} c_iv_i^{(\ell)},
\end{align*}
where the coefficient vector $\textbf{c}=(c_1,\dots,c_{d_\ell})^T\in\mathbb{R}^{d_\ell}$ is determined by the linear equation system 
\begin{align*}
\textbf{B}(\omega)\textbf{c}=\textbf{F}(\omega),
\end{align*}
with a stochastic stiffness matrix $\textbf{B}(\omega)_{i,j}=B_{a_{K,A}^{(\varepsilon_W,\varepsilon_l)}(\omega)}(v_i^{(\ell)},v_j^{(\ell)})$ and load vector $\mathbf{F}(\omega)_i=F_\omega (v_i^{(\ell)})$ for $i,j=1,\dots,d_\ell$.

	Let $(\mathcal{K}_\ell,~\ell\in \mathbb{N}_0)$ be a sequence of triangulations on $\mathcal{D}$ and denote by $\theta_\ell>0$ the minimum interior angle of all triangles in $\mathcal{K}_\ell$. We assume $\theta_\ell\geq \theta>0$ for a positive constant $\theta$ and define the maximum diameter of the triangulation $\mathcal{K}_\ell$ by 
	$h_\ell:=\underset{K\in\mathcal{K}_\ell}{\max}\, diam(K),$
	for $\ell\in\mathbb{N}_0$ as well as the finite dimensional subspaces by 
	$
	V_\ell:=\{v\in V~|~v|_K\in\mathcal{P}_1,K\in \mathcal{K}_\ell\},$
	where $\mathcal{P}_1$ denotes the space of all polynomials up to degree one. If we assume that for $\mathbb{P}-$almost all $\omega\in\Omega$ it holds $u_{K,A}^{(\varepsilon_W,\varepsilon_l)}(\omega,\cdot)\in H^{1+\kappa_a}(\mathcal{D})$ for some positive number $\kappa_a>0$, and that  there exists a finite bound $\|u_{K,A}^{(\varepsilon_W,\varepsilon_l)}\|_{L^2(\Omega;H^{1+\kappa_a}(\mathcal{D}))}\leq C_u=C_u(K,A)$ for the fixed approximation parameters $K,A$, we immediately obtain the following estimate using C\'ea's lemma (see \cite[Section 4]{AStudyOfElliptic}, \cite[Section 6]{SGRFPDE}, \cite[Chapter 8]{EllipticDifferentialEquations})
	\begin{align*}
	\|u_{K,A}^{(\varepsilon_W,\varepsilon_l)}-u_{K,A,\ell}^{(\varepsilon_W,\varepsilon_l)}\|_{L^2(\Omega;V)}\leq C_{\theta,\mathcal{D}} \frac{A}{\overline{a}_-}C_u h_\ell^{\min(\kappa_a,1)}.
	\end{align*}	
	By construction of the subordinated GRF, we always obtain an interface geometry with fixed angles and bounded jump height in the diffusion coefficient, which have great influence on the solution regularity, see e.g.~\cite{RegularityResultsForLaplaceInterfaceProblemsInTroDimensions}. Note that, for general deterministic interface problems, one obtains a pathwise discretization error of order $\kappa_a \in(1/2,1)$ and in general one cannot expect the full order of convergence $\kappa_a=1$ without special treatment of the discontinuities of the diffusion coefficient (see \cite{TheFiniteElementMethodForELlipticEquationsWithDiscontinuousCoefficients} and~\cite{AStudyOfElliptic}). The convergence may be improved by the use of sample-adapted triangulations.
	
	\subsection{Sample-adapted triangulations}\label{subsec:SampleAdaptedFE}
	In~\cite{AStudyOfElliptic}, the authors suggest sample-adapted triangulations to improve the convergence of the FE approximation for elliptic jump diffusion coefficients. This approach is also used in this paper and the convergence of the corresponding FE method is compared to the performance with the use of standard triangulations. The construction of the sample-adapted triangulations is explained in the following. Consider a fixed $\omega\in \Omega$ and assume that the discontinuities of the diffusion coefficient are described by the partition $\mathcal{T}(\omega)=(\mathcal{T}_i,~i=1,\dots, \tau(\omega))$ of the domain $\mathcal{D}$ with $\tau(\omega)\in \mathbb{N}$ and $\mathcal{T}_i\subset \mathcal{D}$. Assume that $\mathcal{K}_\ell(\omega)$ is a triangulation of $\mathcal{D}$ which is adjusted to the partition $\mathcal{T}(\omega)$ in the sense that for every $i=1,\dots,\tau(\omega)$ it holds
	\begin{align*}
	\mathcal{T}_i\subset \bigcup_{\kappa\in\mathcal{K}_\ell(\omega)}\kappa \text{ and }\hat{h}_\ell(\omega):=\underset{K\in\mathcal{K}_\ell(\omega)}{\max}\, diam(K)\leq \overline{h}_\ell,
	\end{align*}
	for all $\ell\in\mathbb{N}_0$, where $(\overline{h}_\ell,~\ell\in\mathbb{N}_0)$ is a deterministic, decreasing sequence of refinement thresholds which converges to zero.  We denote by $\hat{V}_\ell(\omega)\subset V$ the corresponding finite-dimensional subspaces with dimension $\hat{d}_\ell(\omega)\in\mathbb{N}$.  Figure \ref{fig:AdaptTrian} illustrates the adapted triangulation for a sample of the diffusion coefficient where we used a Poisson($5$)-subordinated  Matérn-1.5-GRF.
	
	\begin{figure}[ht]
	\centering
	\subfigure{\includegraphics[scale=0.4]{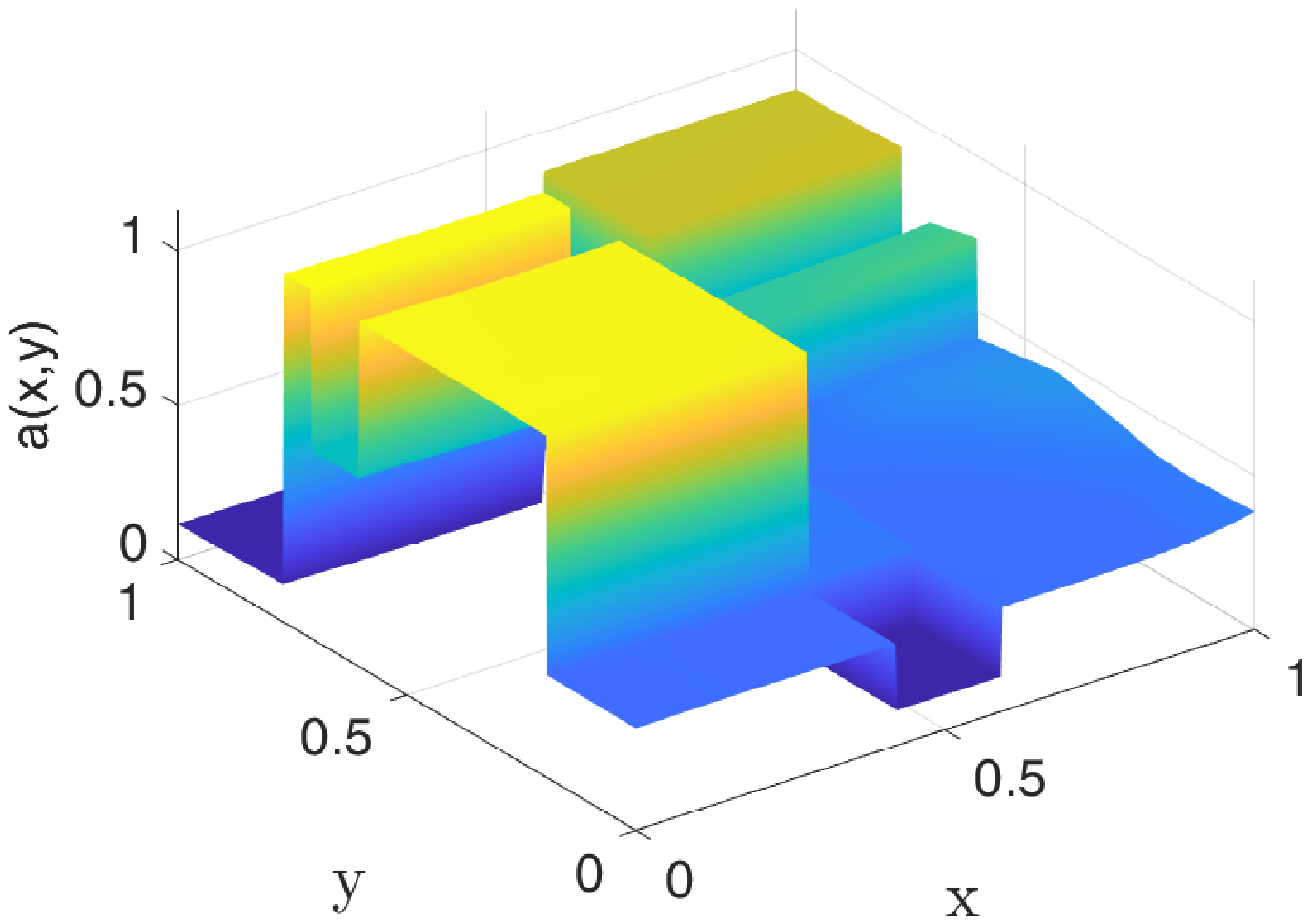}}
	\subfigure{\includegraphics[scale=0.4]{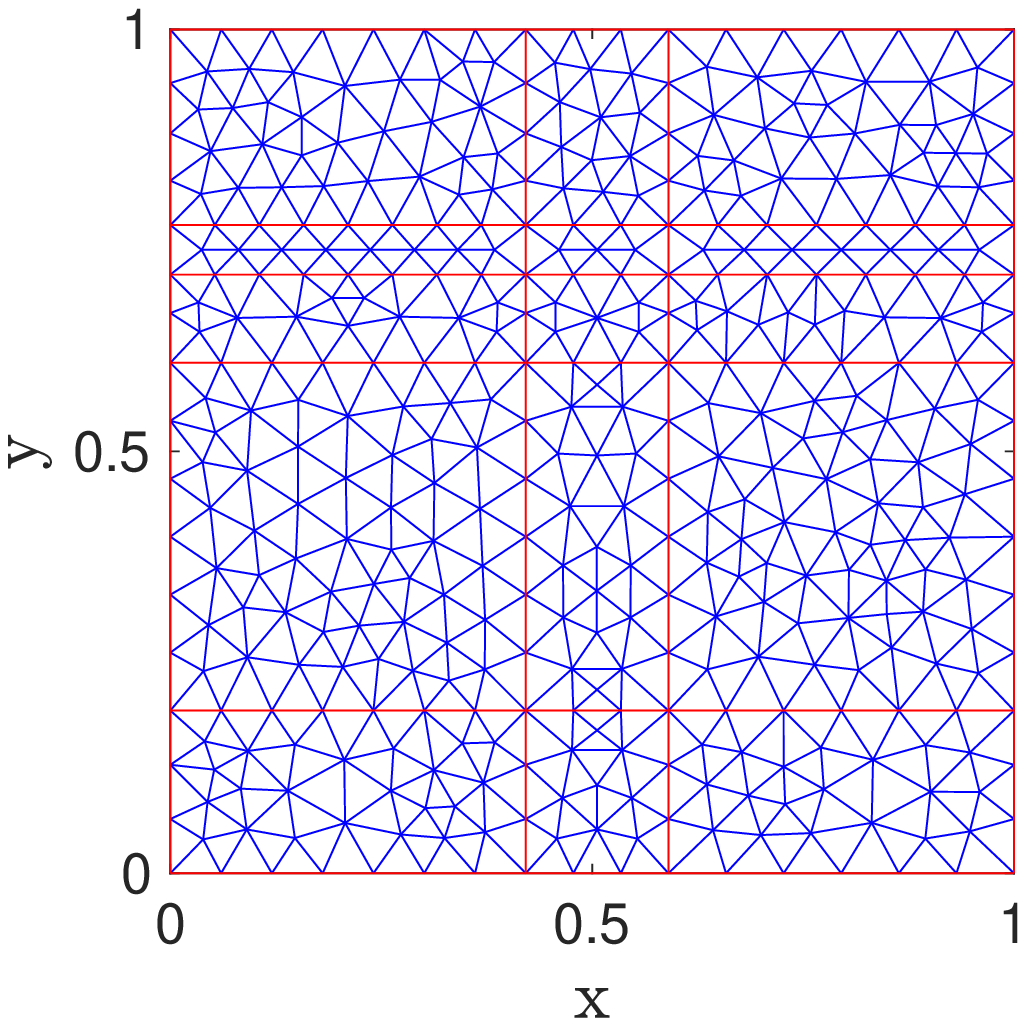}}
	\caption{Sample of the diffusion coefficient using a Poisson-subordinated Mat\'ern-1.5-GRF (left) with corresponding sample-adapted triangulation (right).}
	\label{fig:AdaptTrian}
	\end{figure}

	The sample-adapted approach leads to an improved sample-wise convergence rate for the elliptic PDE with discontinuous diffusion coefficient (see e.g. \cite[Section 4.1]{AStudyOfElliptic}). This is particularly true in the situation of jump diffusion coefficients with polygonal jump geometry, which is the case for the diffusion coefficients considered in this paper (see Figure \ref{fig:AdaptTrian}, \cite{AStudyOfElliptic}, \cite{FEMAndTheirConvergenceForEllipticAndParabolicProblems}, \cite[Section 7]{SGRFPDE}). \\
	\noindent While mean squared convergence rates cannot be derived theoretically  in our general setting due to the stochastic regularity of the PDE solutions, in practice one at least recovers the convergence rates of the deterministic jump diffusion problem in the strong error, which also has been investigated numerically in \cite{SGRFPDE}. This observation, together with the comments in the end of Subsection \ref{subs:PathwiseFE}, motivate the following assumption for the remaining theoretical analysis (see~\cite[Assumption 6.2]{SGRFPDE}).
	\begin{assumption}\label{ASS:ConvRateFEMMeanSquareError}
	There exist deterministic constants $\hat{C}_{u}, C_{u},\hat{\kappa}_a,\kappa_a>0$ such that for any $\varepsilon_W,\varepsilon_l>0$ and any $\ell\in\mathbb{N}_0$, the FE approximation errors of $\hat{u}_{K,A,\ell}^{(\varepsilon_W,\varepsilon_l)}\approx u_{K,A}^{(\varepsilon_W,\varepsilon_l)}$ in the (sample-adapted) subspaces $\hat{V}_\ell$, respectively $u_{K,A,\ell}^{(\varepsilon_W,\varepsilon_l)}\approx u_{K,A}^{(\varepsilon_W,\varepsilon_l)}$ in $V_\ell$, are bounded by
	\begin{align*}
	\|u_{K,A}^{(\varepsilon_W,\varepsilon_l)} - \hat{u}_{K,A,\ell}^{(\varepsilon_W,\varepsilon_l)}\|_{L^2(\Omega;V)}&\leq \hat{C}_{u} \mathbb{E}(\hat{h}_\ell^{2\hat{\kappa}_a})^{1/2} , \text{ respectively,}\\
	\|u_{K,A}^{(\varepsilon_W,\varepsilon_l)} - u_{K,A,\ell}^{(\varepsilon_W,\varepsilon_l)}\|_{L^2(\Omega;V)}&\leq C_{u}h_\ell^{\kappa_a},
	\end{align*}
	where the constants $\hat{C}_{u},C_{u}$ may depend on $a,f,g,K,A$ but are independent of $\hat{h}_\ell,h_\ell,\hat{\kappa}_a$ and $\kappa_a$.
	\end{assumption}	
	
	
 \section{MLMC estimation of the solution}\label{sec:MLMC}

In this section we construct a multilevel Monte Carlo (MLMC) estimator for the expectation $\mathbb{E}(u_{K,A})$ of the PDE solution and prove an a-priori bound for the approximation error. We start with the introduction of a general singlelevel Monte Carlo (SLMC) estimation since the MLMC estimator is an extension of this approach. \\[0.5em]
\noindent The next lemma follows by the definition of the inner product on the Sobolev space $H^1(\mathcal{D}) =: V$ and will be useful in our theoretical investigations.
\begin{lemma}\label{LE:IndepRVsExpectations}
For independent, centered $V$-valued random variables $Z_1$ and $Z_2$ it holds
\begin{align*}
\mathbb{E}((Z_1,Z_2)_V ) = 0.
\end{align*}
\begin{proof}
We use the definition of the inner product on $V$ together with the independence of $Z_1$ and $Z_2$ to calculate
\begin{align*}
\mathbb{E}((Z_1,Z_2)_V) & = \int_\mathcal{D} \mathbb{E}(\partial_x Z_1)\mathbb{E}( \partial_x Z_2) + \mathbb{E}(\partial_yZ_1)\mathbb{E}(\partial_yZ_2) + \mathbb{E}(Z_1)\mathbb{E}(Z_2) d(x,y) = 0.
\end{align*}
\end{proof}

\end{lemma}
Let $(u^{(i)},~i\in \mathbb{N})\subset L^2(\Omega;V)$ be a sequence of $i.i.d.$ random variables and $M\in\mathbb{N}$ a fixed sample number. The \textit{singlelevel Monte Carlo} estimator for the approximation of the mean $\mathbb{E}(u^{(1)})$ is defined by 
\begin{align*}
E_M(u^{(1)}):=\frac{1}{M}\sum_{i=1}^M u^{(i)}\approx \mathbb{E}(u^{(1)}),
\end{align*}
and we have the following standard result (see also \cite{MLMCFEMEllPDE} and \cite{AStudyOfElliptic}).
\begin{lemma}\label{LE:SLMCstandardEst}
Let $M\in\mathbb{N}$ and $(u^{(i)},~i\in \mathbb{N})\subset L^2(\Omega;V)$ be a sequence of $i.i.d.$ random variables. It holds
\begin{align*}
\|\mathbb{E}(u^{(1)}) - E_M(u^{(1)})\|_{L^2(\Omega;V)} = \sqrt{\frac{Var(u^{(1)})}{M}} &\leq \frac{\|u^{(1)}\|_{L^2(\Omega;V)}}{\sqrt{M}} \text{ and }\\ \|E_M(u^{(1)})\|_{L^2(\Omega;V)} &\leq \|u^{(1)}\|_{L^2(\Omega;V)}.
\end{align*}
\end{lemma}

One major disadvantage of the SLMC estimator described above is the slow convergence of the (statistical) error for increasing sample numbers $M$ (see Lemma \ref{LE:SLMCstandardEst} and \cite{MultilevelMonteCarloPathSimulation}). Multilevel Monte Carlo (MLMC) uses multigrid concepts to reduce the computational complexity for the estimation of the mean compared to the singlelevel approach. The idea is to compute samples of FE approximations with different accuracy where one takes many samples of FE approximations with lower accuracy (and lower computationally costs) and less samples of FE approximations with higher accuracy (and higher computational cost), see also \cite{MultilevelMonteCarloPathSimulation} and \cite{GilesMLMCMethods}.\\
 
 \noindent For fixed parameters $K,A$ the goal is to approximate the value $\mathbb{E}(u_{K,A})$. For ease of notation, we focus here on the sample-adapted discretization with the corresponding approximation $\hat{u}_{K,A,\ell}^{(\varepsilon_W,\varepsilon_l)}$ with average refinement parameter $\mathbb{E}(\hat{h}_\ell^{2\hat{\kappa}_a})^{1/2}$ and convergence rate $\hat{\kappa}_a$ in this section (see Assumption \ref{ASS:ConvRateFEMMeanSquareError}). However, the reader should always keep in mind that all results also hold in the case of standard triangulations where $\mathbb{E}(\hat{h}_\ell^{2\hat{\kappa}_a})^{1/2}$ should be replaced by $h_\ell^{\kappa_a}$.\\[0.5em]
\noindent Assume a maximum level $L\in\mathbb{N}$ is given. We consider finite-dimensional subspaces $(\hat{V}_\ell,~\ell=0,\dots,L)$ of $V$ with refinement sizes $\hat{h}_0>\dots>\hat{h}_L>0$ and approximation parameters $\varepsilon_{W,0}>\dots>\varepsilon_{W,L}$ for the GRFs and $\varepsilon_{l,0}>\dots>\varepsilon_{l,L}$ for the Lévy subordinators. Since we fix the parameters $K$ and $A$ in this analysis, we omit them in the following and use the notation $\hat{u}_{\varepsilon_{W,\ell},\varepsilon_{l,\ell},\ell}:=\hat{u}_{K,A,\ell}^{(\varepsilon_{W,\ell},\varepsilon_{l,\ell})}$ for the FEM approximation of $u_{K,A}^{(\varepsilon_{W,\ell},\varepsilon_{l,\ell})}$ on $\hat{V}_\ell$, for $\ell=-1,\dots,L$, where we set $\hat{u}_{\varepsilon_{W,-1},\varepsilon_{l,-1},-1}:=0$. If we expand the expectation on the finest level in a telescopic sum we obtain the following representation
 \begin{align}\label{EQ:MLMCTelescopicSum}
 \mathbb{E}(\hat{u}_{\varepsilon_{W,L},\varepsilon_{l,L},L})=\sum_{\ell=0}^L \mathbb{E}(\hat{u}_{\varepsilon_{W,\ell},\varepsilon_{l,\ell},\ell} - \hat{u}_{\varepsilon_{W,\ell-1},\varepsilon_{l,\ell-1},\ell-1}).
 \end{align}
 This motivates the multilevel Monte Carlo estimator, which estimates the left hand side of Equation \eqref{EQ:MLMCTelescopicSum} by singlelevel Monte Carlo estimations of each summand on the right hand side (see \cite{MultilevelMonteCarloPathSimulation}). To be precise, let $M_\ell$ be a natural number for $\ell=0,...,L$. The multilevel Monte Carlo estimator of $\hat{u}_{\varepsilon_{W,L},\varepsilon_{l,L},L}$ is then defined by
 \begin{align*}
 E^L(\hat{u}_{\varepsilon_{W,L},\varepsilon_{l,L},L})&:= \sum_{\ell=0}^L E_{M_\ell} (\hat{u}_{\varepsilon_{W,\ell},\varepsilon_{l,\ell},\ell} - \hat{u}_{\varepsilon_{W,\ell-1},\varepsilon_{l,\ell-1},\ell-1}),\\
 &= \sum_{\ell=0}^L \frac{1}{M_\ell}\sum_{i=1}^{M_\ell} (\hat{u}_{\varepsilon_{W,\ell},\varepsilon_{l,\ell},\ell}^{(i,\ell)} - \hat{u}_{\varepsilon_{W,\ell-1},\varepsilon_{l,\ell-1},\ell-1}^{(i,\ell)})
 \end{align*}
 where $(\hat{u}_{\varepsilon_{W,\ell},\varepsilon_{l,\ell},\ell}^{(i,\ell)})_{i=1}^{M_\ell}$ (resp. $(\hat{u}_{\varepsilon_{W,\ell-1},\varepsilon_{l,\ell-1},\ell-1}^{(i,\ell)})_{i=1}^{M_\ell}$)  are $M_\ell$ $i.i.d.$ copies of the random variable $\hat{u}_{\varepsilon_{W,\ell},\varepsilon_{l,\ell},\ell}$ (resp. $\hat{u}_{\varepsilon_{W,\ell-1},\varepsilon_{l,\ell-1},\ell-1}$) for $\ell=0,\dots,L$  (see also \cite{MultilevelMonteCarloPathSimulation}).
 The following result gives an a-priori bound on the MLMC error. Similar formulations can be found, for example, in \cite{MLMCFEMEllPDE}, \cite{MLMCForStochEllMultiscalePDEs} and \cite{AStudyOfElliptic}.

 \begin{theorem}\label{TH:MLMCEst}
 We set $r=2$ and assume $q>2$ in Assumption \ref{ASS:CutProblemEigenvalues}. Further, let $b,c\geq 1$ be given such that Theorem \ref{TH:ErrorBoundE2} holds. For $L\in\mathbb{N}$, let $\hat{h}_\ell>0$, $M_\ell$, $\varepsilon_{W,\ell}>0$ and $\varepsilon_{l,\ell}>0$ be the level-dependent approximation parameters for $\ell=0,...,L$ such that $\hat{h}_{\ell},~ \varepsilon_{W,\ell},$ and $\varepsilon_{l,\ell}$ are decreasing with respect to $\ell$. It holds
 \begin{align*}
 \|\mathbb{E}(u_{K,A}) - E^L(\hat{u}_{\varepsilon_{W,L},\varepsilon_{l,L},L})\|_{L^2(\Omega;V)}\leq C &\Big(\varepsilon_{W,L}^\gamma + \varepsilon_{l,L}^\frac{1}{2c} + \mathbb{E}(\hat{h}_L^{2\hat{\kappa}_a})^{1/2} + \frac{1}{\sqrt{M_0}} \\
 &+\sum_{\ell=0}^{L-1} \frac{\varepsilon_{W,\ell}^\gamma + \varepsilon_{l,\ell}^\frac{1}{2c} + \mathbb{E}(\hat{h}_\ell^{2\hat{\kappa}_a})^{1/2}}{\sqrt{M_{\ell+1}}}\Big),
 \end{align*}
 where $C>0$ is a constant which is independent of $L$ and the level-dependent approximation parameters. Note that the numbers $\gamma>0$ and $c\geq 1$ are determined by the GRFs resp. the subordinators (cf. Theorem \ref{TH:ErrorBoundE2}).
 \end{theorem}
 
 \begin{proof}
 We estimate
 \begin{align*}
 \|\mathbb{E}(u_{K,A}) - E^L(\hat{u}_{\varepsilon_{W,L},\varepsilon_{l,L},L})\|_{L^2(\Omega;V)}&\leq \|\mathbb{E}(u_{K,A}) - \mathbb{E}(\hat{u}_{\varepsilon_{W,L},\varepsilon_{l,L},L}) \|_{L^2(\Omega;V)} \\
 &+ \|\mathbb{E}(\hat{u}_{\varepsilon_{W,L},\varepsilon_{l,L},L}) - E^L(\hat{u}_{\varepsilon_{W,L},\varepsilon_{l,L},L})\|_{L^2(\Omega;V)}\\
 &=:I_1+I_2.
 \end{align*}
 We use the triangular inequality, Theorem \ref{TH:ErrorBoundE2} and Assumption \ref{ASS:ConvRateFEMMeanSquareError} to obtain
 \begin{align*}
 I_1&\leq \mathbb{E}(\|u_{K,A} - \hat{u}_{\varepsilon_{W,L},\varepsilon_{l,L},L}\|_V)
 \leq C_{reg}C(\overline{a}_-,\mathcal{D})(\varepsilon_{W,L}^\gamma + \varepsilon_{l,L}^\frac{1}{2c}) + \hat{C}_{u}\mathbb{E}(\hat{h}_L^{2\hat{\kappa}_a})^\frac{1}{2}.
 \end{align*}
 For the second term we use the definition of the MLMC estimator $E^L$ and Lemma \ref{LE:SLMCstandardEst} to obtain
 \begin{align*}
 I_2&\leq 
 \sum_{\ell=0}^L \Big(\frac{Var(\hat{u}_{\varepsilon_{W,\ell},\varepsilon_{l,\ell},\ell}-\hat{u}_{\varepsilon_{W,\ell-1},\varepsilon_{l,\ell-1},\ell-1})}{M_\ell}\Big) ^{1/2}\\
 &\leq\sum_{\ell=0}^L \frac{1}{\sqrt{M_\ell}}\Big( \|\hat{u}_{\varepsilon_{W,\ell},\varepsilon_{l,\ell},\ell} - u_{K,A}\|_{L^2(\Omega;V)} + \|u_{K,A}- \hat{u}_{\varepsilon_{W,\ell-1},\varepsilon_{l,\ell-1},\ell-1}\|_{L^2(\Omega;V)}\Big).
 \end{align*}
 Similar as for the first summand $I_1$ we apply Theorem \ref{TH:ErrorBoundE2} and Assumption \ref{ASS:ConvRateFEMMeanSquareError}  to get
 \begin{align*}
 \|\hat{u}_{\varepsilon_{W,\ell},\varepsilon_{l,\ell},\ell} - u_{K,A}\|_{L^2(\Omega;V)}\leq C_{reg}C(\overline{a}_-,\mathcal{D})(\varepsilon_{W,\ell}^\gamma + \varepsilon_{l,\ell}^\frac{1}{2c}) + \hat{C}_{u}\mathbb{E}(\hat{h}_\ell^{2\hat{\kappa}_a})^\frac{1}{2},
 \end{align*}
 for $\ell=0,\dots,L$ and for $\ell=-1$ it follows from Theorem \ref{TH:ExistenceOfSolutionSubord} that
 \begin{align*}
 \|u_{K,A}\|_{L^2(\Omega;V)}\leq C(\overline{a}_-,\mathcal{D})(\|f\|_{L^q(\Omega;H)} + \|g\|_{L^q(\Omega;L^2(\Gamma_2))}),
 \end{align*}
 since $q>2$. Finally, we calculate
 \begin{align*}
  I_1+I_2&\leq C_{reg}C(\overline{a}_-,\mathcal{D})\Big( \varepsilon_{W,L}^\gamma + \varepsilon_{l,L}^\frac{1}{2c} + \frac{\varepsilon_{W,0}^\gamma + \varepsilon_{l,0}^\frac{1}{2c}}{\sqrt{M_0}} + \sum_{\ell=1}^L \frac{1}{\sqrt{M_\ell}}( \varepsilon_{W,\ell}^\gamma + \varepsilon_{l,\ell}^\frac{1}{2c} + \varepsilon_{W,\ell-1}^\gamma + \varepsilon_{l,\ell-1}^\frac{1}{2c}) \Big)\\
  &+ \hat{C}_{u} \Big( \mathbb{E}(\hat{h}_L^{2\hat{\kappa}_a})^\frac{1}{2} + \frac{\mathbb{E}(\hat{h}_0^{2\hat{\kappa}_a})^{1/2}}{\sqrt{M_0}} +  \sum_{\ell=1}^L \frac{1}{\sqrt{M_\ell}} (\mathbb{E}(\hat{h}_\ell^{2\hat{\kappa}_a})^\frac{1}{2} +\mathbb{E}(\hat{h}_{\ell-1}^{2\hat{\kappa}_a})^\frac{1}{2})\Big)\\
  &+\frac{1}{\sqrt{M_0}} C(\overline{a}_-,\mathcal{D})(\|f\|_{L^q(\Omega;H)} + \|g\|_{L^q(\Omega;L^2(\Gamma_2))})\\
  &\leq C\Big(\varepsilon_{W,L}^\gamma + \varepsilon_{l,L}^\frac{1}{2c} + \mathbb{E}(\hat{h}_L^{2\hat{\kappa}_a})^{1/2} + \frac{1}{\sqrt{M_0}} +\sum_{\ell=0}^{L-1} \frac{\varepsilon_{W,\ell}^\gamma + \varepsilon_{l,\ell}^\frac{1}{2c} + \mathbb{E}(\hat{h}_\ell^{2\hat{\kappa}_a})^{1/2}}{\sqrt{M_{\ell+1}}}\Big),
 \end{align*}
 where we used the monotonicity of $(\varepsilon_{W,\ell})_{\ell=0}^L$, $(\varepsilon_{l,\ell})_{\ell=0}^L$ and $(\hat{h}_\ell)_{\ell=0}^L$.
 \end{proof}
 The error estimate of Theorem \ref{TH:MLMCEst} allows for an equilibriation of the error contributions resulting from the approximation of the diffusion coefficient and the approximation of the pathwise solution with the FE method which then leads to a higher computational efficiency compared to the singlelevel approach. This leads in general to the strategy that one takes only few of the accurate, but expensive samples for large $\ell\in\{0,\dots,L\}$ and one generates more on the cheap, but less accurate samples on the lower levels, which can be seen in the following corollary (see also \cite[Section 5]{AStudyOfElliptic}, \cite{MultilevelMonteCarloPathSimulation} and \cite{GilesMLMCMethods}).

 \begin{corollary}\label{COR:ParameterChoiceMLMC}
 Let the assumptions of Theorem \ref{TH:MLMCEst} hold. For $L\in\mathbb{N}$ and given (stochastic) refinement parameters $\hat{h}_0>\dots>\hat{h}_L>0$ choose $\varepsilon_{W,\ell}>0$ and $\varepsilon_{l,\ell}>0$ such that 
\begin{align}\label{EQ:RuleApprParams}
\varepsilon_{W,\ell} \simeq \mathbb{E}(\hat{h}_\ell^{2\hat{\kappa}_a})^{1/(2\gamma)} \text{ and } \varepsilon_{l,\ell} \simeq \mathbb{E}(\hat{h}_\ell^{2\hat{\kappa}_a})^{c}.
\end{align} 
 and sample numbers $M_\ell\in\mathbb{N}$ according to
 \begin{align} \label{EQ:RuleSampleNumbers}
 M_0\simeq \mathbb{E}(\hat{h}_L^{2\hat{\kappa}	_a})^{-1} \text{ and } M_\ell\simeq \mathbb{E}(\hat{h}_L^{2\hat{\kappa}_a})^{-1}\mathbb{E}(\hat{h}_{\ell-1}^{2\hat{\kappa}_a})(\ell+1)^{2(1+\xi)}\text{ for }\ell=1,\dots,L,
 \end{align}
 for some positive parameter $\xi>0$. Then, it holds
 \begin{align*}
 \|\mathbb{E}(u_{K,A}) - E^L(\hat{u}_{\varepsilon_{W,L},\varepsilon_{l,L},L})\|_{L^2(\Omega;V)}=\mathcal{O}(\mathbb{E}(\hat{h}_L^{2\hat{\kappa}_a})^{1/2}).
 \end{align*}
 \end{corollary}
 
 \begin{proof}
 We use Theorem \ref{TH:MLMCEst} together with Equation \eqref{EQ:RuleApprParams} and Equation \eqref{EQ:RuleSampleNumbers} to obtain
 \begin{align*}
 \|\mathbb{E}(u_{K,A}) &- E^L(\hat{u}_{\varepsilon_{W,L},\varepsilon_{l,L},L})\|_{L^2(\Omega;V)}\\ &\leq C \Big(\varepsilon_{W,L}^\gamma + \varepsilon_{l,L}^\frac{1}{2c} + \mathbb{E}(\hat{h}_L^{2\hat{\kappa}_a})^{1/2} + \frac{1}{\sqrt{M_0}} +\sum_{\ell=0}^{L-1} \frac{\varepsilon_{W,\ell}^\gamma + \varepsilon_{l,\ell}^\frac{1}{2c} + \mathbb{E}(\hat{h}_\ell^{2\hat{\kappa}_a})^{1/2}}{\sqrt{M_{\ell+1}}}\Big)\\
 &\leq C\mathbb{E}(\hat{h}_L^{2\hat{\kappa}_a})^{1/2} \Big( 4 + \sum_{\ell=1}^L \frac{1}{(\ell+1)^{1+\xi}}  \Big)\leq C(4+ \zeta(1+\xi))\mathbb{E}(\hat{h}_L^{2\hat{\kappa}_a})^{1/2},
 \end{align*}
 where $\zeta(\cdot)$ denotes the Riemann zeta function.
 \end{proof}
	

\section{Multilevel Monte Carlo with Control Variates}\label{sec:MLMCCV}
	
The jump-discontinuities in the coefficient $a_{K,A}$ of the elliptic problem \eqref{EQ:EllProblemCutUpperCut} - \eqref{EQ:DiffCoeffDefiCutUpperCut} have a negative impact on the FE convergence due to the low regularity of the solution (see Section \ref{sec:approx_solution} and \cite{SGRFPDE}). In Subsection \ref{subsec:SampleAdaptedFE} we presented one possible approach to enhance the FE convergence for discontinuous diffusion coefficients: the sample-adapted FE approach with triangulations adjusted to the discontinuities. However, this approach may be computationally not feasible anymore if one has many jump interfaces. For instance, using subordinators with high jump activity (e.g. Gamma subordinators) may result in a very high number of discontinuities making the construction of sample-adapted triangulations extremely expensive. Besides the usage of adapted triangulations, \textit{variance reduction} techniques can also be used to improve the computational efficiency of the MLMC estimation of the mean of the PDE solution, as we see in this section. We start with an introduction to a specific variance reduction technique, the \textit{Control Variates} (CV), and show subsequently how we use a Control Variate in our setting (cf. \cite{NobileTeseiMLCV}).

\subsection{Control variates as a variance reduction technique}
Assume $Y$ is a real-valued, square integrable random variable and $(Y_i,~i\in \mathbb{N})$ is a sequence of i.i.d. random variables which follow the same distribution as $Y$. For a fixed number of samples $M\in\mathbb{N}$, the SLMC estimator for the estimation of the expectation $\mathbb{E}(Y)$ is given by $
E_M(Y)=1/M\sum_{i=1} ^M Y_i$
(see Section \ref{sec:MLMC}) and we have the following representation for the statistical error (see Lemma \ref{LE:SLMCstandardEst}):
\begin{align}\label{EQ:StatisticalErrorSLMC}
\|\mathbb{E}(Y) - E_M(Y)\|_{L^2(\Omega;\mathbb{R})} = \sqrt{\frac{Var(Y)}{M}}.
\end{align}
The use of Control Variates aims to reduce the statistical error of a MC estimation by reducing the variance on the right hand side of \eqref{EQ:StatisticalErrorSLMC}. Assume we are given another real valued, square integrable random variable $X$ with known expectation $\mathbb{E}(X)$ and a corresponding sequence of $i.i.d.$ random variables $(X_i,~i\in\mathbb{N})$ following the same distribution as $X$. For a given number of samples $M\in \mathbb{N}$, the control variate estimator is then defined by
\begin{align*}
E_M^{CV} (Y) = \frac{1}{M}\sum_{i=1}^M Y_i-(X_i-\mathbb{E}(X)),
\end{align*}
(see, for example, \cite[Section 4.1]{GlassermanMCMethods}). The estimator $E_M^{CV}(Y)$ is unbiased for the estimation of $\mathbb{E}(Y)$ and it can be shown that the variance of the estimator $E_M^{CV}(Y)$, i.e. the statistical error, can be reduced, if the random variables $X$ and $Y$ are correlated (see \cite[Section 4.1.1]{GlassermanMCMethods}). \\

In \cite{NobileTeseiMLCV}, the authors presented a MLMC-CV combination for the estimation of the mean of the solution to the problem \eqref{EQ:EllProblem} - \eqref{EQ:EllProblemBCN}, where the diffusion coefficient $a$ is modeled as a lognormal GRF. They use a smoothed version of the GRF and the pathwise solution to the corresponding PDE problem to construct a highly-correlated Control Variate. The considered GRFs have at least continuous paths leading to continuous diffusion coefficients. In the following, we show how we use a similar approach for our discontinuous diffusion coefficients to enhance the efficiency of the MLMC estimator for the case of subordinators with high jump activity.

\subsection{Smoothing the diffusion coefficient}
In this section we construct the Control Variate which is used to enhance the MLMC estimation of the mean of the solution to \eqref{EQ:EllProblemCutUpperCut} - \eqref{EQ:DiffCoeffDefiCutUpperCut} for subordinators with high jump activity. Our approach is motivated by \cite{NobileTeseiMLCV}.\\

\noindent For a positive smoothing parameter $\nu_s>0$ we consider the Gaussian kernel on $\mathbb{R}^2$:
\begin{align*}
\phi_{\nu_s}(x,y) = e^{-\frac{x^2 + y^2}{2\nu_s^2}}(2\pi \nu_s ^2)^{-1}\text{, for }(x,y)\in\mathbb{R}^2.
\end{align*}
Further, we identify the jump diffusion coefficient $a_{K,A}$ from Equation \eqref{EQ:DiffCoeffDefiCutUpperCut} with its extended version on the domain $\mathbb{R}^2$, where we set $a_{K,A} (x,y)= 0$ for $(x,y)\in\mathbb{R}^2\setminus \mathcal{D}$, and define the smoothed version $a_{K,A}^{(\nu_s)}$ by convolution with the Gaussian kernel:

\begin{align*}
a_{K,A}^{(\nu_s)} (x,y) &= \int_{\mathbb{R}^2} \phi_{\nu_s}(x',y')a_{K,A}(x-x',y-y')d(x',y')\\
 &= \int_{\mathbb{R}^2} \phi_{\nu_s}(x-x',y-y')a_{K,A}(x',y')d(x',y'),\text{ for }(x,y)\in \mathcal{D}.
\end{align*}
Obviously, Theorem \ref{TH:ExistenceOfSolutionSubord} applies also to the smoothed diffusion coefficient $a_{K,A}^{(\nu_s)}$ which guarantees the existence of a solution $u_{K,A}^{(\nu_s)}\in L^r(\Omega;V)$, for $r\in[1,q)$ with $f\in L^q(\Omega;H)$ and $g\in L^q(\Omega;L^2(\Gamma_2))$, and yields the bound
\begin{align}\label{EQ:SolBoundSmoothProblem}
\|u_{K,A}^{(\nu_s)}\|_{L^r(\Omega;V)}\leq C(\overline{a}_-,\mathcal{D},\nu_s)(\|f\|_{L^q(\Omega;H)} + \|g\|_{L^q(\Omega;L^2(\Gamma_2))}).
\end{align}

\noindent If the smoothing parameter $\nu_s$ is small, the solution corresponding to the smoothed coefficient $a_{K,A}^{(\nu_s)}$ is highly correlated with the solution to the PDE with (unsmoothed) diffusion coefficient $a_{K,A}$. Therefore, the smoothed solution is a reasonable choice as a Control Variate in the MLMC estimator being both: highly correlated with the solution to the rough problem and easy to approximate using the FE method due to the high regularity compared to the rough problem (see also \cite{NobileTeseiMLCV} and \cite[Sections 8 and 9]{EllipticDifferentialEquations}). Figure \ref{Fig:GKSmoothingSamples} shows a sample of the diffusion coefficient and smoothed versions using a Gaussian kernel with different smoothness parameters.

\begin{figure}[ht]
	\centering
	\subfigure{\includegraphics[scale=0.18]{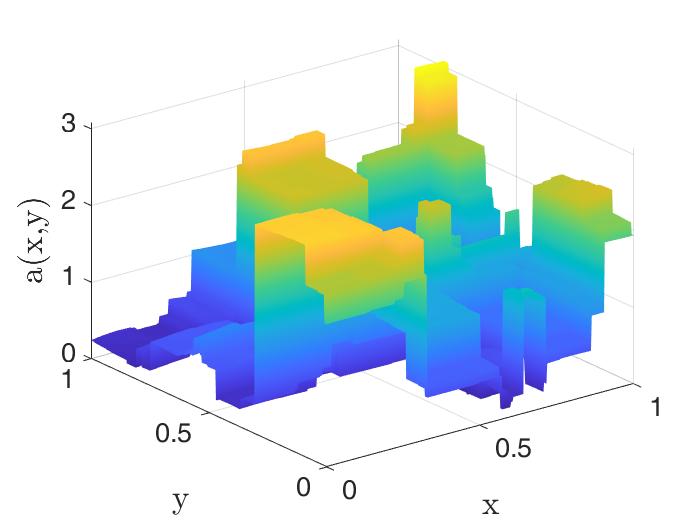}}
	\subfigure{\includegraphics[scale=0.18]{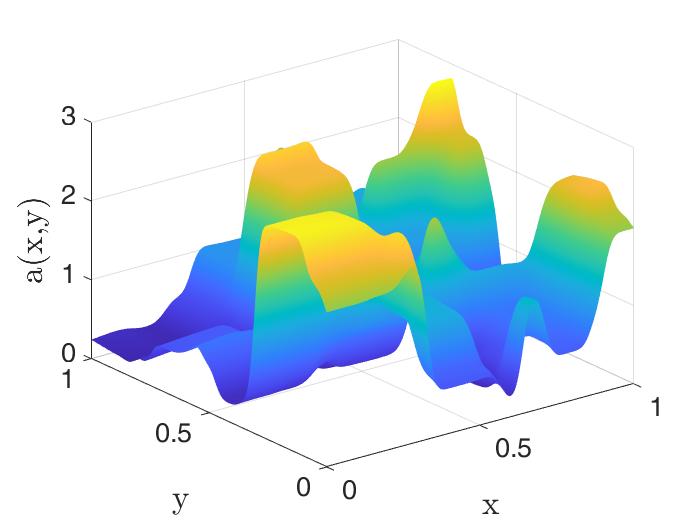}}
	\subfigure{\includegraphics[scale=0.18]{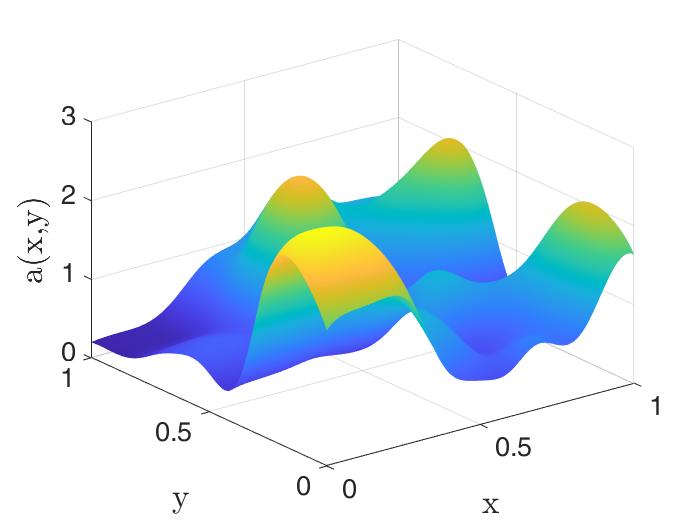}}
	\caption{Sample of the diffusion coefficient $a_{K,A}$ using a Gamma-subordinated Matérn GRF (left), smoothed versions of the coefficient using Gaussian kernel smoothing with smoothness parameter $\nu_s=0.02$ (middle) and $\nu_s = 0.06$ (right).}\label{Fig:GKSmoothingSamples}
\end{figure}

\subsection{MLMC-CV estimator}	
Next, we define the MLMC-CV estimator following \cite{NobileTeseiMLCV}. We fix a positive smoothing parameter $\nu_s>0$. The smoothness parameter $\nu_s$ controls the variance reduction achieved in the MLMC-CV estimator and its specific choice is problem dependent (see Subsection \ref{subsec:NumExMLCV} and \cite{NobileTeseiMLCV}).  We assume $L\in\mathbb{N}$ and consider finite-dimensional subspaces $(\hat{V}_\ell,~\ell=0,\dots,L)$ of $V$ with refinement sizes $\hat{h}_0>\dots>\hat{h}_L>0$ and approximation parameters $\varepsilon_{W,0}>\dots>\varepsilon_{W,L}$ for the GRFs and $\varepsilon_{l,0}>\dots>\varepsilon_{l,L}$ for the Lévy subordinators (see Subsection \ref{subsec:SampleAdaptedFE}). To unify notation, we focus here again on the sample-adapted discretization with corresponding approximation $\hat{u}_{K,A,\ell}^{(\varepsilon_W,\varepsilon_l)}$ with averaged refinement parameter $\mathbb{E}(\hat{h}_\ell^{2\hat{\kappa}_a})^{1/2}$ and convergence rate $\hat{\kappa}_a$ for the theoretical analysis of the estimator (see Assumption \ref{ASS:ConvRateFEMMeanSquareError} and Section \ref{sec:MLMC}) and point out again that similar results hold for the non-adapted FE approach. Since we again fix the parameters $K$ and $A$ in this analysis, we omit them in the following and use the notation $\hat{u}_{\varepsilon_{W,\ell},\varepsilon_{l,\ell},\ell}:=\hat{u}_{K,A,\ell}^{(\varepsilon_{W,\ell},\varepsilon_{l,\ell})}$ for the FEM approximation on $\hat{V}_\ell$, for $\ell=0,\dots,L$. Similar, we denote by $\hat{u}_{\varepsilon_{W,\ell},\varepsilon_{l,\ell},\ell}^{(\nu_s)} := \hat{u}_{K,A,\ell}^{(\nu_s,\varepsilon_{W,\ell},\varepsilon_{l,\ell})}$, for $\ell=0,\dots,L$, the (pathwise) solution to problem \eqref{EQ:EllProblem} - \eqref{EQ:EllProblemBCN} with diffusion coefficient $a_{K,A}^{(\nu_s,\varepsilon_{W,\ell},\varepsilon_{l,\ell})}$ as the smoothed version of the coefficient $a_{K,A}^{(\varepsilon_{W,\ell},\varepsilon_{l,\ell})}$ constructed in \eqref{EQ:DiffCoeffApprDef}. We define the CV basis experiment by
 \begin{align}\label{EQ:BasisExCV}
 \hat{u}_{\varepsilon_{W,\ell},\varepsilon_{l,\ell},\ell}^{CV} :=  \hat{u}_{\varepsilon_{W,\ell},\varepsilon_{l,\ell},\ell} - ( \hat{u}_{\varepsilon_{W,\ell},\varepsilon_{l,\ell},\ell}^{(\nu_s)} - \mathbb{E}(u_{K,A}^{(\nu_s)})),\text{ for }\ell = 0,\dots,L,
 \end{align}
 and we set $\hat{u}_{\varepsilon_{W,-1},\varepsilon_{l,-1},-1}^{CV} = 0$.
 For the moment, we assume that the expectation $\mathbb{E}(u_{K,A}^{(\nu_s)})$ of the solution to the smoothed problem is known. Later, we elaborate more on appropriate approximations of this expectation (see Remark \ref{rem:ApprCVMean}). The MLMC-CV estimator  for the estimation of the expectation of the solution is then defined by
 \begin{align*}
 E^{CV,L}(\hat{u}_{\varepsilon_{W,L},\varepsilon_{l,L},L}^{CV})&:= \sum_{\ell=0}^L E_{M_\ell} (\hat{u}_{\varepsilon_{W,\ell},\varepsilon_{l,\ell},\ell}^{CV} - \hat{u}_{\varepsilon_{W,\ell-1},\varepsilon_{l,\ell-1},\ell-1}^{CV}),
 \end{align*}
 with sample sizes $M_\ell\in\mathbb{N}$ for $\ell = 0,\dots,L$. \\
 
\subsection{Convergence of the MLMC-CV estimator} 
 
 For the theoretical investigation of the MLMC-CV estimator we extend Assumption \ref{ASS:ConvRateFEMMeanSquareError} by the following assumption on the mean-square convergence rate of the pathwise FE method for the smoothed problem.
 
	\begin{assumption}\label{ASS:ConvRateFEMMeanSquareErrorSmoothedProblem}
	There exist deterministic constants $\hat{C}_{u,s}, C_{u,s}$ such that for any $\varepsilon_W,\varepsilon_l>0$ and any $\ell\in\mathbb{N}_0$, the FE approximation errors of $\hat{u}_{K,A,\ell}^{(\nu_s,\varepsilon_W,\varepsilon_l)}\approx u_{K,A}^{(\nu_s,\varepsilon_W,\varepsilon_l)}$ in the subspaces $\hat{V}_\ell$, respectively $u_{K,A,\ell}^{(\nu_s,\varepsilon_W,\varepsilon_l)}\approx u_{K,A}^{(\nu_s,\varepsilon_W,\varepsilon_l)}$ in $V_\ell$, are bounded by
	\begin{align*}
	\|u_{K,A}^{(\nu_s,\varepsilon_W,\varepsilon_l)} - \hat{u}_{K,A,\ell}^{(\nu_s,\varepsilon_W,\varepsilon_l)}\|_{L^2(\Omega;V)}&\leq \hat{C}_{u,s} \mathbb{E}(\hat{h}_\ell^{2})^{1/2} , \text{ respectively,}\\
	\|u_{K,A}^{(\nu_s,\varepsilon_W,\varepsilon_l)} - u_{K,A,\ell}^{(\nu_s,\varepsilon_W,\varepsilon_l)}\|_{L^2(\Omega;V)}&\leq C_{u,s}h_\ell,
	\end{align*}
	where the constants $\hat{C}_{u,s},C_{u,s}$ may depend on $a,f,g,K,A$ but are independent of $\hat{h}_\ell$ and $h_\ell$. Further, we assume that Assumption \ref{ASS:IntegrabilityOfSolGradient} also holds for the solution $u_{K,A}^{(\nu_s)}$ corresponding to the elliptic PDE with the smoothed coefficient $a_{K,A}^{(\nu_s)}$.
\end{assumption}	
	 Note that this assumption is natural since we expect (pathwise) full order convergence of the linear FE method for the smoothed elliptic PDE (see also \cite{MLMCForStochEllMultiscalePDEs}, \cite{MLMCFEMEllPDE}, \cite{AStudyOfElliptic}, \cite{NobileTeseiMLCV} and \cite[Section 8.5]{EllipticDifferentialEquations} together with \cite[Section 6.3]{PartialDifferentialEquations}). The assumption on the integrability of the gradient of the solution corresponding to the smoothed problem is also natural under Assumption \ref{ASS:IntegrabilityOfSolGradient}, since the solution has a higher regularity than the solution $u_{K,A}$ to the elliptic problem with the jump diffusion coefficient $a_{K,A}$. 	 The following lemma states that the approximation error of the smoothed coefficient can be bounded by the approximation error of the rough diffusion coefficient.
	 
\begin{lemma}\label{LE:ApprDiffCoeffSmoothed}
For $t> 1$ and fixed parameters $\nu_s,K,A>0$ and any $\varepsilon_W,\varepsilon_l>0$ it holds for $\mathbb{P}$-almost every $\omega\in \Omega$ 
\begin{align*}
\|a_{K,A}^{(\nu_s)} - a_{K,A}^{(\nu_s,\varepsilon_W,\varepsilon_l)}\|_{L^t(\mathcal{D})}^t \leq C(t,\nu_s,\mathcal{D})\|a_{K,A} - a_{K,A}^{(\varepsilon_W,\varepsilon_l)}\|_{L^t(\mathcal{D})}^t,
\end{align*}
with a constant $C=C(t,\nu_s,\mathcal{D})$ which depends only on the indicated parameters.
\end{lemma}	 
	 
\begin{proof}
Let $t'>1$ such that $1/t+1/t' = 1$. We calculate using Hölder's inequality and the integrability of the Gaussian kernel $\phi_{\nu_s}$ 
\begin{align*}
&\|a_{K,A}^{(\nu_s)} - a_{K,A}^{(\nu_s,\varepsilon_W,\varepsilon_l)}\|_{L^t(\mathcal{D})}^t \\
&\leq \int_\mathcal{D} \Big(\int_{\mathbb{R}^2} |\phi_{\nu_s}(x',y')(a_{K,A}(x-x',y-y')-a_{K,A}^{(\varepsilon_W,\varepsilon_l)}(x-x',y-y'))|d(x',y')\Big)^t d(x,y)\\
&\leq \int_\mathcal{D}\Big( \int_{\mathbb{R}^2}|a_{K,A}(x-x',y-y') - a_{K,A}^{(\varepsilon_W,\varepsilon_l)}(x-x',y-y')|^t d(x',y')\Big) \\
&\phantom{\int_\mathcal{D}\Big( \int_{\mathbb{R}^2}|a_{K,A}(x-x',y-y') - a_{K,A}^{(\varepsilon_W,\varepsilon_l)}(x-x',y-y')|^t}\times \Big(\int_{\mathbb{R}^2} \phi_{\nu_s}(x',y')^{t'}d(x',y')\Big)^{t/t'}  d(x,y)\\
&\leq  C(t,\nu_s,\mathcal{D})\|a_{K,A}-a_{K,A}^{(\varepsilon_W,\varepsilon_l)}\|_{L^t(\mathcal{D})}^t.
\end{align*}
\end{proof}	 
	 
In order to proof the convergence of the MLMC-CV estimator we need the following error bound on the approximation of the solution of the smoothed problem (cf. Theorem \ref{TH:ErrorBoundE2}).	 
	 \begin{theorem}\label{TH:ErrorBoundE2SmoothProblem}
	Let $r\geq 2$ and $b,c\in[1,+\infty]$ be given such that it holds
	\begin{align*}
	rc\gamma\geq 2 \text{ and }2b\leq rc< \eta 
	\end{align*}
	with a fixed real number $\gamma\in(0,min(1,\beta/(2\alpha))$. Here, the parameters $\eta, \alpha$ and $\beta$ are determined by the GRFs $W_1$, $W_2$ and the L\'evy subordinators $l_1$, $l_2$ (see Assumption \ref{ASS:CutProblemEigenvalues}). \\
	Let $m,n\in[1,+\infty]$ be real numbers such that 
	\begin{align*}
	\frac{1}{m} + \frac{1}{c} = \frac{1}{n} + \frac{1}{b}=1,
	\end{align*}
	and let $k_{reg}\geq 2$ and $j_{reg}>0$ be the regularity specifiers given by Assumption~\ref{ASS:IntegrabilityOfSolGradient}.
	If it holds that
	\begin{align*}
	n<1+\frac{j_{reg}}{2} \text{ and } rm< k_{reg},
	\end{align*}
	then the approximated solution $u_{K,A}^{(\nu_s,\varepsilon_W,\varepsilon_l)}$ of the smoothed problem converges to the solution $u_{K,A}^{(\nu_s)}$ of the truncated smoothed problem for $\varepsilon_W,\varepsilon_l\rightarrow 0$ and it holds
	\begin{align*}
	\|u_{K,A}^{(\nu_s)}-u_{K,A}^{(\nu_s,\varepsilon_W,\varepsilon_l)}\|_{L^r(\Omega;V)}&\leq  C(\overline{a}_-,\mathcal{D},r)C_{reg}\|a_{K,A}^{(\nu_s,\varepsilon_W,\varepsilon_l)}-a_{K,A}^{(\nu_s)}\|_{L^{rc}(\Omega;L^{2b}(\mathcal{D}))}\\
	&\leq  C_{reg}C(\overline{a}_-,\mathcal{D},r,\nu_s)(\varepsilon_W^\gamma + \varepsilon_l^\frac{1}{rc}).
	\end{align*}
	\end{theorem}
	
	 \begin{proof}
	 This theorem follows by the same arguments used in \cite[Theorem 5.9]{SGRFPDE} together with Lemma \ref{LE:ApprDiffCoeffSmoothed}.
	 \end{proof}
	 We are now able to prove the following a-priori bound on the mean-square error of the MLMC-CV estimator, similar to Theorem \ref{TH:MLMCEst}.
 
 \begin{theorem}\label{TH:MLMCCVEst} 
 We set $r=2$ and assume $q>2$. Further, let $b,c\geq 1$ be given such that Theorem \ref{TH:ErrorBoundE2} (and Theorem \ref{TH:ErrorBoundE2SmoothProblem}) hold. For $L\in\mathbb{N}$, let $\hat{h}_\ell>0$, $M_\ell$, $\varepsilon_{W,\ell}>0$ and $\varepsilon_{l,\ell}>0$ be the level-dependent approximation parameters, for $\ell=0,...,L$, such that $\hat{h}_{\ell},~ \varepsilon_{W,\ell},$ and $\varepsilon_{l,\ell}$ decrease with respect to $\ell$. It holds
 \begin{align*}
 \|\mathbb{E}(u_{K,A}) - E^{CV,L}(\hat{u}_{\varepsilon_{W,L},\varepsilon_{l,L},L}^{CV})\|_{L^2(\Omega;V)}\leq C &\Big(\varepsilon_{W,L}^\gamma + \varepsilon_{l,L}^\frac{1}{2c} + \mathbb{E}(\hat{h}_L^{2\hat{\kappa}_a})^{1/2} + \frac{1}{\sqrt{M_0}} \\
 &+\sum_{\ell=0}^{L-1} \frac{\varepsilon_{W,\ell}^\gamma + \varepsilon_{l,\ell}^\frac{1}{2c} + \mathbb{E}(\hat{h}_\ell^{2\hat{\kappa}_a})^{1/2}}{\sqrt{M_{\ell+1}}}\Big),
 \end{align*}
 where $C>0$ is a constant which is independent of $L$ and the level-dependent approximation parameters. Note that the numbers $\gamma>0$ and $c\geq 1$ are determined by the GRFs resp. the subordinators (cf. Theorem \ref{TH:ErrorBoundE2} and Theorem \ref{TH:ErrorBoundE2SmoothProblem}).
 \end{theorem}

\begin{proof}
We split the error by
\begin{align*}
\|\mathbb{E}(u_{K,A}) - E^{CV,L}(\hat{u}_{\varepsilon_{W,L},\varepsilon_{l,L},L}^{CV})\|_{L^2(\Omega;V)}&\leq \|\mathbb{E}(u_{K,A}) - \mathbb{E}(\hat{u}_{\varepsilon_{W,L},\varepsilon_{l,L},L}^{CV}) \|_{L^2(\Omega;V)} \\
 &+ \|\mathbb{E}(\hat{u}_{\varepsilon_{W,L},\varepsilon_{l,L},L}^{CV}) - E^{CV,L}(\hat{u}_{\varepsilon_{W,L},\varepsilon_{l,L},L}^{CV})\|_{L^2(\Omega;V)}\\
 &=:I_1+I_2.
\end{align*}
For the first term we estimate using Theorem \ref{TH:ErrorBoundE2} and Assumption \ref{ASS:ConvRateFEMMeanSquareError} together with Theorem \ref{TH:ErrorBoundE2SmoothProblem} and Assumption \ref{ASS:ConvRateFEMMeanSquareErrorSmoothedProblem} to obtain
 \begin{align*}
 I_1 &\leq \mathbb{E}(\|u_{K,A} - \hat{u}_{\varepsilon_{W,L},\varepsilon_{l,L},L}\|_V) + 
 \mathbb{E}(\|\hat{u}_{\varepsilon_{W,L},\varepsilon_{l,L},L}^{(\nu_s)}- u_{K,A}^{(\nu_s)}\|_V)\\
 &\leq \|u_{K,A} - u_{K,A}^{(\varepsilon_{W,L},\varepsilon_{l,L})}\|_{L^2(\Omega;V)} + \|u_{K,A}^{(\varepsilon_{W,L},\varepsilon_{l,L})} - \hat{u}_{\varepsilon_{W,L},\varepsilon_{l,L},L}\|_{L^2(\Omega;V)}\\
 & + \|u_{K,A}^{(\nu_s)} - u_{K,A}^{(\nu_s,\varepsilon_{W,L},\varepsilon_{l,L})}\|_{L^2(\Omega;V)} + \|u_{K,A}^{(\nu_s,\varepsilon_{W,L},\varepsilon_{l,L})} - \hat{u}_{\varepsilon_{W,L},\varepsilon_{l,L},L}^{(\nu_s)}\|_{L^2(\Omega;V)}\\
 &\leq C_{reg}C(\overline{a}_-,\mathcal{D},\nu_s)(\varepsilon_{W,L}^\gamma + \varepsilon_{l,L}^\frac{1}{2c}) + \hat{C}_{u}\mathbb{E}(\hat{h}_L^{2\hat{\kappa}_a})^\frac{1}{2} +  \hat{C}_{u,s} \mathbb{E}(\hat{h}_L^{2})^{1/2}\\
 & \leq C_{reg}C(\overline{a}_-,\mathcal{D},\nu_s)(\varepsilon_{W,L}^\gamma + \varepsilon_{l,L}^\frac{1}{2c}) + \tilde{C}\mathbb{E}(\hat{h}_L^{2\hat{\kappa}_a})^\frac{1}{2}.
 \end{align*}
 For the second term we use the definition of the MLMC-CV estimator $E^{CV,L}$ and Lemma \ref{LE:SLMCstandardEst} to estimate
 
 \begin{align*}
  I_2&\leq \sum_{\ell=0}^{L} \|\mathbb{E}(\hat{u}_{\varepsilon_{W,\ell},\varepsilon_{l,\ell},\ell}^{CV} - \hat{u}_{\varepsilon_{W,\ell-1},\varepsilon_{l,\ell-1},\ell-1}^{CV}) - E_{M_\ell}(\hat{u}_{\varepsilon_{W,\ell},\varepsilon_{l,\ell},\ell}^{CV}-\hat{u}_{\varepsilon_{W,\ell-1},\varepsilon_{l,\ell-1},\ell-1}^{CV})\|_{L^2(\Omega;V)}\\
 &= \sum_{\ell=0}^L \Big(\frac{Var(\hat{u}_{\varepsilon_{W,\ell},\varepsilon_{l,\ell},\ell}^{CV}-\hat{u}_{\varepsilon_{W,\ell-1},\varepsilon_{l,\ell-1},\ell-1}^{CV})}{M_\ell}\Big) ^{1/2}\\
  &\leq \sum_{\ell=0}^L \frac{1}{\sqrt{M_\ell}}\|\hat{u}_{\varepsilon_{W,\ell},\varepsilon_{l,\ell},\ell}^{CV}-\hat{u}_{\varepsilon_{W,\ell-1},\varepsilon_{l,\ell-1},\ell-1}^{CV}\|_{L^2(\Omega;V)}\\
 &\leq\sum_{\ell=0}^L \frac{1}{\sqrt{M_\ell}}\Big( \|\hat{u}_{\varepsilon_{W,\ell},\varepsilon_{l,\ell},\ell}^{CV} - u_{K,A} +(u_{K,A}^{(\nu_s)} - \mathbb{E}(u_{K,A}^{(\nu_s)}))\|_{L^2(\Omega;V)} \\
 &\phantom{\leq\sum_{\ell=0}^L \frac{1}{\sqrt{M_\ell}}\Big( \|\hat{u}_{\varepsilon_{W,\ell},\varepsilon_{l,\ell},\ell}^{CV} - u_{K,A} +b}+ \|u_{K,A} -(u_{K,A}^{(\nu_s)} - \mathbb{E}(u_{K,A}^{(\nu_s)}))- \hat{u}_{\varepsilon_{W,\ell-1},\varepsilon_{l,\ell-1},\ell-1}^{CV}\|_{L^2(\Omega;V)}\Big).
 \end{align*}

\noindent We estimate each term in this summand with the same strategy as we did for the term $I_1$ using Theorem \ref{TH:ErrorBoundE2} and Assumption \ref{ASS:ConvRateFEMMeanSquareError} together with Theorem \ref{TH:ErrorBoundE2SmoothProblem} and Assumption \ref{ASS:ConvRateFEMMeanSquareErrorSmoothedProblem} to obtain

\begin{align*}
\|\hat{u}_{\varepsilon_{W,\ell},\varepsilon_{l,\ell},\ell}^{CV} - u_{K,A} &+(u_{K,A}^{(\nu_s)} - \mathbb{E}(u_{K,A}^{(\nu_s)}))\|_{L^2(\Omega;V)}\\
& \leq \| \hat{u}_{\varepsilon_{W,\ell},\varepsilon_{l,\ell},\ell} - u_{K,A}\|_{L^2(\Omega,V)} + \|u_{K,A}^{(\nu_s)} - \hat{u}_{\varepsilon_{W,\ell},\varepsilon_{l,\ell},\ell}^{(\nu_s)}\|_{L^2(\Omega;V)}\\
& \leq C_{reg}C(\overline{a}_-,\mathcal{D},\nu_s)(\varepsilon_{W,\ell}^\gamma + \varepsilon_{l,\ell}^\frac{1}{2c}) + \tilde{C}\mathbb{E}(\hat{h}_\ell^{2\hat{\kappa}_a})^\frac{1}{2},
\end{align*}
for $\ell = 0,\dots,L$ and for $\ell = -1$ we get by Theorem \ref{TH:ExistenceOfSolutionSubord} and Equation \eqref{EQ:SolBoundSmoothProblem}
\begin{align*}
\|u_{K,A} -(u_{K,A}^{(\nu_s)} - \mathbb{E}(u_{K,A}^{(\nu_s)}))\|_{L^2(\Omega;V)}&\leq \|u_{K,A}\|_{L^2(\Omega;V)}  + \,Var(u_{K,A}^{(\nu_s)})^\frac{1}{2} \\
& \leq C(\overline{a}_-,\mathcal{D},\nu_s)(\|f\|_{L^q(\Omega;H)} + \|g\|_{L^2(\Omega;L^2(\Gamma_2))}).
\end{align*}
Together, we obtain
\begin{align*}
I_1+I_2& \leq C_{reg}C(\overline{a}_-,\mathcal{D},\nu_s)\Big( \varepsilon_{W,L}^\gamma + \varepsilon_{l,L}^\frac{1}{2c} + \frac{\varepsilon_{W,0}^\gamma + \varepsilon_{l,0}^\frac{1}{2c}}{\sqrt{M_0}} + \sum_{\ell=1}^L \frac{1}{\sqrt{M_\ell}}( \varepsilon_{W,\ell}^\gamma + \varepsilon_{l,\ell}^\frac{1}{2c} + \varepsilon_{W,\ell-1}^\gamma + \varepsilon_{l,\ell-1}^\frac{1}{2c}) \Big)\\
  &+ \tilde{C} \Big( \mathbb{E}(\hat{h}_L^{2\hat{\kappa}_a})^\frac{1}{2} + \frac{\mathbb{E}(\hat{h}_0^{2\hat{\kappa}_a})^{1/2}}{\sqrt{M_0}} +  \sum_{\ell=1}^L \frac{1}{\sqrt{M_\ell}} (\mathbb{E}(\hat{h}_\ell^{2\hat{\kappa}_a})^\frac{1}{2} +\mathbb{E}(\hat{h}_{\ell-1}^{2\hat{\kappa}_a})^\frac{1}{2})\Big)\\
  &+\frac{1}{\sqrt{M_0}} C(\overline{a}_-,\mathcal{D},\nu_s)(\|f\|_{L^q(\Omega;H)} + \|g\|_{L^q(\Omega;L^2(\Gamma_2))})\\
  &\leq C\Big(\varepsilon_{W,L}^\gamma + \varepsilon_{l,L}^\frac{1}{2c} + \mathbb{E}(\hat{h}_L^{2\hat{\kappa}_a})^{1/2} + \frac{1}{\sqrt{M_0}} +\sum_{\ell=0}^{L-1} \frac{\varepsilon_{W,\ell}^\gamma + \varepsilon_{l,\ell}^\frac{1}{2c} + \mathbb{E}(\hat{h}_\ell^{2\hat{\kappa}_a})^{1/2}}{\sqrt{M_{\ell+1}}}\Big),
\end{align*}
where we used monotonicity of $(\varepsilon_{W,\ell})_{\ell=0}^L$, $(\varepsilon_{l,\ell})_{\ell=0}^L$ and $(\hat{h}_\ell)_{\ell=0}^L$ in the last step.
\end{proof}
As it is the case for the a-priori error bound for the MLMC estimator (see Theorem \ref{TH:MLMCEst}), Theorem \ref{TH:MLMCCVEst} allows for an equilibration of all error contributions resulting from the approximation of the diffusion coefficient and the approximation of the pathwise solution by the FE method, which can be seen by the following corollary.

 \begin{corollary}\label{COR:ParameterChoiceMLMCCV}
 Let the assumptions of Theorem \ref{TH:MLMCCVEst} hold. For $L\in\mathbb{N}$ and given (stochastic) refinement parameters $\hat{h}_0>\dots>\hat{h}_L>0$ choose $\varepsilon_{W,\ell}>0$ and $\varepsilon_{l,\ell}>0$ such that 
\begin{align*}
\varepsilon_{W,\ell} \simeq \mathbb{E}(\hat{h}_\ell^{2\hat{\kappa}_a})^{1/(2\gamma)} \text{ and } \varepsilon_{l,\ell} \simeq \mathbb{E}(\hat{h}_\ell^{2\hat{\kappa}_a})^{c}.
\end{align*} 
 and sample numbers $M_\ell\in\mathbb{N}$ such that for some positive parameter $\xi>0$ it holds
 \begin{align*} 
 M_0\simeq \mathbb{E}(\hat{h}_L^{2\hat{\kappa}	_a})^{-1} \text{ and } M_\ell\simeq \mathbb{E}(\hat{h}_L^{2\hat{\kappa}_a})^{-1}\mathbb{E}(\hat{h}_{\ell-1}^{2\hat{\kappa}_a})(\ell+1)^{2(1+\xi)}\text{ for }\ell=1,\dots,L.
 \end{align*}

 Then, it holds
 \begin{align*}
 \|\mathbb{E}(u_{K,A}) - E^{CV,L}(\hat{u}_{\varepsilon_{W,L},\varepsilon_{l,L},L}^{CV})\|_{L^2(\Omega;V)}=\mathcal{O}(\mathbb{E}(\hat{h}_L^{2\hat{\kappa}_a})^{1/2}).
 \end{align*}
 \end{corollary}
 
 \begin{proof}
 See Corollary \ref{COR:ParameterChoiceMLMC}.
 \end{proof}

We want to emphasize that Theorem \ref{TH:MLMCCVEst} and Corollary \ref{COR:ParameterChoiceMLMCCV} imply the same asymptotical convergence of the MLMC-CV estimator as the MLMC estimator which has been considered in Section \ref{sec:MLMC}. However, it is to be expected that the MLMC-CV estimator is more efficient due to the samplewise correction by the Control Variate and the resulting variance reduction on the different levels. We close this section with a remark on how to compute the mean of the Control Variate.

\begin{rem}\label{rem:ApprCVMean}
Unlike we assumed the CV mean $\mathbb{E}(u_{K,A}^{(\nu_s)})$ is in general unknown for fixed parameters $K,A,\nu_s>0$. Corollary \ref{COR:ParameterChoiceMLMCCV} yields that it is sufficient to approximate the CV mean with any estimator which is convergent with order $\mathbb{E}(h_L^{2\hat{\kappa}_a})^{1/2}$. In fact, we denote by 
\begin{align*}
Est_{CV}^{L}(u_{K,A}^{(\nu_s)})\approx \mathbb{E}(u_{K,A}^{(\nu_s)}),
\end{align*}
the realization of the desired estimator and we assume the existence of a constant $C_{CV}>0$ such that it holds
\begin{align*}
\|Est_{CV}^L(u_{K,A}^{(\nu_s)})- \mathbb{E}(u_{K,A}^{(\nu_s)})\|_{L^2(\Omega;V)}\leq C_{CV}\mathbb{E}(\hat{h}_L^{2\hat{\kappa}_a})^{1/2},
\end{align*}
in the notation of Theorem \ref{TH:MLMCCVEst}. Further, instead of the basis experiment $\hat{u}_{\varepsilon_{W,\ell},\varepsilon_{l,\ell},\ell}^{CV}$ from \eqref{EQ:BasisExCV} we consider
\begin{align*}
 \tilde{u}_{\varepsilon_{W,\ell},\varepsilon_{l,\ell},\ell}^{CV} :=  \hat{u}_{\varepsilon_{W,\ell},\varepsilon_{l,\ell},\ell} -( \hat{u}_{\varepsilon_{W,\ell},\varepsilon_{l,\ell},\ell}^{(\nu_s)} - Est_{CV}^L(u_{K,A}^{(\nu_s)})),\text{ for }\ell = 0,\dots,L,
 \end{align*}
 and we set $\tilde{u}_{\varepsilon_{W,-1},\varepsilon_{l,-1},-1}^{CV} = 0$ and denote the corresponding MLMC-CV estimator by \newline $E^{CV,L}(\tilde{u}_{\varepsilon_{W,\ell},\varepsilon_{l,\ell},\ell}^{CV})$. Then, by Corollary \ref{COR:ParameterChoiceMLMCCV}, it holds
 \begin{align*}
  &\|\mathbb{E}(u_{K,A}) - E^{CV,L}(\tilde{u}_{\varepsilon_{W,L},\varepsilon_{l,L},L}^{CV})\|_{L^2(\Omega;V)}\\
  &\leq  \|\mathbb{E}(u_{K,A}) - E^{CV,L}(\hat{u}_{\varepsilon_{W,L},\varepsilon_{l,L},L}^{CV})\|_{L^2(\Omega;V)} + \|Est_{CV}^L(u_{K,A}^{(\nu_s)})- \mathbb{E}(u_{K,A}^{(\nu_s)})\|_{L^2(\Omega;V)}\\
  &=\mathcal{O}(\mathbb{E}(\hat{h}_L^{2\hat{\kappa}_a})^{1/2}).
 \end{align*}
 For example, the CV mean could be estimated by another MLMC estimator on the level $L$ where the parameters are choosen according to Corollary \ref{COR:ParameterChoiceMLMCCV}.
\end{rem}

	
\section{Numerical examples}\label{sec:numerics}
 In the following section we present numerical examples for the estimation of the mean of the solution to the elliptic PDE \eqref{EQ:EllProblemCutUpperCut} - \eqref{EQ:DiffCoeffDefiCutUpperCut}. We perform convergence tests with the proposed multilevel Monte Carlo estimators defined in Section \ref{sec:MLMC} and Section \ref{sec:MLMCCV}. In our numerical examples, we consider different levels $L\in\mathbb{N}$ and choose the sample numbers $(M_\ell,~\ell=0,\dots,L)$ and the level dependent approximation parameters for the GRFs $(\varepsilon_{W,\ell},~\ell=0,\dots,L)$ and the subordinators $(\varepsilon_{l,\ell},~\ell=0,\dots,L)$ according to Corollary \ref{COR:ParameterChoiceMLMC}  resp. Corollary \ref{COR:ParameterChoiceMLMCCV} if nothing else is explicitly mentioned. Our numerical examples aim to compare the performance of the MLMC estimator with non-adapted triangulations with the MLMC estimator which uses sample-adapted triangulations. Further, we compare the performance of the standard MLMC estimator with the MLMC-CV estimator for high-intensity subordinators where the sample-adapted triangulations are not feasible anymore. All our numerical experiments are performed in MATLAB R2021a on a workstation with 16 GB memory and Intel quadcore processor with 3.4 GHz.

\subsection{PDE parameters}\label{subsec:PDEParameters}

In our numerical examples we consider the domain $\mathcal{D}=(0,1)^2$ and choose $\overline{a}\equiv 1/10$, $f\equiv 10$, $\Phi_1=1/100\,\exp(\cdot)$ and $\Phi_2=5\,|\cdot|$ for the diffusion coefficient in \eqref{EQ:DiffCoeffDefiCutUpperCut} if nothing else is explicitly mentioned. Further, we impose the following mixed Dirichlet-Neumann boundary conditions: we split the domain boundary $\partial\mathcal{D}$ by $\Gamma_1=\{0,1\}\times[0,1]$ and $\Gamma_2=(0,1)\times\{0,1\}$ and impose the pathwise mixed Dirichlet-Neumann boundary conditions
\begin{align*}
u_{K,A}(\omega,\cdot)=\begin{cases} 0.1  & ~on~ \{0\}\times [0,1] \\ 0.3 &~on~\{1\}\times [0,1]  \end{cases} \text{ and } a_{K,A}\overrightarrow{n}\cdot\nabla u_{K,A}=0 \text{ on } \Gamma_2,
\end{align*}
for $\omega\in\Omega$.
 We use a reference grid with $401\times 401$ equally spaced points on the domain $\mathcal{D}$ for interpolation and prolongation. The GRFs $W_1$ and $W_2$ are set to be a Mat\'ern-1.5-GRFs on $\mathcal{D}$ (resp. on $[0,K]^2$) with varying correlation lengths and variance parameters. Note that for Mat\'ern-1.5-GRFs we can expect $\gamma=1$ in Theorem \ref{TH:ErrorBoundE2} (see~ \cite[Section 7]{SGRFPDE}, \cite[Chapter 5]{AMultilevelMonteCarloAlgorithmForParabolicAdvectionDiffusionProblemsWithDiscontinuousCoefficients},~\cite{FiniteElementErrorAnalysisOfEllipticPDEsWIthRandomCoefficients}).
We simulate the GRFs $W_1$ and $W_2$ by the circulant embedding method (see~ \cite{AnalysisOfCirculantEmbeddingMethodsForSamplingStationaryRandomFields} and \cite{CirculantEmbeddingWithWMCAnalysisForEllipicPDEWithLognormalCoefficients}) to obtain approximations $W_1^{(\varepsilon_W)}\approx W_1$ and $W_2^{(\varepsilon_W)}\approx W_2$ as described in Section \ref{sec:apprOfGRFAndSubord}. In the experiments, we choose the diffusion cut-off $A$ in \eqref{EQ:DiffCoeffDefiCutUpperCut} large enough such that it has no influence on the numerical experiments for our choice of the GRFs, e.g. $A=100$ and choose the cut-off level $K$ for each experiment individually depending on the specific choice of the subordinator.

\subsection{Numerical examples for the MLMC estimator}\label{subsec:NumExMLMC}
In this section we conduct experiments with the MLMC estimator introduced in Section \ref{sec:MLMC}. We consider subordinators with different intensity and GRFs with varying correlation lengths in order to cover problems with different solution regularity. The comparatively low intensity of the subordinators used in this section (see also Subsection \ref{subsec:NumExMLCV}) allows the application of the pathwise sample-adapted approach introduced in Subsection \ref{subsec:SampleAdaptedFE} which can then be compared with the performance of the MLMC estimator with standard triangulations. During this section, we refer to these approaches with \textit{adapted FEM MLMC} and \textit{non-adapted FEM MLMC}. In our experiments, we use Poisson processes to subordinate the GRF $W_2$ in the diffusion coefficient in \eqref{EQ:DiffCoeffDefiCutUpperCut}. We consider both, Poisson processes with high and low intensity parameter leading to a different number of jumps in the diffusion coefficient. For the simulation of the Poisson processes we have two options: the processes may be approximated under Assumption \ref{ASS:CutProblemEigenvalues} \textit{v} but they may also be simulated exactly (see Subsection \ref{subsubsec:TheTwoApprMethods}). Hence, using Poisson subordinators allows for a detailed investigation of the approximation error caused by the approximation of the L\'evy subordinators $l_1$ and $l_2$. This will be explained briefly in the following subsection (see also \cite[Section 7.3.1]{SGRFPDE}).
\subsubsection{The two approximation methods}
\label{subsubsec:TheTwoApprMethods}
We simulate the Poisson processes by two conceptional different approaches: the first approach is an exact and grid-independent simulation of a Poisson process using the \textit{Uniform Method} (see \cite[Section 8.1.2]{LevyProcessesInFinance}). On the other hand, we may simulate approximations of the Poisson processes satisfying Assumption \ref{ASS:CutProblemEigenvalues} \textit{v} in the following way (see \cite[Section 7.3.1]{SGRFPDE}):
We sample values of the Poisson($\lambda$)-processes $l_1$ and $l_2$ on an equidistant grid $\{x_i,~i=0,...,N_l\}$ with $x_0=0$ and $x_{N_l}=1$ and step size $|x_{i+1}-x_i|\leq \varepsilon_l\leq 1$ for all $i=0,\dots,N_l-1$ and approximate the stochastic processes by a piecewise constant extension $l_j^{(\varepsilon_l)}\approx l_j$ of the values on the grid:
\begin{align*}
l_j^{(\varepsilon_l)}(x)=\begin{cases} l_j(x_i) & x\in[x_i,x_{i+1}) \text{ for } i=0,...,N_l-1,  \\ l_j(x_{N_l-1}) & x=1.  \end{cases}
\end{align*}
for $j=1,2$. Since the Poisson process has independent, Poisson distributed increments, values of the Poisson process at the discrete points $\{x_i,~i=0,\dots,N_l\}$ may be generated by adding independent Poisson distributed random variables with appropriately scaled intensity parameters. For the rest of this paper, we refer to this approach as the \textit{approximation approach} to simulate a Poisson process. Comparing the results of the MLMC experiments using the two described approaches for the simulation of the Poisson processes allows conclusions to be drawn on the numerical influence of an additional approximation of the subordinator (see Subsection \ref{subsubsec:Poisson1Subord}). This is further important especially for situations in which the choice of the subordinators does not allow for an exact simulation of the process.\\
 Note that Poisson processes satisfy Assumption \ref{ASS:CutProblemEigenvalues} \textit{v} with $\eta =+\infty$ (see \cite[Section 7.3.1]{SGRFPDE}). Since $\gamma=1$ (see Subection \ref{subsec:PDEParameters}), $\eta=+\infty$ and $f\in L^q(\Omega;H)$ for every $q\geq 1$ we choose for any positive $\delta>0$
\begin{align*}
r=2, ~c=b=1+\delta 
\end{align*}
to obtain from Theorem \ref{TH:ErrorBoundE2}
\begin{align*}
\|u_{K,A}-u_{K,A}^{(\varepsilon_W,\varepsilon_l)}\|_{L^2(\Omega;V)}\leq C_{reg}C(\overline{a}_-,\mathcal{D}) (\varepsilon_W + \varepsilon_l^ \frac{1}{2c}),
\end{align*}
where we have to assume that $j_{reg}> 2((1+\delta)/\delta-1)$ and $k_{reg}> 2(1+\delta)/\delta$ for the regularity constants $j_{reg},k_{reg}$ given in Assumption \ref{ASS:IntegrabilityOfSolGradient}. For $\delta=0.5$ we obtain 
\begin{align*}
\|u_{K,A}-u_{K,A}^{(\varepsilon_W,\varepsilon_l)}\|_{L^2(\Omega;V)}\leq C_{reg}C(\overline{a}_-,\mathcal{D}) (\varepsilon_W + \varepsilon_l^ \frac{1}{3}).
\end{align*}
Therefore, we get $\gamma=1$ and $c=1.5$ in the equilibration formula \eqref{EQ:RuleApprParams} for the numerical examples with the Poisson subordinators.

\subsubsection{Poisson($1$) subordinators} \label{subsubsec:Poisson1Subord}

In our first numerical example we use Poisson($1$) - subordinators. With this choice, we get on average one jump in each direction of the diffusion coefficient. The standard deviation and the correlation parameters for the GRF $W_1$ (resp. $W_2$) are set to be $\sigma_1^2= 1.5^2$ and $r_1=0.5$ (resp. $\sigma_2^2=0.1^2$ and $r_2=0.5$). Figure \ref{fig:poisson1_samples} shows samples of the diffusion coefficient and the corresponding PDE solution.

	\begin{figure}[ht]
	\centering
	\subfigure{\includegraphics[scale=0.14]{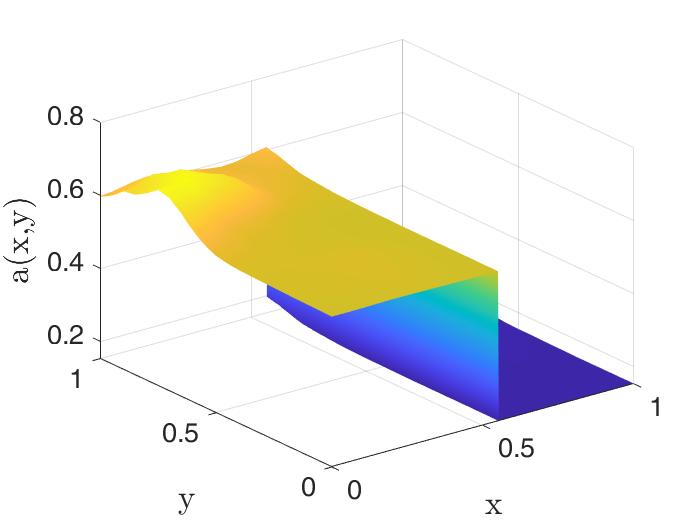}}
	\subfigure{\includegraphics[scale=0.14]{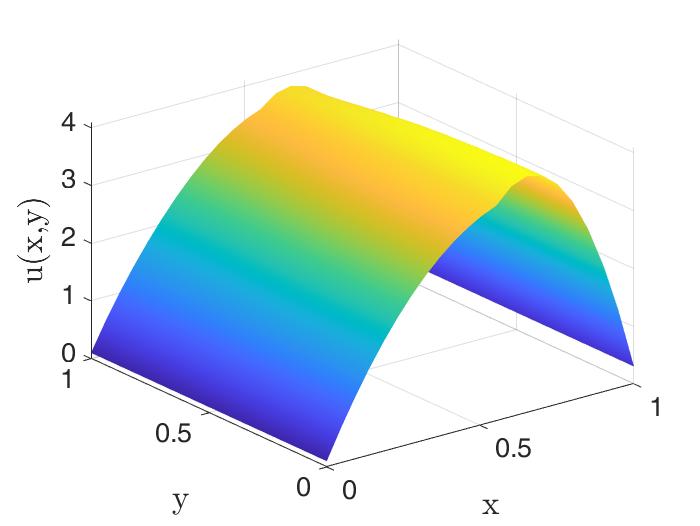}}
	\subfigure{\includegraphics[scale=0.14]{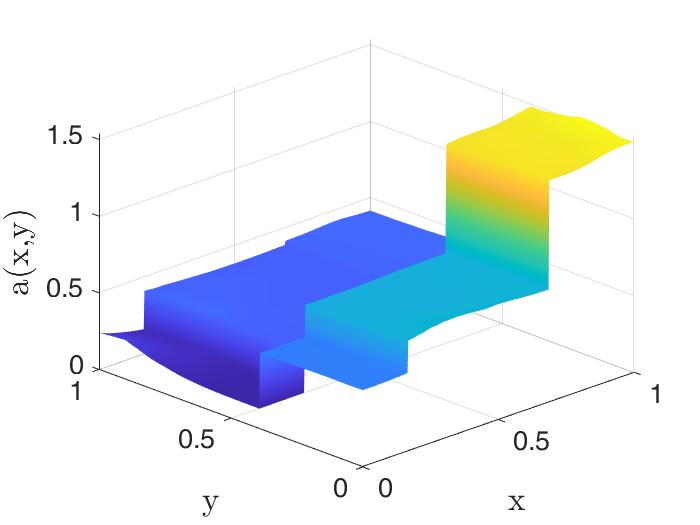}}
	\subfigure{\includegraphics[scale=0.14]{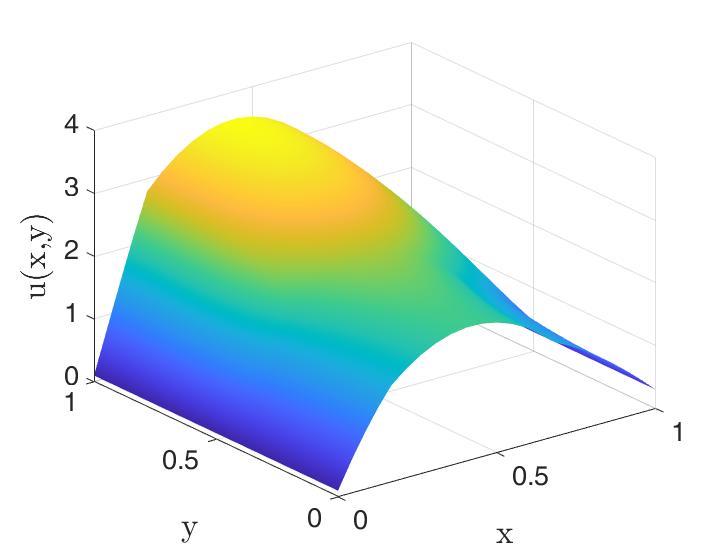}}
    \caption{Different samples of the diffusion coefficient with Poisson($1$)-subordinators and the corresponding PDE solutions with mixed Dirichlet-Neumann boundary conditions.}\label{fig:poisson1_samples}
    \end{figure}
\noindent The cut-off threshold $K$ for the subordinators in \eqref{EQ:DiffCoeffDefiCutUpperCut} is choosen to be $K=8$. With this choice we obtain
\begin{align*}
\mathbb{P}(\underset{t\in[0,1]}{\sup}\, l_j(t)\geq K)=\mathbb{P}(l_j(1)\geq K)\approx 1.1252e^{-06},
\end{align*}
for $j=1,2$, such that this cut-off has a negligible influence in the numerical example. We compute the RMSE $ \|\mathbb{E}(u_{K,A}) - E^L(\hat{u}_{\varepsilon_{W,L},\varepsilon_{l,L},L})\|_{L^2(\Omega;V)}$ for the sample-adapted and the non-adapted approach using 10 independent runs of the MLMC estimator on the levels $L=1,\dots,5$, where we set $\overline{h}_\ell=h_\ell=0.3\cdot 1.7^{-(\ell-1)}$, for $\ell = 1,\dots,5$. Further, we use a reference solution computed on level $7$ with singlelevel Monte Carlo. We run this experiment with both approaches for the simulation of the subordinators introduced in Subsection \ref{subsubsec:TheTwoApprMethods}: the approximation approach and the Uniform Method.
	\begin{figure}[ht]
	\centering
	\subfigure{\includegraphics[scale=0.45]{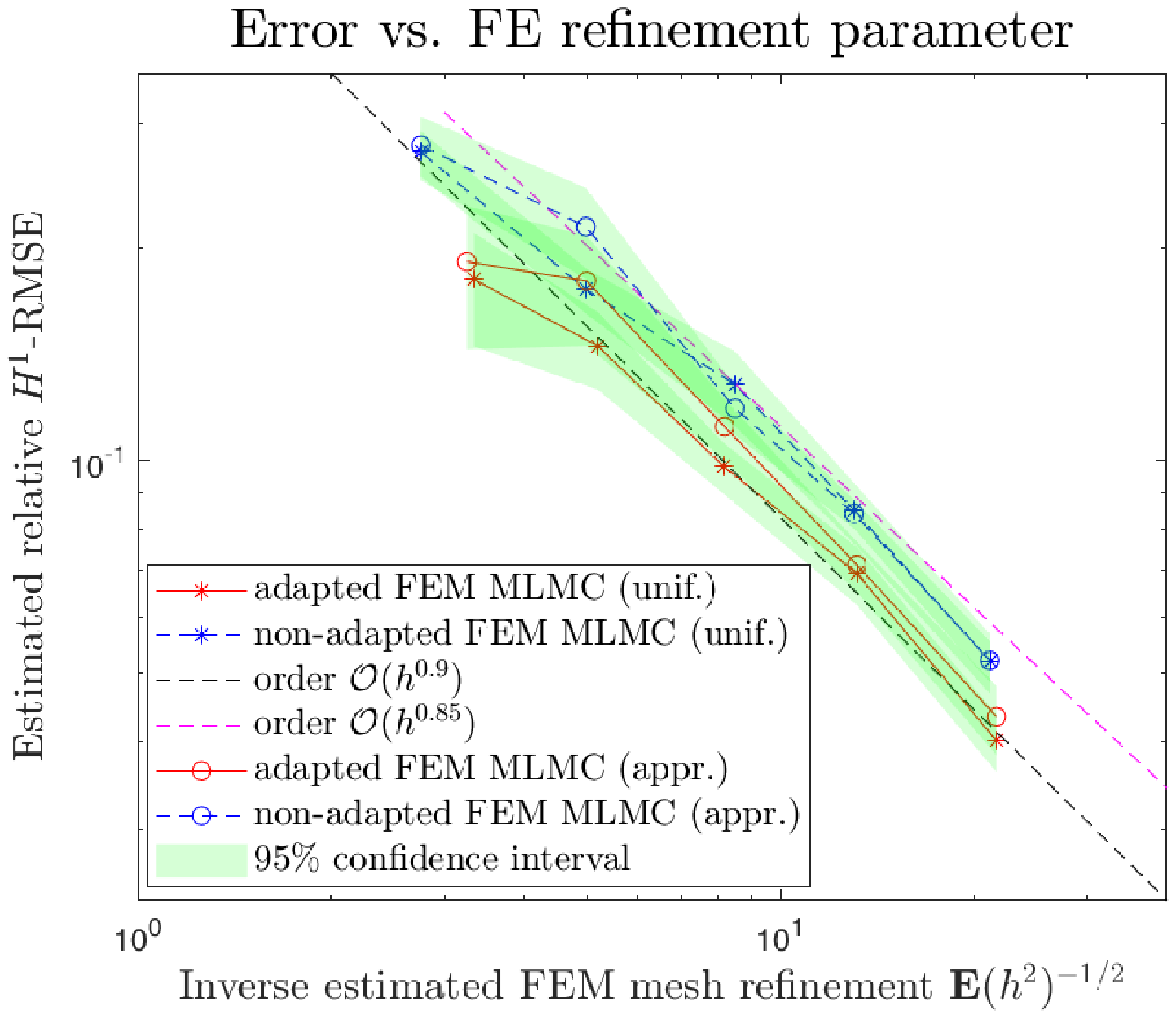}}
	\subfigure{\includegraphics[scale=0.49]{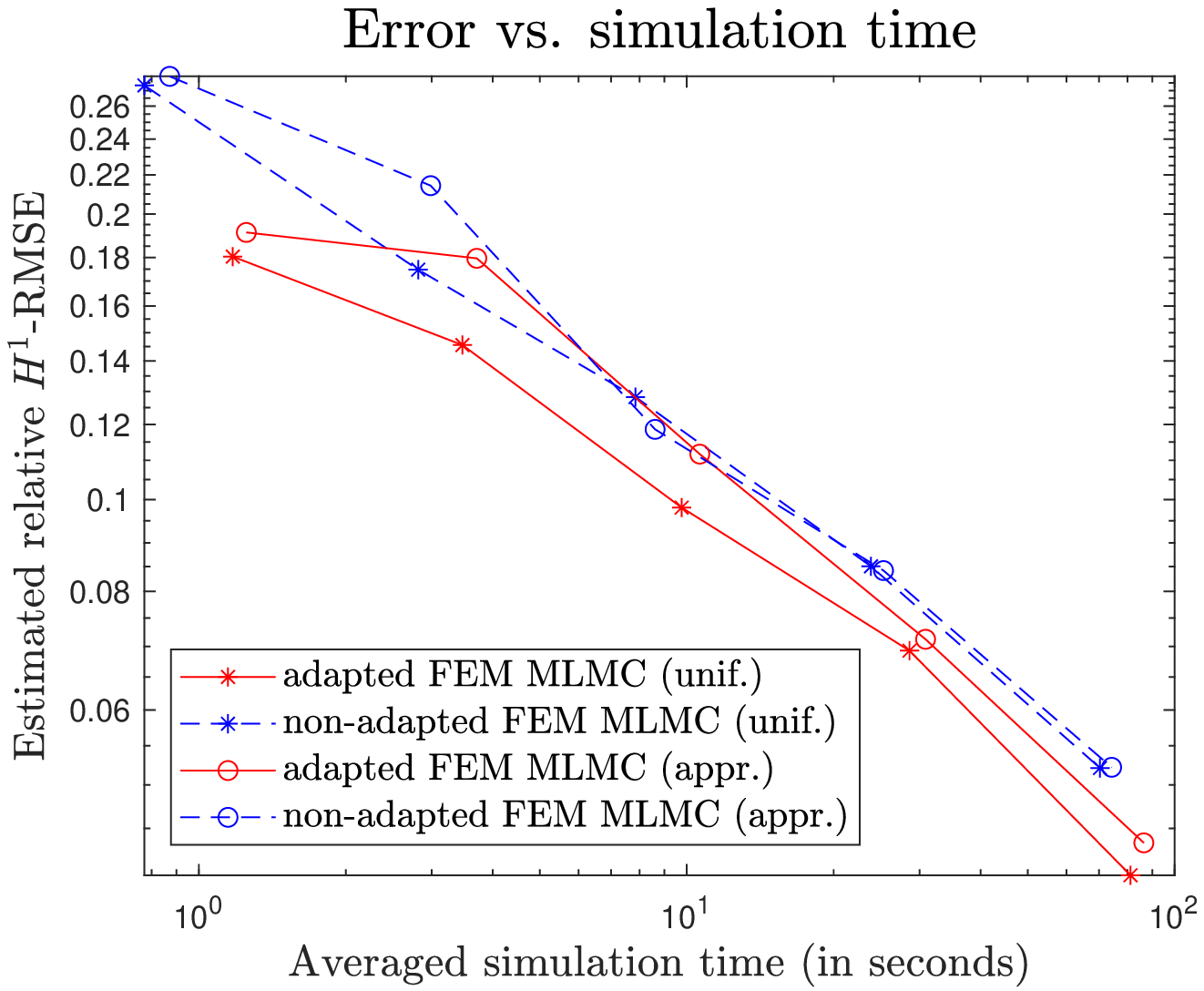}}
	\caption{Convergence of the MLMC estimator for Poisson($1$)-subordinators (left) and time-to-error plot (right).}	\label{fig:Poiss1MLMCResults}
	\end{figure}
	
\noindent The left graph of Figure \ref{fig:Poiss1MLMCResults} shows almost full order convergence of the adapted FEM MLMC method and a slightly slower convergence of the non-adapted FEM MLMC approach. Closer inspection of the figure shows that the choice of the simulation method of the subordinator does not affect the convergence rate of the MLMC estimator: where the Uniform Method yields a slightly smaller RMSE compared to the approximation approach in the sample-adapted case, the behaviour is almost the same for both simulation techniques in the non-adapted FEM MLMC method. The right hand side of Figure \ref{fig:Poiss1MLMCResults} demonstrates a slightly improved efficiency of adapted FEM MLMC compared to non-adapted FEM MLMC. The advantage of the sample-adapted approach can be further emphasized by the use of subordinators with a higher jump intensity and different correlation lengths of the underlying GRF, as we see in the following subsections.

\subsubsection{Poisson($5$) subordinators - smooth underlying GRF}\label{subsubsec:Poiss5Sm}
In the second numerical example we increase the jump-intensity of the subordinators and investigate the effect on the performance of the MLMC estimators. We use Poisson($5$)-subordinators leading to an expected number of 5 jumps in each direction in the diffusion coefficient. The standard deviation and the correlation parameter for the GRF $W_1$ (resp. $W_2$) are set to be $\sigma_1^2= 0.5^2$ and $r_1=0.5$ (resp. $\sigma_2^2=0.3^2$ and $r_2=0.5$). Figure \ref{fig:poisson5sm_samples} shows samples of the diffusion coefficient and the corresponding PDE solution.

	\begin{figure}[ht]
	\centering
	\subfigure{\includegraphics[scale=0.14]{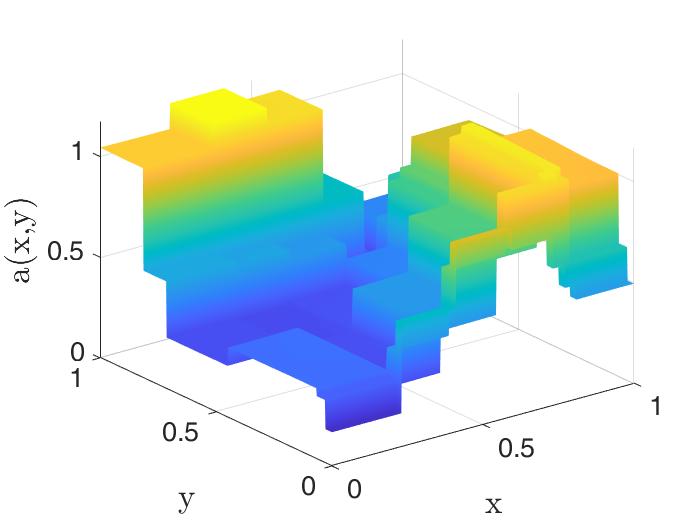}}
	\subfigure{\includegraphics[scale=0.14]{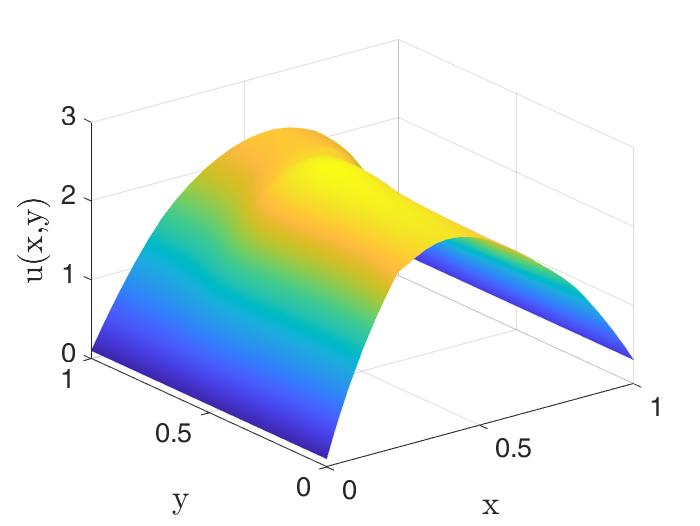}}
	\subfigure{\includegraphics[scale=0.14]{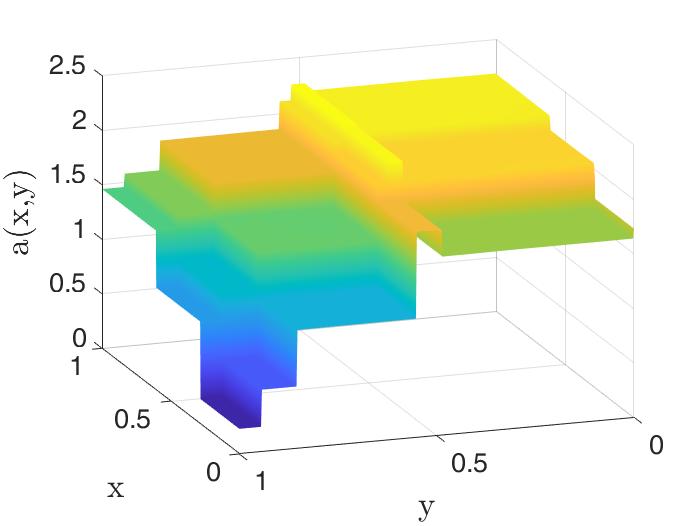}}
	\subfigure{\includegraphics[scale=0.14]{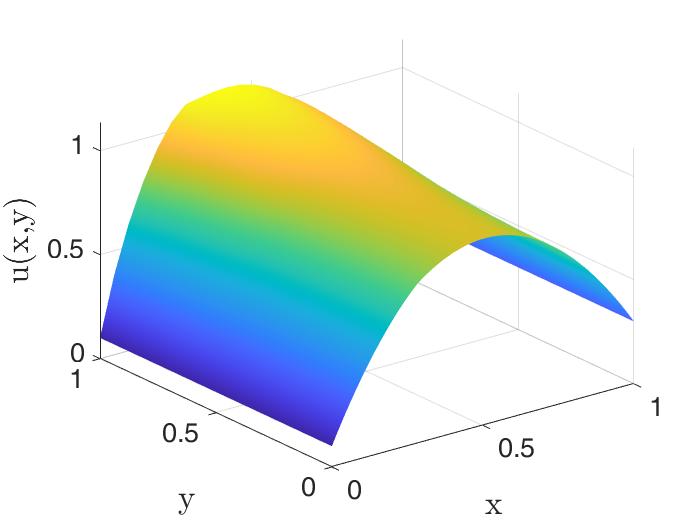}}
    \caption{Different samples of the diffusion coefficient with Poisson($5$)-subordinators and the corresponding PDE solutions with mixed Dirichlet-Neumann boundary conditions.}\label{fig:poisson5sm_samples}
    \end{figure}
    The cut-off threshold $K$ for the subordinators in \eqref{EQ:DiffCoeffDefiCutUpperCut} is choosen to be $K=15$. With this choice we obtain
\begin{align*}
 \mathbb{P}(\underset{t\in[0,1]}{\sup} l_j(t) \geq 15) = \mathbb{P}(l_j(1)\geq 15)\approx 6.9008e^{- 05},
\end{align*}
for $j=1,2$, such that this cut-off has a negligible influence in the numerical example. In order to avoid an expensive simulation of the GRF $W_2$ on the domain $[0,15]^2$ we set $K=1$ instead and consider the downscaled processes
\begin{align*}
\tilde{l}_j(t) = \frac{1}{15}l_j(t),
\end{align*}
for $t\in [0,1]$ and $j=1,2$. Note that this has no effect on the expected number of jumps of the processes. We use the Uniform Method to simulate the Poisson subordinators and estimate the RMSE of the MLMC estimators for the sample-adapted and the non-adapted approach using 10 independent MLMC runs on the levels $L=1,\dots,5$, where we set $\overline{h}_\ell=h_\ell=0.2\cdot 1.7^{-(\ell-1)}$ for $\ell = 1,\dots,5$. Further, we use a reference solution computed on level $7$ with singlelevel Monte Carlo.
		\begin{figure}[ht]
	\centering
	\subfigure{\includegraphics[scale=0.49]{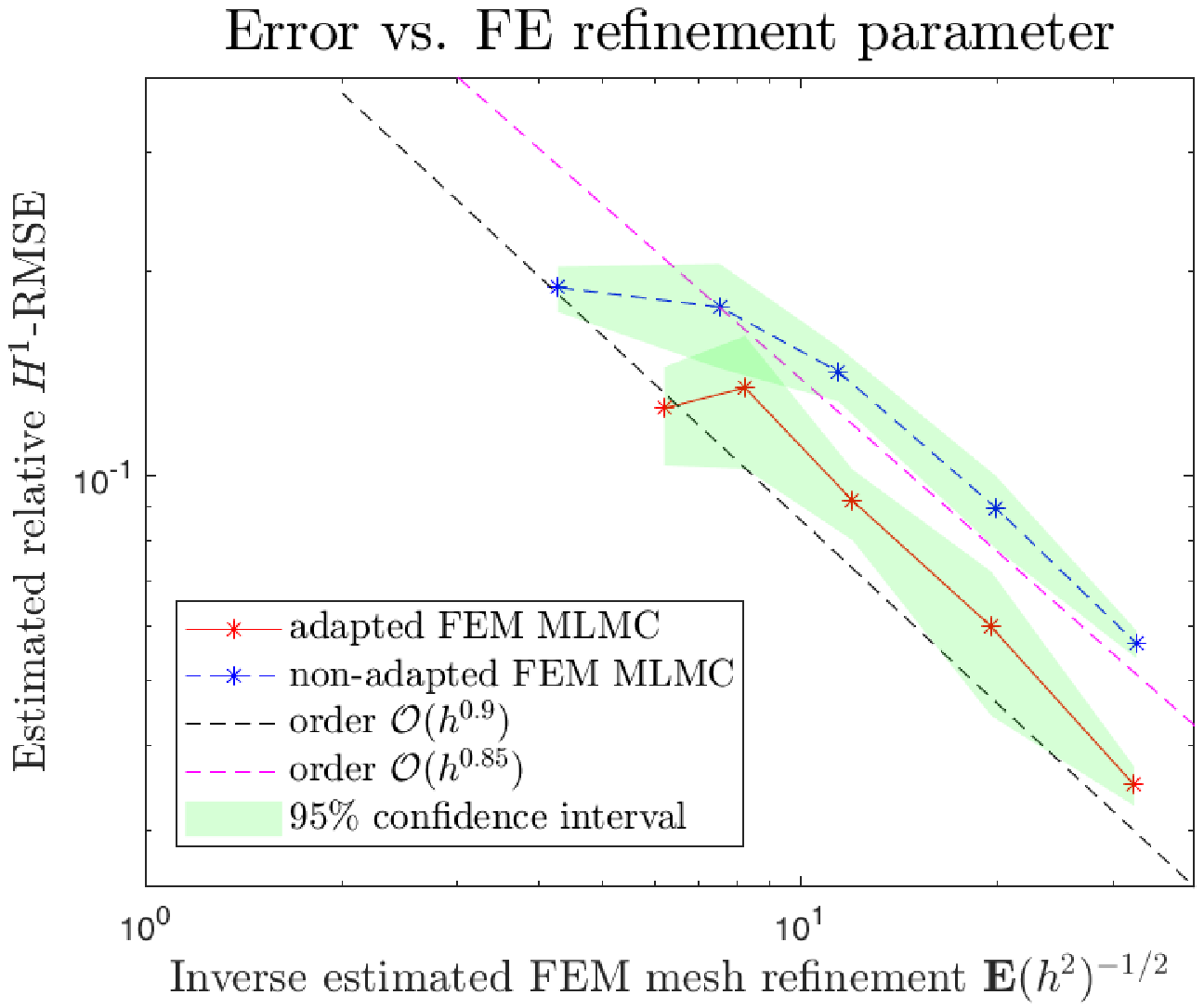}}
	\subfigure{\includegraphics[scale=0.49]{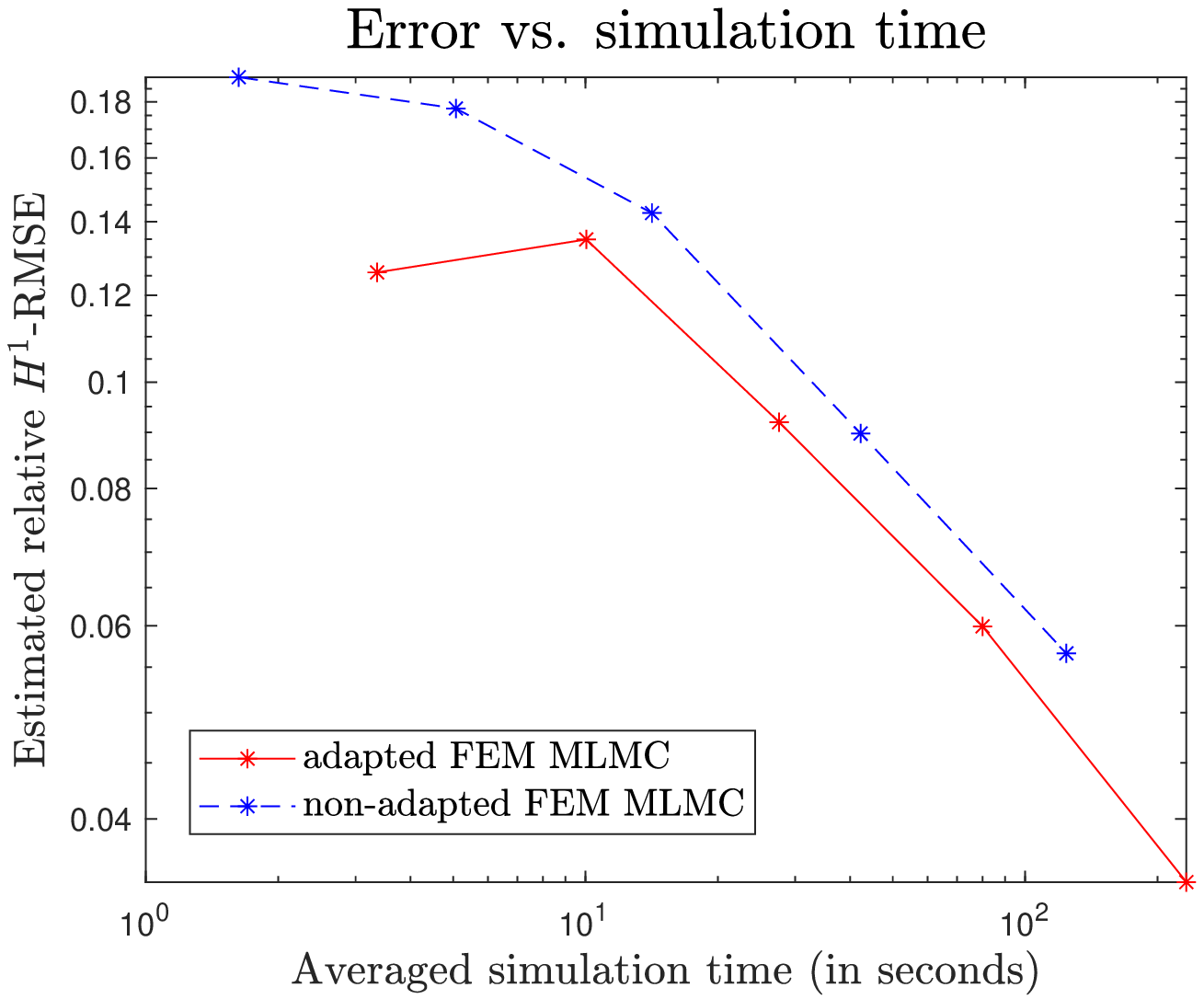}}
	\caption{Convergence of the MLMC estimator for Poisson($5$)-subordinators (left) and time-to-error plot (right).}	\label{fig:Poiss5SmoothMLMCResults}
	\end{figure}
	
Figure \ref{fig:Poiss5SmoothMLMCResults} shows almost full order convergence of the adapted FEM MLMC method and a slightly slower convergence for the non-adapted FEM MLMC approach. The right hand side of Figure \ref{fig:Poiss5SmoothMLMCResults} demonstrates a higher efficiency of the sample-adapted approach. However, one has to mention that differences in the performance of the estimators are rather small due to the comparatively high convergence rate for the non-adapted MLMC approach of approximately $0.85$. This is due to the fact that the jumps in the diffusion coefficient are comparatively small on account of the high correlation length of the underlying GRF $W_2$. We will see in the following subsection that a higher intensity of the jump heights has a significant negative influence on the performance of the non-adapted FEM MLMC approach.
	
\subsubsection{Poisson($5$) subordinators - rough underlying GRF}\label{subsubsec:P5SubordRough}

In the jump diffusion coefficient (see \eqref{EQ:DiffCoeffDefiCutUpperCut}), the jumps are generated by the subordinated GRF in the following way: the number of spatial jumps is determined by the subordinators and the jump intensities (measured in the differences in diffusion values across a jump) are essentially determined by the GRF $W_2$ and its correlation length. Hence, we may control the jump intensities of the diffusion coefficient by the correlation parameter of the underlying GRF $W_2$. In the following experiment we investigate the influence of the jump intensities of the diffusion coefficient on the convergence rates of the MLMC estimators.

\noindent In Subsection \ref{subsubsec:Poiss5Sm} we subordinated a Mat\'ern-1.5-GRF with correlation length $r_2=0.5$ by Poisson($5$)-processes. In the following experiment we set the correlation length of the GRF $W_2$ to $r_2=0.1$ and leave all the other parameters unchanged. Figure~\ref{fig:Grf_samples} presents samples of  the resulting GRFs with the different correlation lengths.

	\begin{figure}[ht]
	\centering
	\subfigure{\includegraphics[scale=0.14]{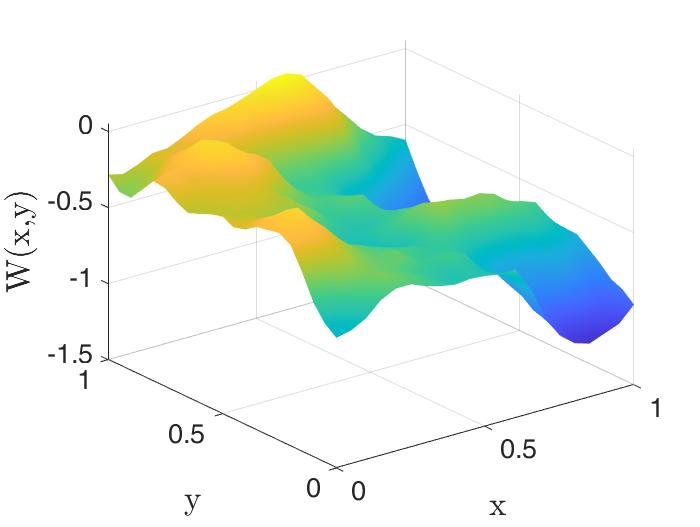}}
	\subfigure{\includegraphics[scale=0.14]{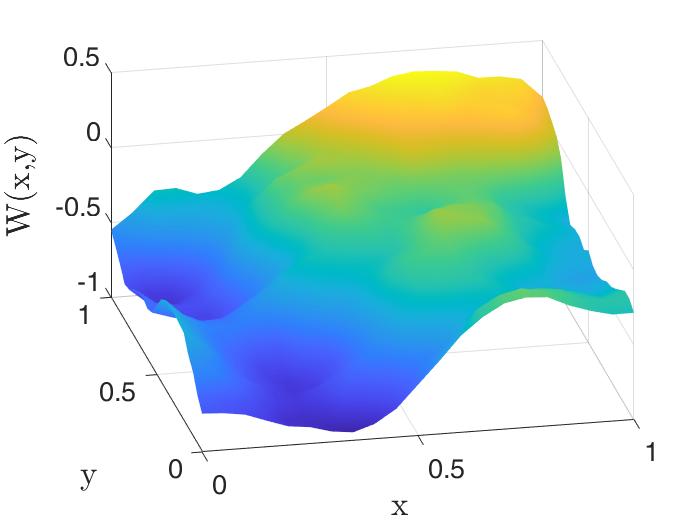}}
	\subfigure{\includegraphics[scale=0.14]{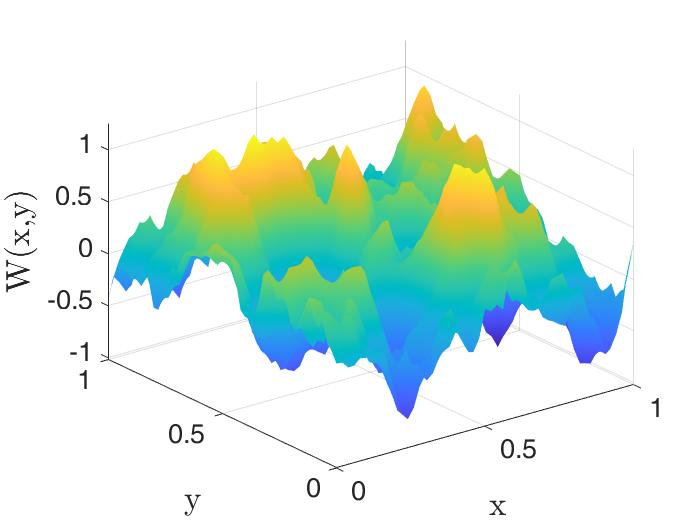}}
	\subfigure{\includegraphics[scale=0.14]{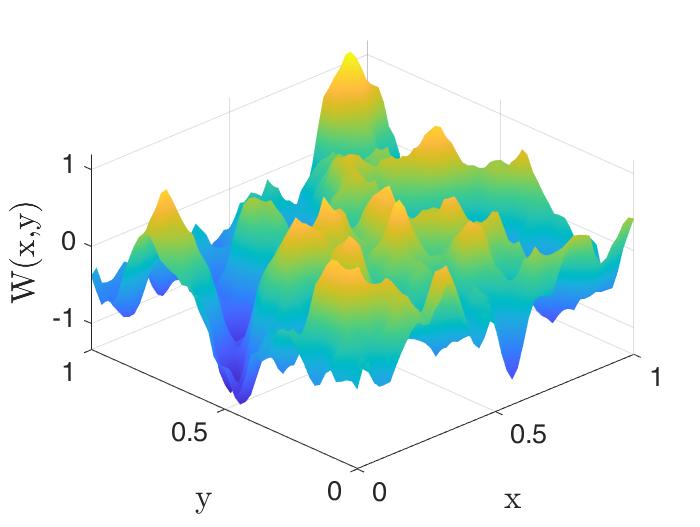}}
    \caption{Different samples of Matérn-1.5-GRFs with correlation lengths $r=0.5$ (left) and $r=0.1$ (right).}\label{fig:Grf_samples}
    \end{figure}

\noindent By construction of the diffusion coefficient, the subordination of GRFs with small correlation lengths (right plots in Figure \ref{fig:Grf_samples}) results in higher jump intensities in the diffusion coefficient as the subordination of GRFs with higher correlation lengths (left plots in Figure \ref{fig:Grf_samples}). This relationship is demonstrated in Figure \ref{fig:poisson5r_samples} (cf. Figure \ref{fig:poisson5sm_samples}).

	\begin{figure}[ht]
	\centering
	\subfigure{\includegraphics[scale=0.14]{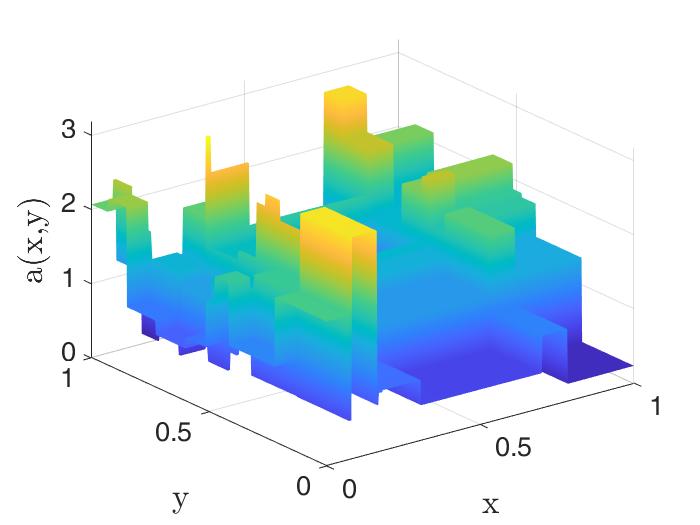}}
	\subfigure{\includegraphics[scale=0.14]{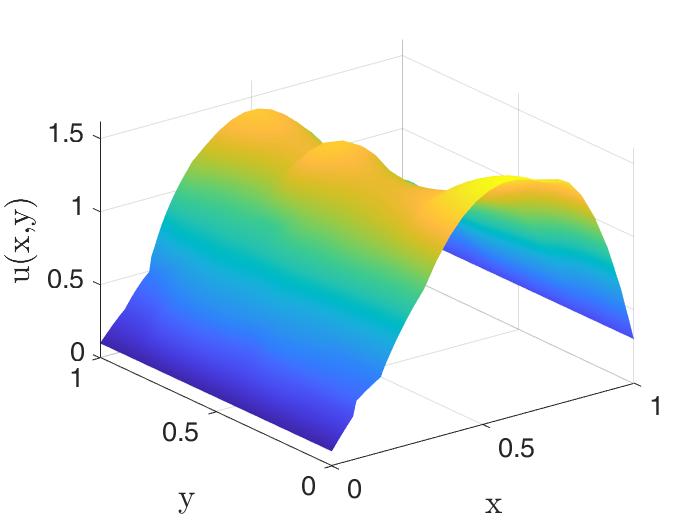}}
	\subfigure{\includegraphics[scale=0.14]{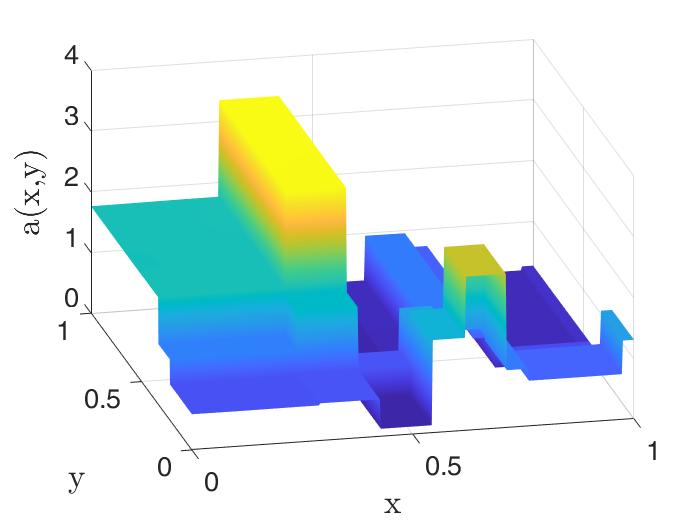}}
	\subfigure{\includegraphics[scale=0.14]{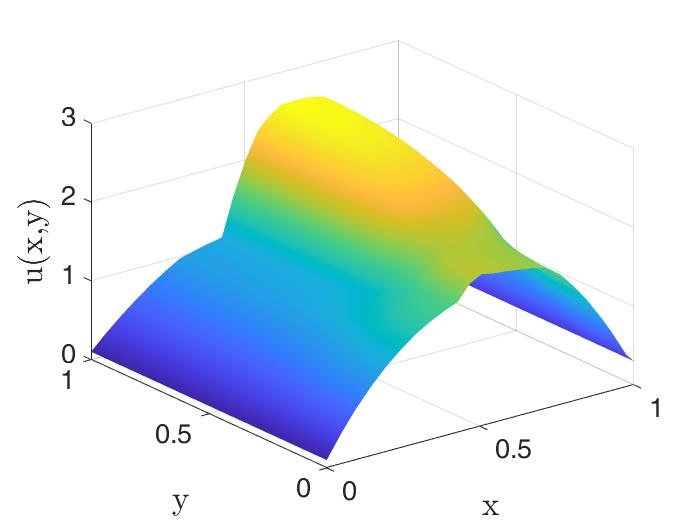}}
    \caption{Different samples of the diffusion coefficient with Poisson($5$)-subordinators and small correlation length in the underlying GRF and the corresponding PDE solutions with mixed Dirichlet-Neumann boundary conditions.}\label{fig:poisson5r_samples}
    \end{figure}
\noindent We use the Uniform Method to compute the RMSE of the MLMC estimators for the sample-adapted and the non-adapted approach using 10 independent MLMC runs on the levels $L=1,\dots,5$, where we set $\overline{h}_\ell=h_\ell=0.2\cdot 1.7^{-(\ell-1)}$ for $\ell = 1,\dots,5$. Further, we use a reference solution computed on level $7$ with singlelevel Monte Carlo. 
	\begin{figure}[ht]
	\centering
	\subfigure{\includegraphics[scale=0.49]{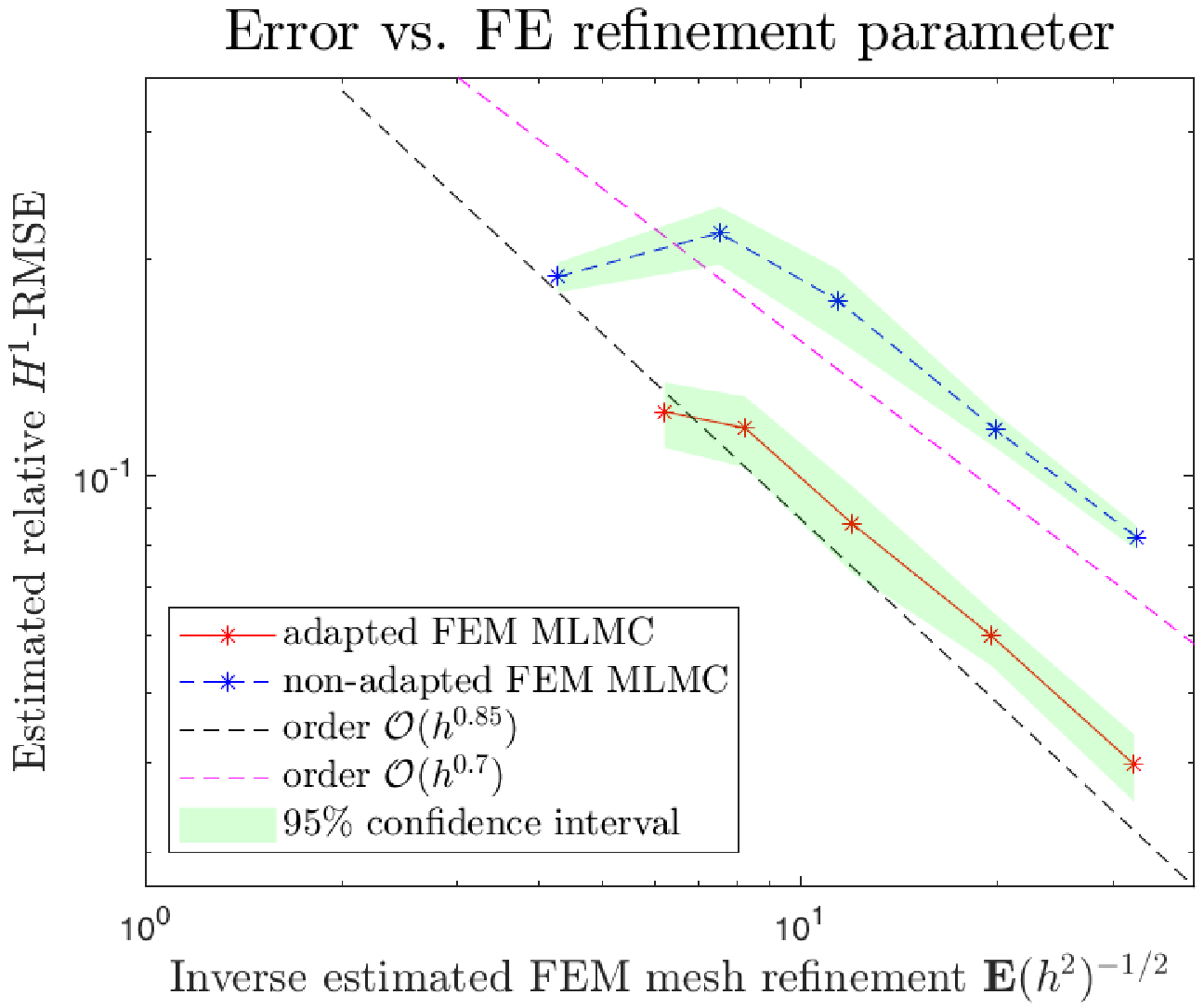}}
	\subfigure{\includegraphics[scale=0.49]{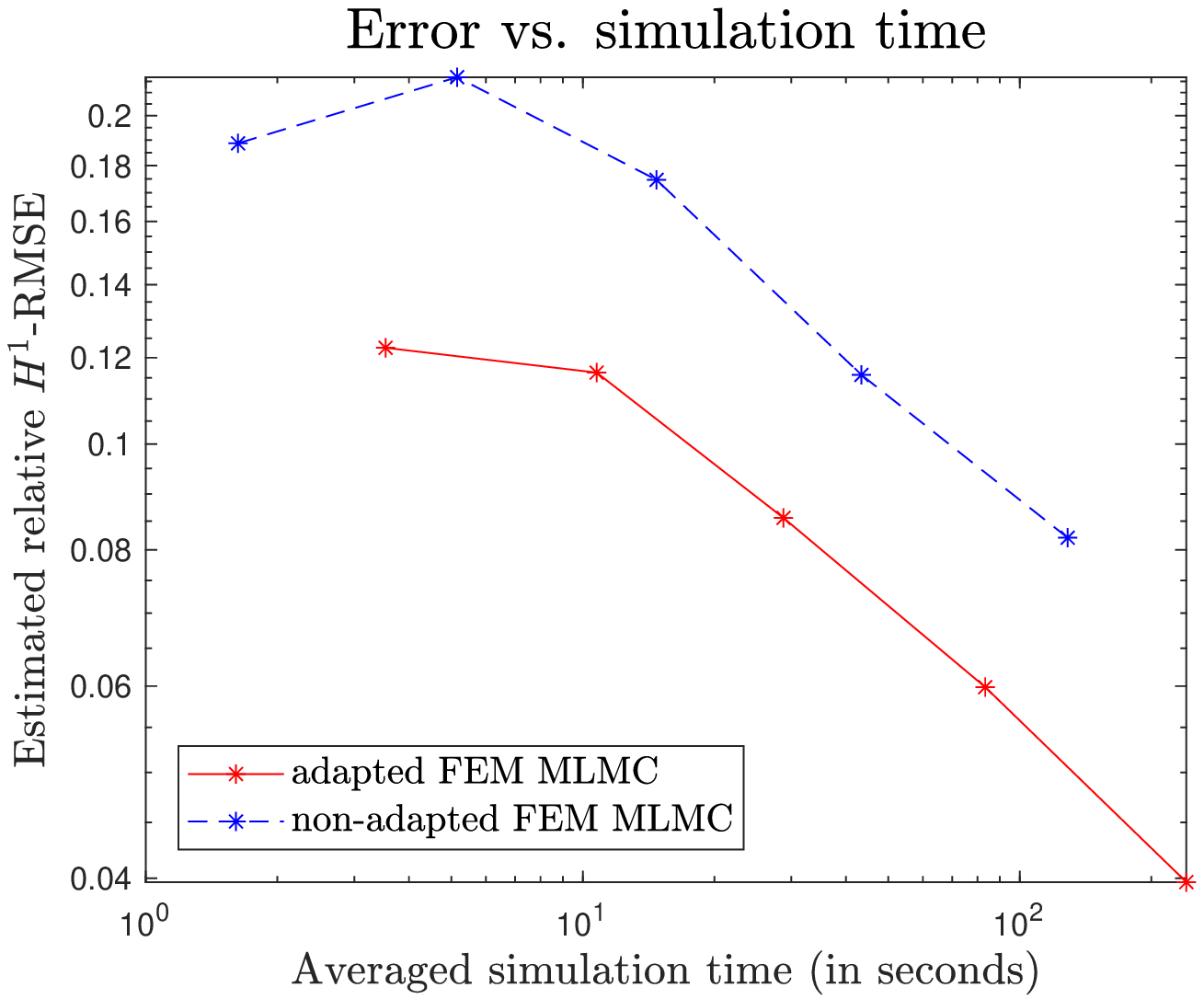}}
	\caption{Convergence of the MLMC estimator for Poisson($5$)-subordinators and small correlation length in the underlying GRF (left) and time-to-error plot (right).}	\label{fig:Poiss5roughMLMCResults}
	\end{figure}
Figure \ref{fig:Poiss5roughMLMCResults} reveals that the higher jump intensities in the diffusion coefficient have a negative impact on the convergence rates of both estimators: the adapted and the non-adapted FEM MLMC approach. We obtain a convergence rate of approximately $0.85$ for the adapted FEM MLMC estimator and a smaller rate of approximately $0.7$ for the MLMC estimator with non-adapted triangulations. Compared to the experiment discussed in Subsection \ref{subsubsec:Poiss5Sm}, where we used Poisson($5$)-subordinators and a higher correlation length in the underlying GRF, we observe that both convergence rates are smaller in the current example. This matches our expectations since the FEM convergence rate has been shown to be influenced by the regularity of the jump diffusion coefficient (see e.g. \cite{AStudyOfElliptic} and  \cite{RegularityResultsForLaplaceInterfaceProblemsInTroDimensions}). It is also important to mention that the RMSE is significantly smaller for the adapted FEM MLMC estimator due to the higher jump intensities in this example.  The higher efficiency of the sample-adapted approach is also demonstrated in the time-to-error plot on the right hand side of \ref{fig:Poiss5roughMLMCResults}: In this example we see a significant improvement in the time-to-error plot for the adapted FEM MLMC approach compared to the non-adapted FEM MLMC estimator.

\subsection{Numerical examples for the MLMC-CV-estimator}\label{subsec:NumExMLCV}

In the following section, we present numerical examples for the MLMC-CV  estimator introduced in Section \ref{sec:MLMCCV}. In Subsection \ref{subsec:NumExMLMC} we considered Poisson subordinators and compared the non-adapted FEM MLMC estimator with the sample-adapted approach and saw that the latter leads to an improved performance of the estimator. However, this approach is computationally not feasible anymore if we consider subordinators with infinite activity, like Gamma subordinators. The aim of this section is to compare the (non-adapted FEM) MLMC estimator with the MLMC-CV estimator for diffusion coefficients with Gamma-subordinated GRFs.
\subsubsection{Gamma subordinators}
 We approximate the Gamma processes in the same way as we approximate the Poisson subordinators in the approximation approach (see Subsection \ref{subsubsec:TheTwoApprMethods}) and obtain a valid approximation in the sense of Assumption \ref{ASS:CutProblemEigenvalues} \textit{v} for any $\eta<+ \infty$ (see \cite[Section 7.4]{SGRFPDE} and \cite{EunfuehrungInDieTheorieDerGammafunktion}). Since we aim to compare the performance of the MLMC estimator with the MLMC-CV estimator we use optimal sample numbers in the numerical experiments in this subsection: Assume level dependent FE discretization sizes $h_\ell$ are given, for $\ell=1,\dots,L$ with $L\in\mathbb{N}$. Further, we denote by $VAR_\ell$ the (estimated) variances of $u_{\varepsilon_{W,\ell},\varepsilon_{l,\ell},\ell} - u_{\varepsilon_{W,\ell-1},\varepsilon_{l,\ell-1},\ell-1}$ (resp. $u_{\varepsilon_{W,\ell},\varepsilon_{l,\ell},\ell}^{CV} - u_{\varepsilon_{W,\ell-1},\varepsilon_{l,\ell-1},\ell-1}^{CV}$ for the MLMC-CV estimator). The optimal sample numbers are then given by the formula
\begin{align*}
M_\ell = h_L^{-2}\sqrt{VAR_\ell}h_\ell \sum_{i=1}^L \sqrt{VAR_i}h_i^{-1},\text{ for } \ell=1,\dots,L,
\end{align*}
since this choice minimizes the variance of the MLMC(-CV) estimator for fixed computational costs (see \cite[Section 1.3]{GilesMLMCMethods}).
In our numerical experiments we choose $l_1 $ and $l_2$ to be Gamma($4,10$) processes. 
We set the diffusion cut-off to $K=2$ to obtain 
\begin{align*}
 \mathbb{P}(\underset{t\in[0,1]}{\sup} l_j(t) \geq K) = \mathbb{P}(l_j(1)\geq 2)\approx 3.2042e^{-06},
\end{align*}
for $j=1,2$. Hence, the influence of the subordinator cut-off is again negligible in our numerical experiments. Due to the high jump intensity we have to choose a sufficiently small smoothness parameter $\nu_s$ since otherwise important detailed information of the diffusion coefficient might be unused. In our two numerical examples, we choose $\nu_s=0.01$ which is small enough for Gamma($4,10$)-subordinators. The expectation of the mean of the control variate $\mathbb{E}(u_{K,A}^{(\nu_s)})$ is estimated by a non-adapted FEM MLMC estimator on level $L$ (see Remark \ref{rem:ApprCVMean}). The standard deviation of the $W_1$ is set to be $\sigma_1^2=1.5^2$ and the correlation length is defined by $r_1=0.5$. The parameters of the GRF $W_2$ are varied in our numerical experiments.

\subsubsection{MLMC-CV vs. MLMC for infinite activity subordinators}
In this numerical example we choose $\sigma_2^2=0.3^2$ and $r_2=0.05$ for the GRF $W_2$. Figure \ref{fig:GammaR_samples} shows samples of the diffusion coefficient and the corresponding PDE solutions.
	\begin{figure}[ht]
	\centering
	\subfigure{\includegraphics[scale=0.14]{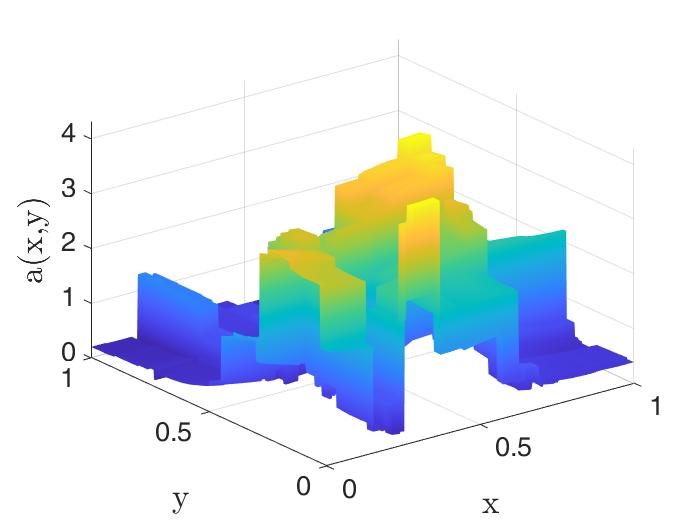}}
	\subfigure{\includegraphics[scale=0.14]{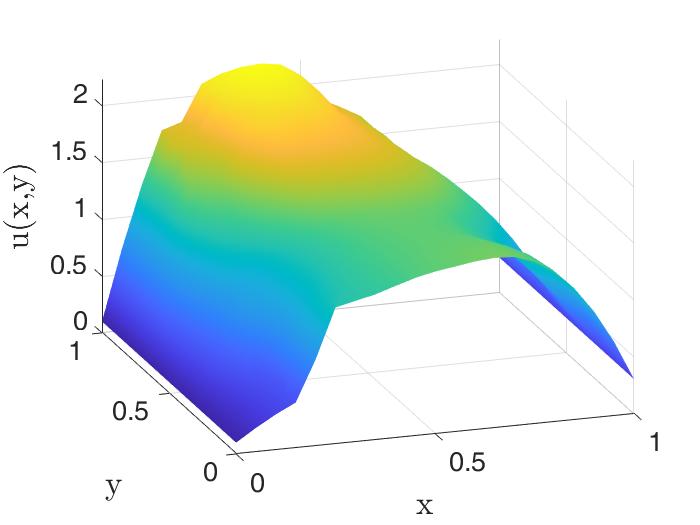}}
	\subfigure{\includegraphics[scale=0.14]{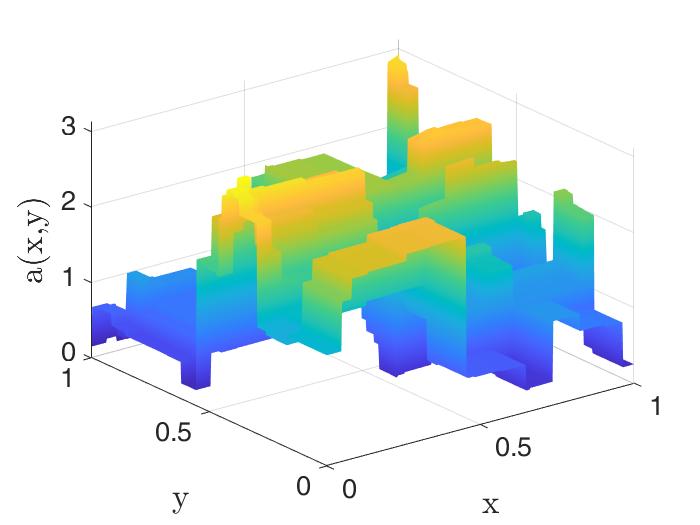}}
	\subfigure{\includegraphics[scale=0.14]{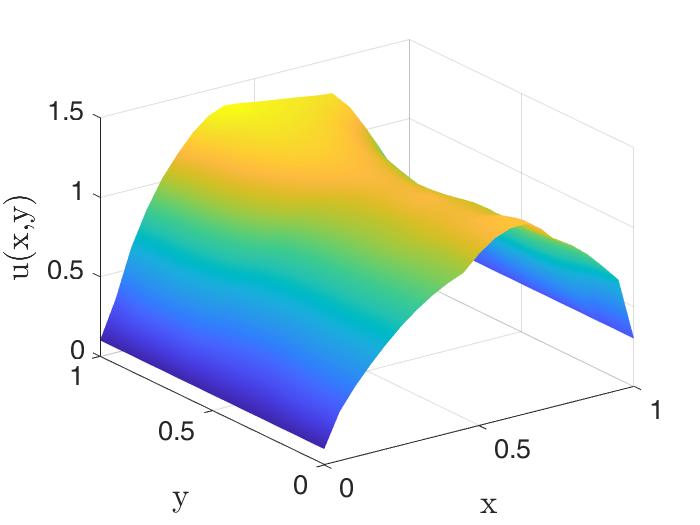}}
    \caption{Different samples of the diffusion coefficient with Gamma($4,10$)-subordinators with small correlation length in the underlying GRF together with the corresponding PDE solutions with mixed Dirichlet-Neumann boundary conditions.}\label{fig:GammaR_samples}
    \end{figure}
    We define the level dependent FE discretization parameters $h_\ell = 0.3\cdot 1.7^{-(\ell-1)}$ for $\ell=1,\dots,5$ and compare the MLMC estimator with the MLMC-CV estimator. We perform 10 independent MLMC runs on the levels $L=1,\dots,5$ to estimate the RMSE where we use a reference solution on level 7 computed by singlelevel Monte Carlo. The results are given in the following figure.
		\begin{figure}[ht]
	\centering
	\subfigure{\includegraphics[scale=0.49]{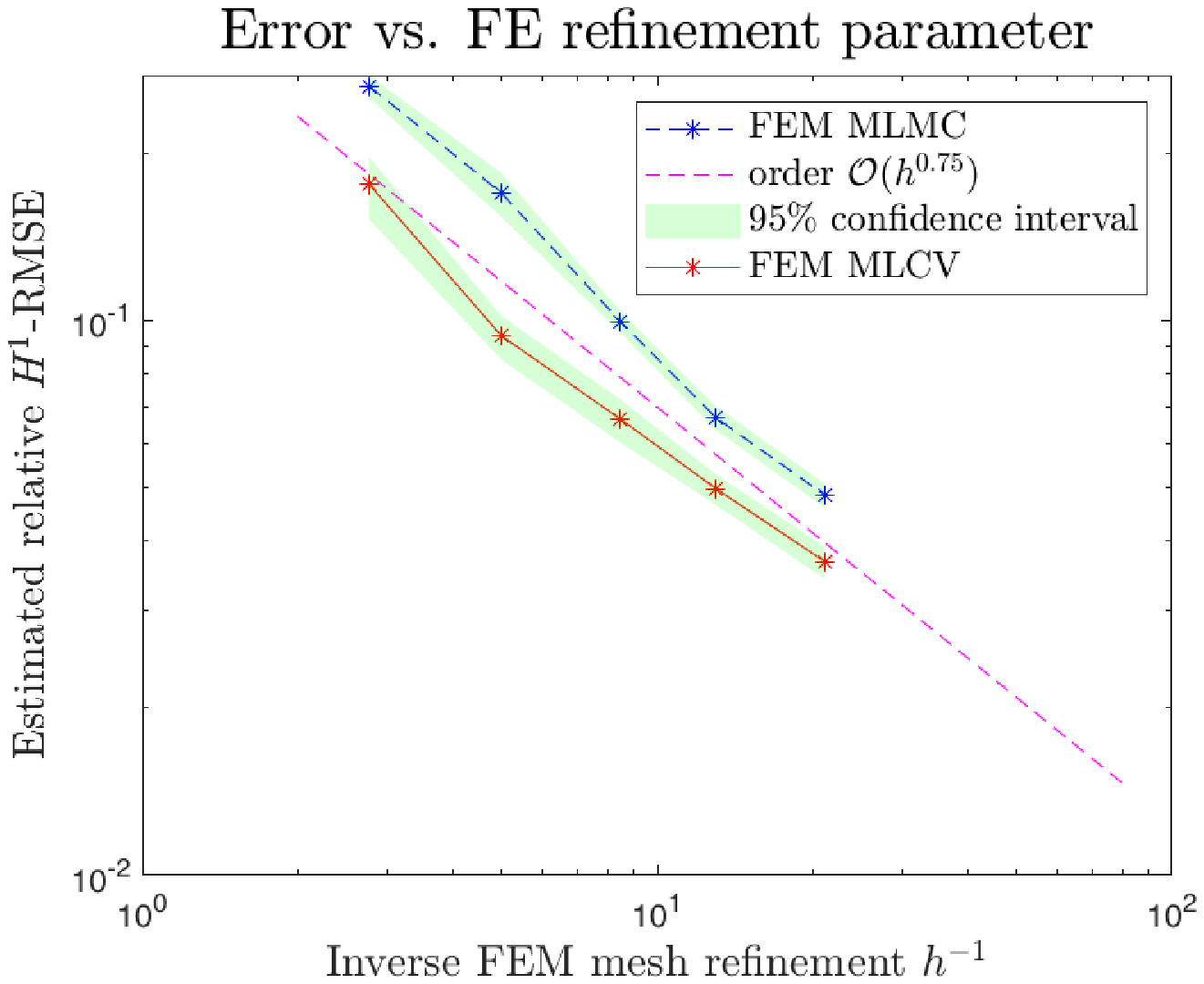}}
	\subfigure{\includegraphics[scale=0.49]{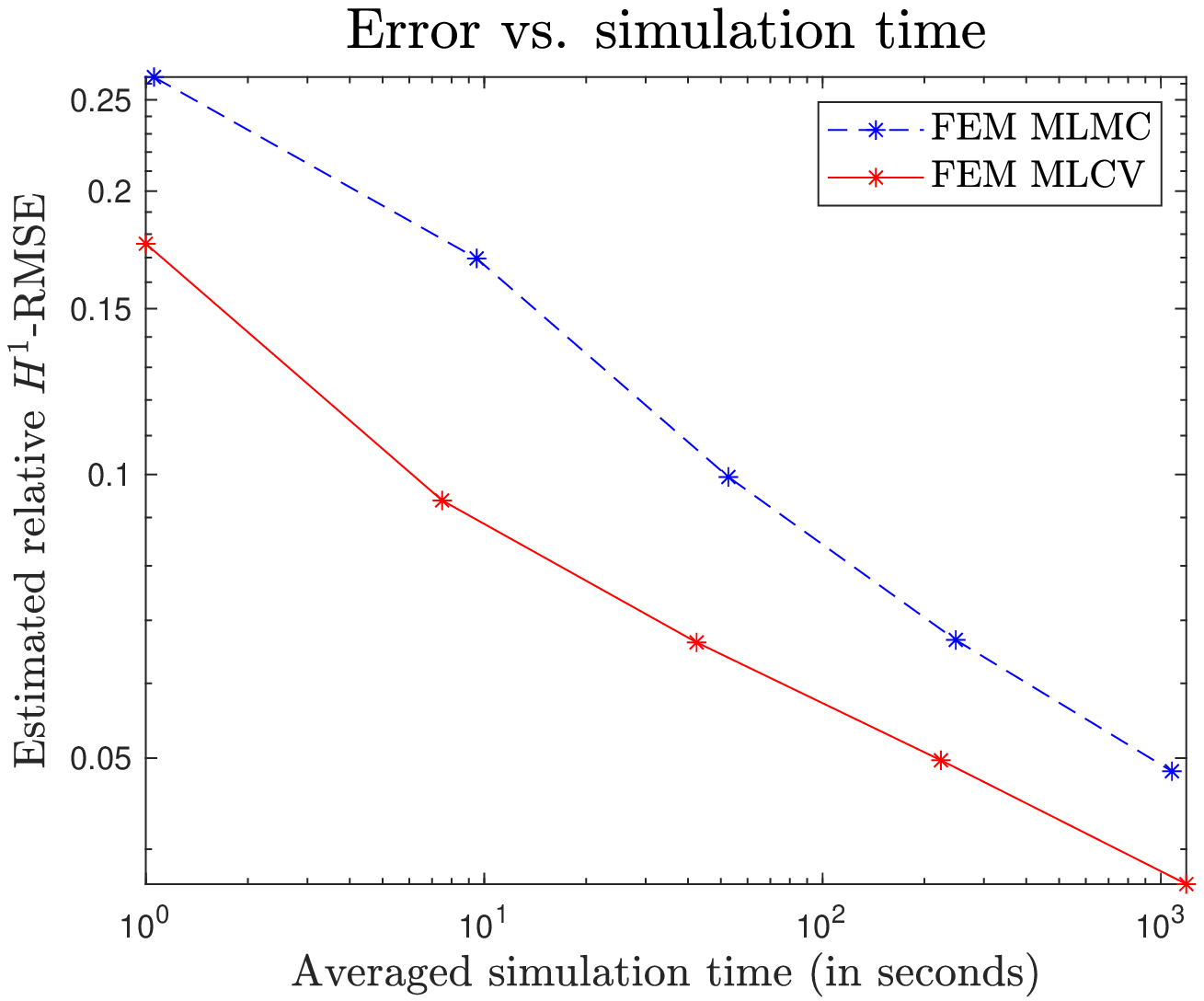}}
	\caption{Convergence of the MLMC and the MLMC-CV estimator for Gamma($4,10$)-subordinators with small noise and small correlation length in the underlying GRF (left) and time-to-error plot (right).}	\label{fig:MLMCvsMLCVGammaEX1}
	\end{figure}
Figure \ref{fig:MLMCvsMLCVGammaEX1} shows a similar convergence rate of approximately $0.75$ for the MLMC and the MLMC-CV estimator. However, the sample-wise correction by the smooth PDE problem in the MLMC-CV estimator improves the approximation which yields significantly smaller values for the RMSE on the different levels compared to the standard MLMC estimator. The efficiency improvement obtained by the Control Variate is further demonstrated on the right hand side of Figure \ref{fig:MLMCvsMLCVGammaEX1}: The time-to-error plot demonstrates that the computational effort which is necessary to achieve a certain accuracy is significantly smaller for the MLMC-CV estimator compared to the standard MLMC estimator.

\noindent  In our last numerical example we choose $\Phi_1=1/5\,exp(\cdot)$, $\Phi_2=3|\cdot|$ and $\sigma_2^2=0.5^2$, $r_2=0.2$ for the GRF $W_2$ and leave all other parameters unchanged. This leads to diffusion coefficient which is more noise accentuated with reduced jump intensity (see also Subsection \ref{subsubsec:P5SubordRough}) as can be seen in Figure \ref{fig:GammaS_samples}.

	\begin{figure}[ht]
	\centering
	\subfigure{\includegraphics[scale=0.14]{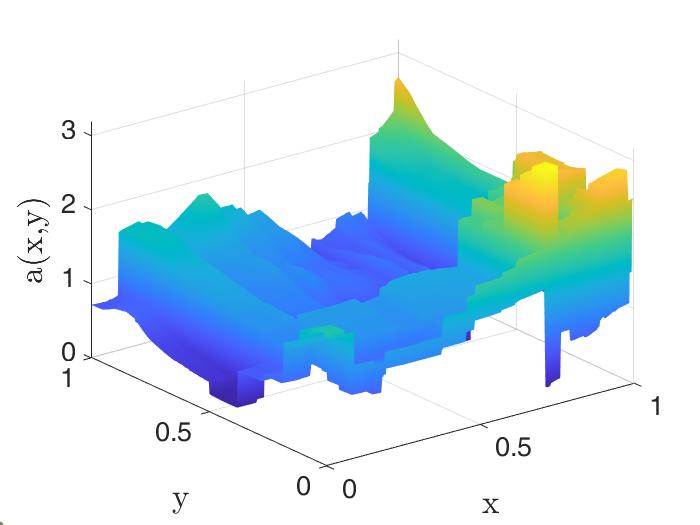}}
	\subfigure{\includegraphics[scale=0.14]{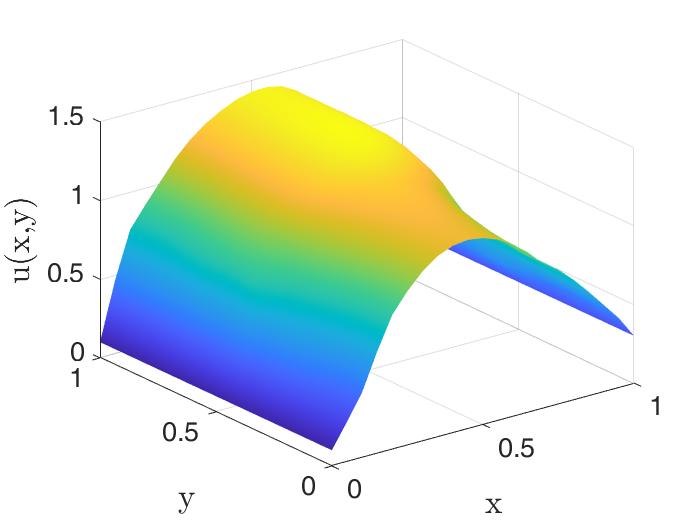}}
	\subfigure{\includegraphics[scale=0.14]{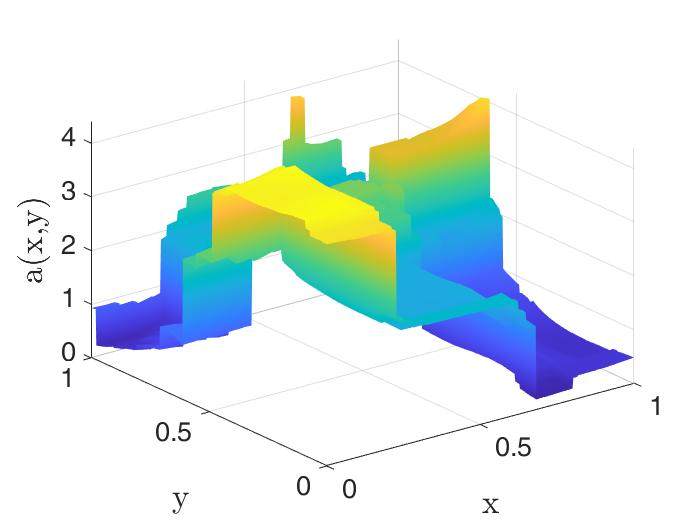}}
	\subfigure{\includegraphics[scale=0.14]{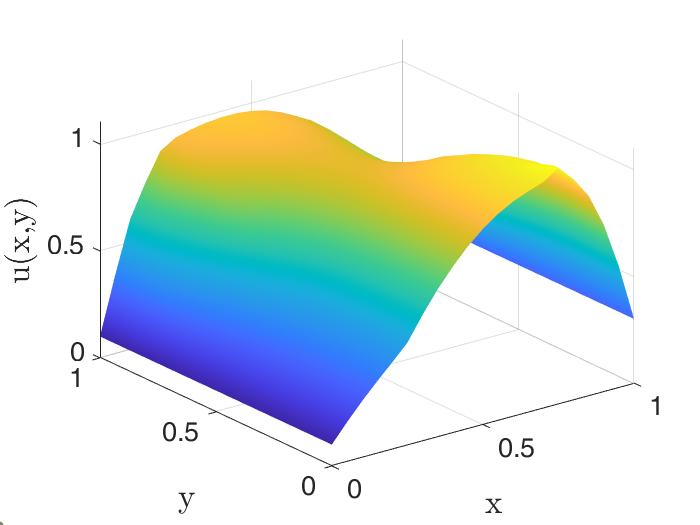}}
    \caption{Different samples of the diffusion coefficient with Gamma($4,10$)-subordinators with strong noise and higher correlation length in the underlying GRF together with the corresponding PDE solutions with mixed Dirichlet-Neumann boundary conditions.}\label{fig:GammaS_samples}
    \end{figure}
     As in the last experiment, we define the level dependent FE discretization parameters $h_\ell = 0.3\cdot 1.7^{-(\ell-1)}$ for $\ell=1,\dots,5$ and compare the MLMC estimator with the MLMC-CV estimator. We use 10 independent MLMC runs on the levels $L=1,\dots,5$ to estimate the RMSE and use a singlelevel Monte Carlo estimation on level 7 as reference solution. The results are given in Figure \ref{fig:MLMCvsMLCVGammaEX2}.
		\begin{figure}[ht]
	\centering
	\subfigure{\includegraphics[scale=0.49]{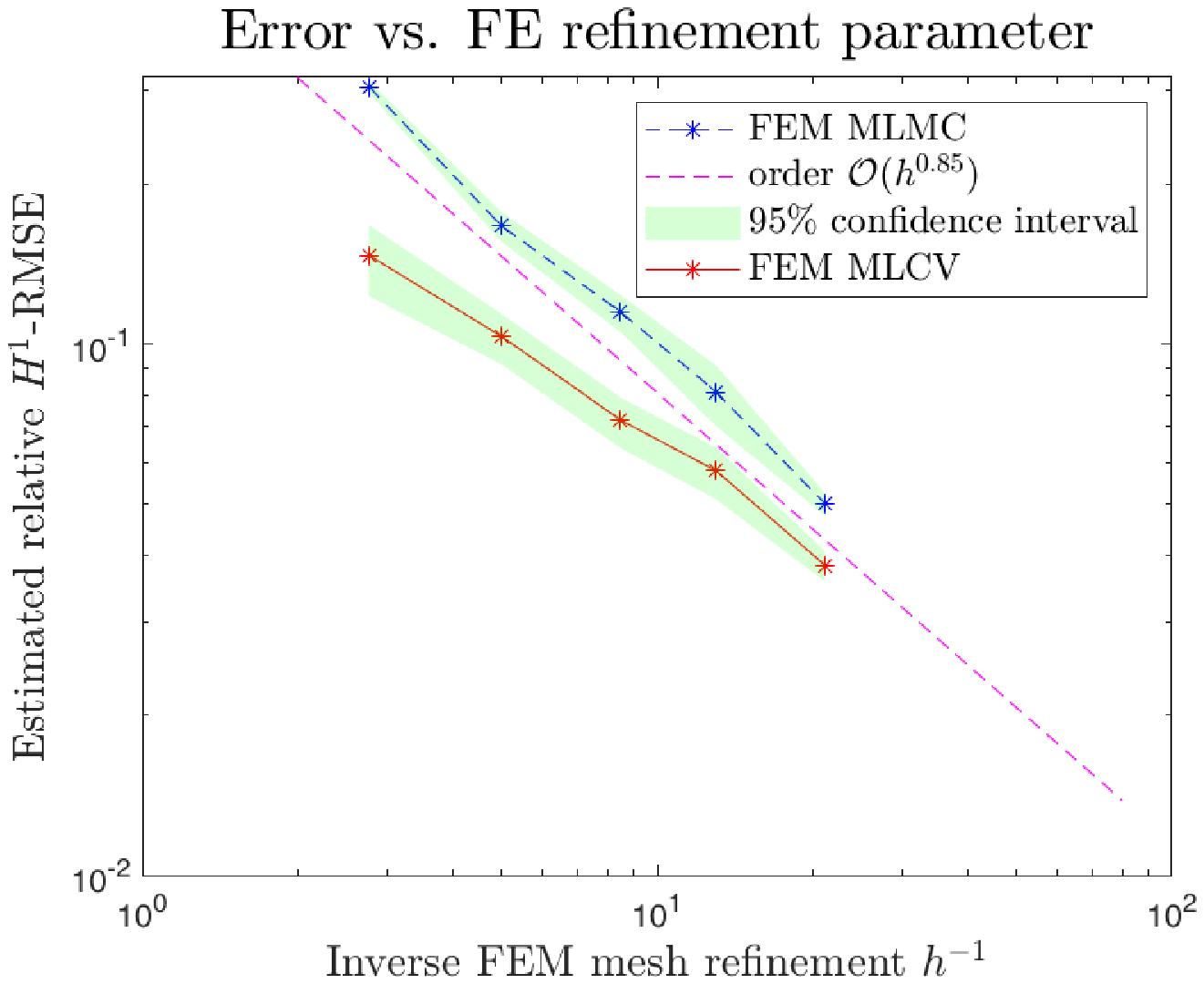}}
	\subfigure{\includegraphics[scale=0.49]{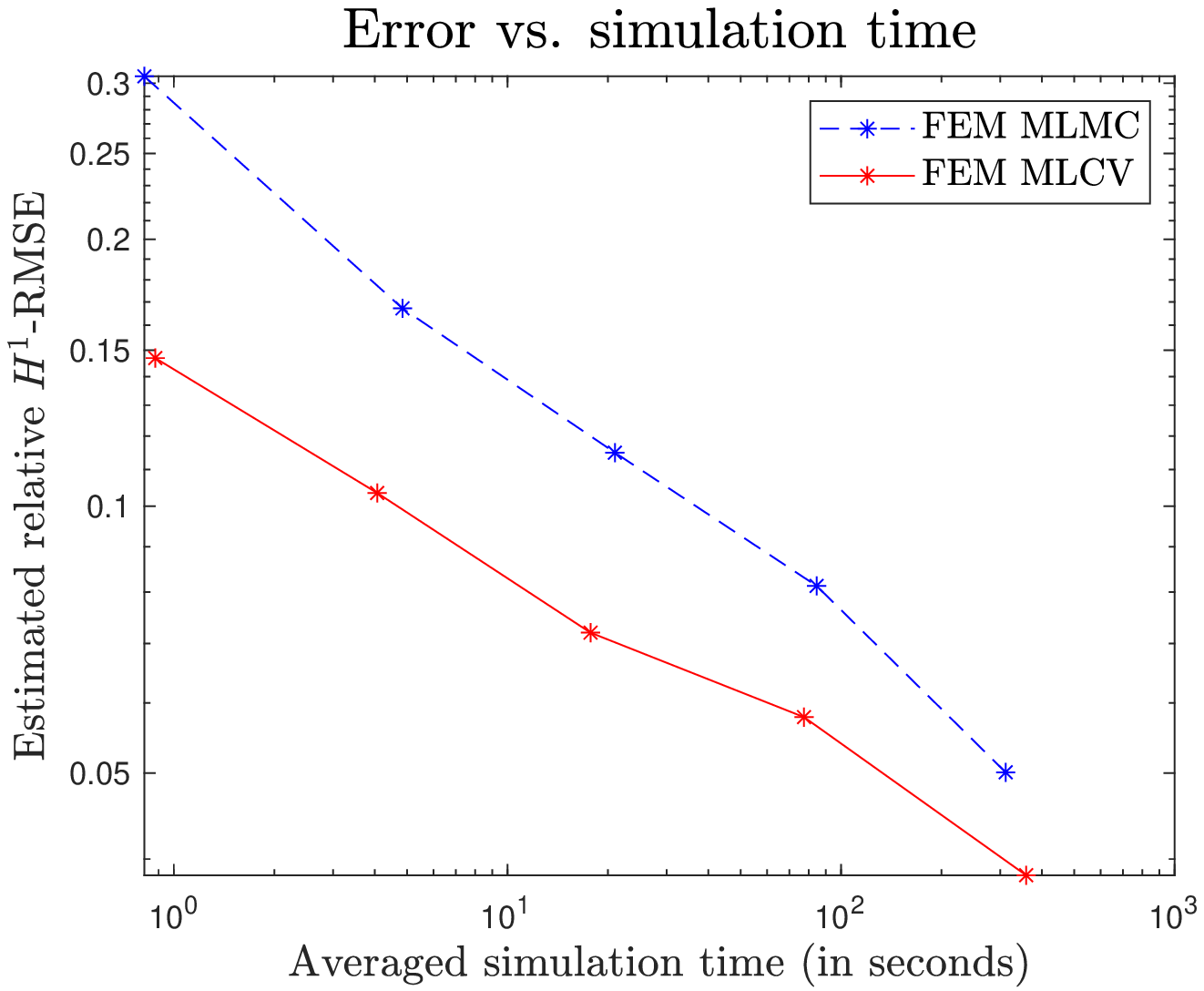}}
	\caption{Convergence of the MLMC and the MLMC-CV estimator for Gamma($4,10$)-subordinators with strong noise and underlying GRF with higher correlation length  (left) and time-to-error plot (right).}	\label{fig:MLMCvsMLCVGammaEX2}
	\end{figure}

The reduced jump intensity together with the emphasized (continuous) noise in the diffusion coefficient leads to a slightly improved convergence rate of approximately $0.85$ for the estimators in this example (cf. Figure \ref{fig:MLMCvsMLCVGammaEX1}). As in the first experiment, we see that the usage of the Control Variate yields a significant improvement which is reflected in smaller values for the RMSE on the different levels compared to the standard MLMC approach. As expected, the right hand side of Figure \ref{fig:MLMCvsMLCVGammaEX2} shows an improved efficiency of the MLMC-CV estimator compared to the standard MLMC estimator without Control Variates.


	\section*{Acknowledgments}
		Funded by Deutsche Forschungsgemeinschaft (DFG, German Research Foundation) under Germany's Excellence Strategy - EXC 2075 -     390740016.


 	\bibliographystyle{siam}
	\bibliography{references}

\begin{thebibliography}{10}

\bibitem{MLMCForStochEllMultiscalePDEs}
{\sc A.~Abdulle, A.~Barth, and C.~Schwab}, {\em Multilevel {M}onte {C}arlo
  methods for stochastic elliptic multiscale {PDE}s}, Multiscale Model. Simul.,
  11 (2013), pp.~1033--1070.

\bibitem{RandomFieldsAndGeometry}
{\sc R.~J. Adler and J.~E. Taylor}, {\em Random fields and geometry}, Springer
  Monographs in Mathematics, Springer, New York, 2007.

\bibitem{LevyProcessesAndStochasticCalculus}
{\sc D.~Applebaum}, {\em L\'{e}vy processes and stochastic calculus}, vol.~116
  of Cambridge Studies in Advanced Mathematics, Cambridge University Press,
  Cambridge, second~ed., 2009.

\bibitem{EunfuehrungInDieTheorieDerGammafunktion}
{\sc E.~Artin}, {\em Einf{\"u}hrung in die Theorie der Gamma-funktion},
  Hamburger mathematische Einzelschriften, B.G. Teubner, 1931.

\bibitem{TheFiniteElementMethodForELlipticEquationsWithDiscontinuousCoefficients}
{\sc I.~Babu\v{s}ka}, {\em The finite element method for elliptic equations
  with discontinuous coefficients}, Computing (Arch. Elektron. Rechnen), 5
  (1970), pp.~207--213.

\bibitem{GalerkiNFEApprOfStochEllPDEs}
{\sc I.~Babu\v{s}ka, R.~Tempone, and G.~E. Zouraris}, {\em Galerkin finite
  element approximations of stochastic elliptic partial differential
  equations}, SIAM J. Numer. Anal., 42 (2004), pp.~800--825.

\bibitem{SGRF}
{\sc A.~Barth and R.~Merkle}, {\em Subordinated gaussian random fields}, ArXiv
  e-prints, arXiv:2012.06353 [math.PR],  (2020).

\bibitem{SGRFPDE}
\leavevmode\vrule height 2pt depth -1.6pt width 23pt, {\em Subordinated
  gaussian random fields in elliptic partial differential equations}, ArXiv
  e-prints, arXiv:2011.09311 [math.NA],  (2020).

\bibitem{MLMCFEMEllPDE}
{\sc A.~Barth, C.~Schwab, and N.~Zollinger}, {\em Multi-level {M}onte {C}arlo
  finite element method for elliptic {PDE}s with stochastic coefficients},
  Numer. Math., 119 (2011), pp.~123--161.

\bibitem{ApproximationAndSimulation}
{\sc A.~Barth and A.~Stein}, {\em Approximation and simulation of
  infinite-dimensional {L}\'{e}vy processes}, Stoch. Partial Differ. Equ. Anal.
  Comput., 6 (2018), pp.~286--334.

\bibitem{AStudyOfElliptic}
\leavevmode\vrule height 2pt depth -1.6pt width 23pt, {\em A study of elliptic
  partial differential equations with jump diffusion coefficients}, SIAM/ASA J.
  Uncertain. Quantif., 6 (2018), pp.~1707--1743.

\bibitem{AMultilevelMonteCarloAlgorithmForParabolicAdvectionDiffusionProblemsWithDiscontinuousCoefficients}
\leavevmode\vrule height 2pt depth -1.6pt width 23pt, {\em A multilevel monte
  carlo algorithm for~parabolic advection-diffusion problems with discontinuous
  coefficients}, in Springer Proceedings in Mathematics {\&} Statistics,
  Springer International Publishing, 2020, pp.~445--466.

\bibitem{FiniteElementErrorAnalysisOfEllipticPDEsWIthRandomCoefficients}
{\sc J.~Charrier, R.~Scheichl, and A.~L. Teckentrup}, {\em Finite element error
  analysis of elliptic {PDE}s with random coefficients and its application to
  multilevel {M}onte {C}arlo methods}, SIAM J. Numer. Anal., 51 (2013),
  pp.~322--352.

\bibitem{FEMAndTheirConvergenceForEllipticAndParabolicProblems}
{\sc Z.~Chen and J.~Zou}, {\em Finite element methods and their convergence for
  elliptic and parabolic interface problems}, Numer. Math., 79 (1998),
  pp.~175--202.

\bibitem{StochasticEquationsInInfiniteDimensions}
{\sc G.~Da~Prato and J.~Zabczyk}, {\em Stochastic equations in infinite
  dimensions}, vol.~152 of Encyclopedia of Mathematics and its Applications,
  Cambridge University Press, Cambridge, second~ed., 2014.

\bibitem{TraceTheoremLipDomain}
{\sc Z.~Ding}, {\em A proof of the trace theorem of {S}obolev spaces on
  {L}ipschitz domains}, Proc. Amer. Math. Soc., 124 (1996), pp.~591--600.

\bibitem{Gottschalk2021}
{\sc O.~G. Ernst, H.~Gottschalk, T.~Kalmes, T.~Kowalewitz, and M.~Reese}, {\em
  Integrability and approximability of solutions to the stationary diffusion
  equation with l\'evy coefficient}, ArXiv e-prints, arXiv:2010.14912v2
  [math.AP],  (2021).

\bibitem{PartialDifferentialEquations}
{\sc L.~C. Evans}, {\em Partial differential equations}, vol.~19 of Graduate
  Studies in Mathematics, American Mathematical Society, Providence, RI,
  second~ed., 2010.

\bibitem{FEsForEllProblemsWithStochCoeff}
{\sc P.~Frauenfelder, C.~Schwab, and R.~A. Todor}, {\em Finite elements for
  elliptic problems with stochastic coefficients}, Comput. Methods Appl. Mech.
  Engrg., 194 (2005), pp.~205--228.

\bibitem{MultilevelMonteCarloPathSimulation}
{\sc M.~B. Giles}, {\em Multilevel {M}onte {C}arlo path simulation}, Oper.
  Res., 56 (2008), pp.~607--617.

\bibitem{GilesMLMCMethods}
\leavevmode\vrule height 2pt depth -1.6pt width 23pt, {\em Multilevel {M}onte
  {C}arlo methods}, Acta Numer., 24 (2015), pp.~259--328.

\bibitem{GlassermanMCMethods}
{\sc P.~Glasserman}, {\em Monte {C}arlo methods in financial engineering},
  vol.~53 of Applications of Mathematics (New York), Springer-Verlag, New York,
  2004.
\newblock Stochastic Modelling and Applied Probability.

\bibitem{QuasiMonteCarloFEMethodsForEllipticPDEsWithLognormalRandomCoefficients}
{\sc I.~G. Graham, F.~Y. Kuo, J.~A. Nichols, R.~Scheichl, C.~Schwab, and I.~H.
  Sloan}, {\em Quasi-{M}onte {C}arlo finite element methods for elliptic {PDE}s
  with lognormal random coefficients}, Numer. Math., 131 (2015), pp.~329--368.

\bibitem{AnalysisOfCirculantEmbeddingMethodsForSamplingStationaryRandomFields}
{\sc I.~G. Graham, F.~Y. Kuo, D.~Nuyens, R.~Scheichl, and I.~H. Sloan}, {\em
  Analysis of circulant embedding methods for sampling stationary random
  fields}, SIAM J. Numer. Anal., 56 (2018), pp.~1871--1895.

\bibitem{CirculantEmbeddingWithWMCAnalysisForEllipicPDEWithLognormalCoefficients}
{\sc I.~G. Graham, F.~Y. Kuo, D.~Nuyens, R.~Scheichl, and I.~H. Sloan}, {\em
  Circulant embedding with {QMC}: analysis for elliptic {PDE} with lognormal
  coefficients}, Numer. Math., 140 (2018), pp.~479--511.

\bibitem{EllipticDifferentialEquations}
{\sc W.~Hackbusch}, {\em Elliptic differential equations}, vol.~18 of Springer
  Series in Computational Mathematics, Springer-Verlag, Berlin, second~ed.,
  2017.
\newblock Theory and numerical treatment.

\bibitem{KallenbergFoundationsOfModernProb}
{\sc O.~Kallenberg}, {\em Foundations of modern probability}, vol.~99 of
  Probability Theory and Stochastic Modelling, Springer, Cham, third~ed.,
  [2021] \copyright 2021.

\bibitem{WTheorie}
{\sc A.~Klenke}, {\em Wahrscheinlichkeitstheorie}, Springer Berlin Heidelberg,
  2013.

\bibitem{MultiLevelMonteCarloWeakGalerkinMethodForEllipticEquationsWithStochasticJumpCoefficients}
{\sc J.~Li, X.~Wang, and K.~Zhang}, {\em Multi-level {M}onte {C}arlo weak
  {G}alerkin method for elliptic equations with stochastic jump coefficients},
  Appl. Math. Comput., 275 (2016), pp.~181--194.

\bibitem{OnTheConvOfTheStochGalerkinMethodForRandomEllPDEs}
{\sc A.~Mugler and H.-J. Starkloff}, {\em On the convergence of the stochastic
  {G}alerkin method for random elliptic partial differential equations}, ESAIM
  Math. Model. Numer. Anal., 47 (2013), pp.~1237--1263.

\bibitem{NobileTeseiMLCV}
{\sc F.~Nobile and F.~Tesei}, {\em A multi level {M}onte {C}arlo method with
  control variate for elliptic {PDE}s with log-normal coefficients}, Stoch.
  Partial Differ. Equ. Anal. Comput., 3 (2015), pp.~398--444.

\bibitem{PeszatZabzykSPDEWithLevyNoise}
{\sc S.~Peszat and J.~Zabczyk}, {\em Stochastic partial differential equations
  with {L}\'{e}vy noise}, vol.~113 of Encyclopedia of Mathematics and its
  Applications, Cambridge University Press, Cambridge, 2007.
\newblock An evolution equation approach.

\bibitem{RegularityResultsForLaplaceInterfaceProblemsInTroDimensions}
{\sc M.~Petzoldt}, {\em Regularity results for laplace interface problems in
  two dimensions}, Zeitschrift f\"{u}r Analysis und ihre Anwendungen, 20
  (2001), pp.~431--455.

\bibitem{LevyProcessesInFinance}
{\sc W.~Schoutens}, {\em Levy Processes in Finance: Pricing Financial
  Derivatives}, Wiley Series in Probability and Statistics, Wiley, 2003.

\bibitem{FurtherAnalysisOfMultilevelMonteCarloMethodsForEllipticPDEsWithRandomCoefficients}
{\sc A.~L. Teckentrup, R.~Scheichl, M.~B. Giles, and E.~Ullmann}, {\em Further
  analysis of multilevel {M}onte {C}arlo methods for elliptic {PDE}s with
  random coefficients}, Numer. Math., 125 (2013), pp.~569--600.

\bibitem{ValliACompactCourseOnLinPDEs}
{\sc A.~Valli}, {\em A compact course on linear {PDE}s}, vol.~126 of Unitext,
  Springer, Cham, [2020] \copyright 2020.
\newblock La Matematica per il 3+2.

\end{thebibliography}

\end{document}